\RequirePackage{fix-cm}
\RequirePackage{amsmath}
\RequirePackage{amssymb}
\documentclass{article}
\usepackage{mathrsfs}
\usepackage{bm}
\usepackage{amsthm}
\usepackage[colorlinks,linkcolor=blue,anchorcolor=red,citecolor=blue]{hyperref}
 \usepackage{geometry}
\newtheorem{Thm}{Theorem}[section]

\newtheorem{Lem}[Thm]{Lemma}

\newtheorem{Prop}[Thm]{Proposition}

\newcommand{\1}{\mathbf{1}}
\newcommand{\R}{\mathbb{R}}
\newcommand{\Rd}{{\mathbb{R}^3}}

\renewcommand{\P}{\mathbf{P}}
\renewcommand{\Re}{\text{Re}}

\newcommand{\F}{\mathscr{F}}
\newcommand{\N}{\mathbb{N}}
\newcommand{\C}{\mathbb{C}}
\newcommand{\E}{\mathcal{E}}
\newcommand{\D}{\mathcal{D}}

\usepackage{cite}

\allowdisplaybreaks
\renewcommand{\Re}{\text{Re}}
\renewcommand{\S}{\mathbb{S}}
\newcommand{\<}{\langle}
\renewcommand{\>}{\rangle}
\newcommand{\T}{\mathcal{T}}
\newcommand{\I}{\mathbf{I}}
\newcommand{\II}{\mathbf{I}_{\pm}}
\renewcommand{\P}{\mathbf{P}}
\newcommand{\PP}{\mathbf{P}_{\pm}}

\newcommand{\vertiii}[1]{{\left\vert\kern-0.25ex\left\vert\kern-0.25ex\left\vert #1 \right\vert\kern-0.25ex\right\vert\kern-0.25ex\right\vert}}

\title{Regularity of the Vlasov-Poisson-Boltzmann System without angular cutoff}
\author{Dingqun DENG \thanks{email: dingqdeng2-c@my.cityu.edu.hk, Department of Mathematics, City University of Hong Kong, ORCID: 0000-0001-9678-314X } }

\usepackage{fancyhdr}
\begin{document}

\maketitle

\begin{abstract}
In this paper we study the regularity of the non-cutoff Vlasov-Poisson-Boltzmann system for plasma particles of two species in the whole space $\mathbb{R}^3$ with hard potential. The existence of global-in-time nearby Maxwellian solutions is known for soft potential from \cite{Duan2013}. However the smoothing effect of these solutions has been a challenging open problem. We establish the global existence and regularizing effect to the Cauchy problem for hard potential with large time decay. Hence, the solutions are smooth with respect to $(t,x,v)$ for any positive time $t>0$. This gives the regularity to Vlasov-Poisson-Boltzmann system, which enjoys a similar smoothing effect as Boltzmann equation. The proof is based on the time-weighted energy method building also upon the pseudo-differential calculus.

\end{abstract}

\tableofcontents

\section{Introduction}
The Vlasov-Poisson-Boltzmann system is an important physical model to describe the time evolution of plasma particles of two species (e.g. ions and electrons). 
This work contains two main results. The first one is the global-in-time existence of two species Vlasov-Poisson-Boltzmann system for non-cutoff hard potential, which provide a global energy control with optimal large time decay. Secondly, under this global-in-time energy control, the regularizing effect of Vlasov-Poisson-Boltzmann system is discovered at any positive time. Such smoothing effect is a long existing open problem since \cite{Duan2013}, where Duan and Liu successfully found the global solution for non-cutoff soft potential with $1/2\le s<1$. Moreover, the smoothing method in this paper should be applicable to other kind of kinetic system with the transport term and high-order dissipation term.

\paragraph{Model and Equation.} We consider the Vlasov-Poisson-Boltzmann system of two species in the whole space $\R^3$, cf. \cite{Krall1973}:
\begin{equation}\begin{aligned}\label{1}
	\partial_tF_+ + v\cdot\nabla_xF_+ +E\cdot\nabla_vF_+ = Q(F_+,F_+) + Q(F_-,F_+),\\
	\partial_tF_- + v\cdot\nabla_xF_- -E\cdot\nabla_vF_- = Q(F_-,F_-) + Q(F_+,F_-).
\end{aligned}
\end{equation}
The self-consistent electrostatic field is taken as $E(t,x) = -\nabla_x\phi$, with the electric potential $\phi$ given by 
\begin{align}\label{2}
	-\Delta_x\phi = \int_{\Rd}(F_+-F_-)\,dv, \quad  \phi\to 0 \text{ as } |x|\to\infty.
\end{align}
The initial data of the system is 
\begin{align}\label{3}
	F_\pm(0,x,v) = F_{\pm,0}(x,v). 
\end{align}
The unknown function $F_\pm(t,x,v)\ge 0$ represents the velocity distribution for the particle with position $x\in \Rd$ and velocity $v\in\Rd$ at time $t\ge 0$. The bilinear collision term $Q(F,G)$ on the right hand side of \eqref{1} is given by 
\begin{align}
	Q(F,G)(v) = \int_{\Rd}\int_{\S^2}B(v-v_*,\sigma)\big(F'_*G' - F_*G\big)\,d\sigma dv_*, 
\end{align} 
where $F' = F(x,v',t)$, $G'_* = G(x,v'_*,t)$, $F = F(x,v,t)$, $G_* =G(x,v_*,t)$. The more rigorous definition in the form of Carleman representation can be found in \cite{Deng2020}. $(v,v_*)$ is the velocity before the collision and $(v',v'_*)$ is the velocity after the collision. They are defined by 
\begin{align*}
	v' = \frac{v+v_*}{2}+\frac{|v-v_*|}{2}\sigma,\ \
	v'_* = \frac{v+v_*}{2}-\frac{|v-v_*|}{2}\sigma.
\end{align*}
This two pair of velocities satisfy the conservation law of momentum and energy:
\begin{align*}
	v+v_*=v'+v'_*,\ \ |v|^2+|v_*|^2=|v'|^2+|v'_*|^2.
\end{align*}

\paragraph{Collision Kernel.}
The Boltzmann collision kernel $B$ is defined as 
\begin{align*}
	B(v-v_*,\sigma) = |v-v_*|^\gamma b(\cos\theta),
\end{align*}
for some function $b$ and $\gamma>-3$ determined by the intermolecular interactive mechanism with $\cos\theta=\frac{v-v_*}{|v-v_*|}\cdot \sigma$. Without loss of generality, we can assume $B(v-v_*,\sigma)$ is supported on $(v-v_*)\cdot\sigma\ge 0$, which corresponds to $\theta\in(0,\pi/2]$, since $B$ can be replaced by its symmetrized form $\overline{B}(v-v_*,\sigma) = B(v-v_*,\sigma)+B(v-v_*,-\sigma)$ in $Q(f,f)$.
The angular function $\sigma\mapsto b(\cos\theta)$ is not integrable on $\S^2$. Moreover, there exists $0<s<1$ such that 
\begin{align*}
	\sin\theta b(\cos\theta)\approx \theta^{-1-2s}\ \text{ on }\theta \in (0,\pi/2],
\end{align*}
It's convenient to call soft potential when $\gamma+2s< 0$, and hard potential when $\gamma+2s\ge0$. In this work, we always assume
\begin{align}
	s\in (0,1), \quad \gamma\in (-3,\infty) \text{ and } \gamma+2s\ge 0.
\end{align}

In this paper, we are going to establish the global existence as well as the smoothing effect of the solutions to Cauchy problem \eqref{1}-\eqref{3} of the Vlasov-Poisson-Boltzmann system near the global Maxwellian equilibrium. For global existence, Guo \cite{Guo2002} firstly investigate hard-sphere model of the Vlasov-Poisson-Boltzmann system in a periodic box. Since then, the energy method was largely developed for Boltzmann equation with the self-consistent electric and magnetic fields; see \cite{Duan2013,Guo2012,Duan2011}.  
For smoothing effect of Boltzmann equation, since the work \cite{Alexandre2000} discover the entropy dissipation property for non-cutoff linearized Boltzmann operator, there's been many discussion in different context. See \cite{Alexandre2011,Global2019,Alexandre2013,Gressman2011a,Mouhot2007} for the dissipation estimate of collision operator, and \cite{Alexandre2011aa,Alexandre2010,Barbaroux2017a,Chen2018,Chen2011,Chen2012,Deng2020b} for smoothing effect of the solution to Boltzmann equation in different aspect. These works show that the Boltzmann operator behaves locally like a fractional operator:
\begin{align*}
	Q(f,g)\sim (-\Delta_v)^sg+\text{lower order terms}.
\end{align*}
More precisely, according to the symbolic calculus developed by \cite{Global2019}, the linearized Boltzmann operator behaves essentially as 
\begin{align*}
	L \sim \<v\>^\gamma(-\Delta_v-|v\wedge\partial_v|^2+|v|^2)^s+\text{lower order terms}.
\end{align*}
However, until now, the smoothing effect of the solutions to Vlasov-Poisson-Boltzmann system remains open and to the best of our knowledge, this is the first paper discussing such smoothing phenomenon.

\paragraph{Reformulation.} We will reformulate the problem near Maxwellian as in \cite{Guo2002}. For this we denote a normalized global Maxwellian $\mu$ by 
\begin{align}
	\mu(v) = (2\pi)^{-3/2}e^{-|v|^2/2}. 
\end{align}
Set $F_\pm(t,x,v)=\mu(v) +\mu^{1/2}f_\pm(t,x,v)$. Denote $f=(f_+,f_-)$ and $f_0=(f_{+,0},f_{-,0})$. Then the Cauchy problem \eqref{1}-\eqref{3} can be reformulated as 
\begin{equation}\label{7}
	\partial_tf_\pm + v\cdot\nabla_xf_\pm \pm \frac{1}{2}\nabla_x\phi\cdot vf_\pm  \mp\nabla_x\phi\cdot\nabla_vf_\pm \pm \nabla_x\phi\cdot v\mu^{1/2} - L_\pm f = \Gamma_{\pm}(f,f),
\end{equation}
\begin{equation}\label{8}
	-\Delta_x \phi = \int_{\Rd}(f_+-f_-)\mu^{1/2}\,dv, \quad \phi\to 0\text{ as }|x|\to\infty,
\end{equation}
with initial data 
\begin{align}\label{9}
	f_\pm(0,x,v) = f_{\pm,0}(x,v). 
\end{align}
The linear operator $L=(L_+,L_-)$ and $\Gamma = (\Gamma_+,\Gamma_-)$ are gives as 
\begin{equation*}
	L_\pm f = \mu^{-1/2}\Big(2Q(\mu,\mu^{1/2}f_\pm) + Q(\mu^{1/2}(f_\pm+f_\mp),\mu)\Big),
\end{equation*}
\begin{equation*}
	\Gamma_\pm(f,g) = \mu^{-1/2}\Big(Q(\mu^{1/2}f_\pm,\mu^{1/2}g_\pm) + Q(\mu^{1/2}f_\mp,\mu^{1/2}g_\pm)\Big).
\end{equation*}
For later use, we introduce the bilinear operator $\T$ by 
\begin{align*}
	\T_\beta(h_1,h_2) = \int_{\Rd}\int_{\S^2}B(v-v_*,\sigma)\partial_\beta(\mu^{1/2}_*)\big(h_1(v'_*)h_2(v')-h_1(v_*)h_2(v)\big)\,d\sigma dv_*, 
\end{align*}
for two scalar functions $h_1,h_2$ and especially $\T=\T_0$. Thus, 
\begin{equation*}
		L_\pm f = 2\T(\mu^{1/2},f_\pm) + \T(f_\pm+f_\mp,\mu^{1/2}),\\
	\end{equation*}
\begin{equation*}
		\Gamma_\pm(f,g) = \T(f_\pm,g_\pm) + \T(f_\mp,g_\pm).
\end{equation*}

\paragraph{Notations.}
Through the paper, $C$ denotes some positive constant (generally large) and $\lambda$ denotes some positive constant (generally small), where both $C$ and $\lambda$ may take different values in different lines. $(\cdot|\cdot)$ is the inner product in $\C^n$. 
For any $v\in\Rd$, we denote $\<v\>=(1+|v|^2)^{1/2}$. For multi-indices $\alpha=(\alpha_1,\alpha_2,\alpha_3)$ and $\beta=(\beta_1,\beta_2,\beta_3)$, write 
\begin{align*}
	\partial^\alpha_\beta = \partial^{\alpha_1}_{x_1}\partial^{\alpha_2}_{x_2}\partial^{\alpha_3}_{x_3}\partial^{\beta_1}_{v_1}\partial^{\beta_2}_{v_2}\partial^{\beta_3}_{v_3}.
\end{align*}The length of $\alpha$ is $|\alpha|=\alpha_1+\alpha_2+\alpha_3$. 
The notation $a\approx b$ (resp. $a\gtrsim b$, $a\lesssim b$) for positive real function $a$, $b$ means there exists $C>0$ not depending on possible free parameters such that $C^{-1}a\le b\le Ca$ (resp. $a\ge C^{-1}b$, $a\le Cb$) on their domain. $\mathscr{S}$ denotes the Schwartz space. $\Re (a)$ means the real part of complex number $a$. $[a,b]=ab-ba$ is the commutator between operators. $\{a(v,\eta),b(v,\eta)\} =  \partial_\eta a_1\partial_va_2 - \partial_va_1\partial_\eta a_2$ is the Poisson bracket. $\Gamma=|dv|^2+|d\eta|^2$ is the admissible metric and $S(m)=S(m,\Gamma)$ is the symbol class. 
For pseudo-differential calculus, we write $(x,v)\in \Rd\times\Rd$ to be the space-velocity variable and $(y,\eta)\in \Rd\times\Rd$ to be the corresponding variable in frequency space (the variable after Fourier transform).

(i) As in \cite{Guo2003a}, the null space of $L$ is given by 
\begin{equation*}
	\ker L  = \text{span}\Big\{(1,0)\mu^{1/2},(0,1)\mu^{1/2},(1,1)v_i\mu^{1/2}(1\le i\le 3),(1,1)|v|^2\mu^{1/2}\Big\}. 
\end{equation*}
We denote $\PP$ to be the orthogonal projection from $L^2_v\times L^2_v$ onto $\ker L$, which is defined by 
\begin{equation}\label{10}
	\P f = \Big(a_+(t,x)(1,0)+a_-(t,x)(0,1)+v\cdot b(t,x)(1,1)+(|v|^2-3)c(t,x)(1,1)\Big)\mu^{1/2},
\end{equation}or equivalently by 
\begin{equation*}
	\PP f = \Big(a_\pm(t,x)+v\cdot b(t,x)+(|v|^2-3)c(t,x)\Big)\mu^{1/2}.
\end{equation*}
Then for given $f$, one can decompose $f$ uniquely as 
\begin{equation*}
	f = \P f+ (\I-\P)f. 
\end{equation*}
The function $a_\pm,b,c$ are given by 
\begin{align*}
	a_\pm &= (\mu^{1/2},f_\pm)_{L^2_v} = (\mu^{1/2},\PP f)_{L^2_v},\\
	b_j&= \frac{1}{2}(v_j\mu^{1/2},f_++f_-)_{L^2_v} = (v_j\mu^{1/2},\PP f)_{L^2_v},\\
	c&=\frac{1}{12}((|v|^2-3)\mu^{1/2},f_++f_-)_{L^2_v} = \frac{1}{6}((|v|^2-3)\mu^{1/2},\PP f)_{L^2_v}. 
\end{align*}

(ii) To describe the behavior of linearized Boltzmann collision operator, \cite{Alexandre2012} introduce the norm $\vertiii{f}$ while \cite{Gressman2011} introduce the norm $N^{s,\gamma}_l$. The previous work \cite{Global2019}\cite{Deng2020} give the pseudo-differential-type norm $\|(\tilde{a}^{1/2})^wf\|_{L^2_v}$. They are all equivalent and we list their results as follows. 

Let $\mathscr{S}'$ be the space of tempered distribution functions. $N^{s,\gamma}$ denotes the weighted geometric fractional Sobolev space 
\begin{align*}
	N^{s,\gamma} = \{f\in\mathscr{S}':|f|_{N^{s,\gamma}}<\infty\},
\end{align*}
with the anisotropic norm 
\begin{align*}
	|f|^2_{N^{s,\gamma}}:&=\|\<v\>^{\gamma/2+s}f\|^2_{L^2}+\int(\<v\>\<v'\>)^{\frac{\gamma+2s+1}{2}}\frac{(f'-f)^2}{d(v,v')^{d+2s}}\1_{d(v,v')\le 1},
\end{align*}with $d(v,v'):=\sqrt{|v-v'|^2+\frac{1}{4}(|v|^2-|v'|^2)^2}$. 
In order to describe the velocity weight $\<v\>$, \cite{Gressman2011} defined 
\begin{align*}
	|f|^2_{N^{s,\gamma}_l} = |w^l\<v\>^{\gamma/2+s}f|^2_{L^2_v}+\int_{\Rd}dv\,w^l\<v\>^{\gamma+2s+1}\int_{\Rd}dv'\,\frac{(f'-f)^2}{d(v,v')^{d+2s}}\1_{d(v,v')\le 1},
\end{align*}which turns out to be equivalent with $|w^lf|_{N^{s,\gamma}}$. This follows from the proof of Proposition 5.1 in \cite{Gressman2011} since the $\psi$ therein has a nice support. 

On the other hand, as in \cite{Alexandre2012}, we define 
\begin{align*}
	\vertiii{f}^2:&=\int B(v-v_*,\sigma)\Big(\mu_*(f'-f)^2+f^2_*((\mu')^{1/2}-\mu^{1/2})^2\Big)\,d\sigma dv_*dv,
\end{align*}

For pseudo-differential calculus as in \cite{Global2019,Deng2020}, one may refer to the appendix as well as \cite{Lerner2010} for more information. Let $\Gamma=|dv|^2+|d\eta|^2$ be an admissible metric. 
Define
\begin{align}\label{11a}
	\tilde{a}(v,\eta):=\<v\>^\gamma(1+|\eta|^2+|\eta\wedge v|^2+|v|^2)^s+K_0\<v\>^{\gamma+2s}
\end{align}to be a $\Gamma$-admissible weight, where $K_0>0$ is chosen as the following. 
Applying theorem 4.2 in \cite{Global2019} and Lemma 2.1 and 2.2 in \cite{Deng2020a}, there exists $K_0>0$ such that the Weyl quantization $\tilde{a}^w:H(\tilde{a}c)\to H(c)$ and $(\tilde{a}^{1/2})^w:H(\tilde{a}^{1/2}c)\to H(c)$ are invertible, with $c$ being any $\Gamma$-admissible metric. The weighted Sobolev space $H(c)$ is defined by \eqref{sobolev_space}. The symbol $\tilde{a}$ is real and gives the formal self-adjointness of Weyl quantization $\tilde{a}^w$. By the invertibility of $(\tilde{a}^{1/2})^w$, we have equivalence 
\begin{align*}
	\|(\tilde{a}^{1/2})^w(\cdot)\|_{L^2_v}\approx\|\cdot\|_{H(\tilde{a}^{1/2})_v},
\end{align*}and hence we will equip $H(\tilde{a}^{1/2})_v$ with norm $\|(\tilde{a}^{1/2})^w(\cdot)\|_{L^2_v}$. Also $\|w^l(\tilde{a}^{1/2})^w(\cdot)\|_{L^2_v}\approx\|(\tilde{a}^{1/2})^ww^l(\cdot)\|_{L^2_v}$ due to Lemma \ref{inverse_bounded_lemma}. 
Notice that $\|\cdot\|_{L^2_v}\lesssim \|(\tilde{a}^{1/2})^w(\cdot)\|_{L^2_v}$ for hard potential $\gamma+2s\ge 0$ and we will use this property in our proof. 

The three norms above are equivalent since for $l\in\R$, 
\begin{align*}
	\|(\tilde{a}^{1/2})^wf\|^2_{L^2_v}\approx\vertiii{f}^2\approx|f|^2_{N^{s,\gamma}}\approx (-Lf,f)_{L^2_v}+\|\<v\>^lf\|_{L^2_v},
\end{align*}which follows from (2.13)(2.15) in \cite{Gressman2011}, Proposition 2.1 in \cite{Alexandre2012} and Theorem 1.2 in \cite{Global2019}. An important result from \cite{Deng2020a} is that 
\begin{align*}
	L\in S(\tilde{a}),
\end{align*}where $S(\tilde{a})=S(\tilde{a},\Gamma)$ is the pseudo-differential symbol class; see \cite{Lerner2010}. This implies that 
\begin{align*}
	|(Lf,f)_{L^2_v}|\lesssim \|(\tilde{a}^{1/2})^wf\|_{L^2}^2. 
\end{align*}

The normal $L^2_{v,x}$ is defined as $L^2_{v,x}=L^2(\R^3_v\times\R^3_x)$. $L^2(B_C)$ is the $L^2_v$ space on Euclidean ball $B_C$ of radius $C$ at the origin. For usual Sobolev space, we will use notation 
\begin{align*}
	\|f\|_{H^k_vH^m_x} = \sum_{|\beta|\le k,|\alpha|\le m}\|\partial^\alpha_\beta f\|_{L^2_{v,x}},
\end{align*}for $k,m\ge 0$. 
We also define the standard velocity-space mixed Lebesgue space $Z_1=L^2(\R^3_v;L^1(\R^3_x))$ with the norm
\begin{equation*}
	\|f\|_{Z_1} = \Big\|\|f\|_{L^1_x}\Big\|_{L^2_v}.
\end{equation*}
In this paper, we write Fourier transform on $x$ as 
\begin{align*}
	\widehat{f}(y) = \int_{\R^3}f(x)e^{-ix\cdot y}\,dx. 
\end{align*}

\paragraph{Main results.}
To state the result of the paper, we let $K\ge 0$ to be the total order of derivatives on $v,x$ and define the velocity weight function $w^l$ for any $l\ge 0$ by \begin{align*}
	w^l = \<v\>^l.
\end{align*}
In order to extract the smoothing effect, we define a useful coefficient  
\begin{equation*}
\psi_k=\left\{\begin{aligned}
	1, \text{  if $k\le 0$},\\
	\psi^k, \text{ if $k> 0$}, 
\end{aligned}\right.
\end{equation*}
where $\psi = 1$ in Section \ref{sec4} and Theorem \ref{main1} and $\psi = t^N$ with $N>0$ large in Section \ref{sec5} and Theorem \ref{main2}. When the second case $\psi=t^N$ arise, we assume $0\le t\le 1$, since the regularity is local property. We will carry $\psi$ in our calculation for brevity of proving the smoothing effect. 
Corresponding to given $f=f(t,x,v)$, we introduce the instant energy functional $\E_{K,l}(t)$ and the instant high-order energy functional $\E^{h}_{K,l}(t)$ to be functionals satisfying the equivalent relations 
\begin{align}\label{Defe}
	\E_{K,l}(t)\notag &\approx \sum_{|\alpha|\le K}\|\psi_{|\alpha|-2}\partial^\alpha E(t)\|^2_{L^2_x}+\sum_{|\alpha|\le K}\|\psi_{|\alpha|-2}\partial^\alpha\P f\|^2_{L^2_{v,x}}+\sum_{|\alpha|\le K}\|\psi_{|\alpha|-2}\partial^{\alpha}(\I-\P)f\|^2_{L^2_{v,x}}\\
	&\qquad+\sum_{\substack{|\alpha|+|\beta|\le K}}\|\psi_{|\alpha|+|\beta|-2}w^{l-|\alpha|-|\beta|}\partial^\alpha_\beta(\I-\P) f\|^2_{L^2_{v,x}}.
\end{align}
\begin{align}\label{Defeh}
	\E^h_{K,l}(t)\notag &\approx \sum_{|\alpha|\le K}\|\psi_{|\alpha|-2}\partial^\alpha E(t)\|^2_{L^2_x}+\sum_{1\le|\alpha|\le K}\|\psi_{|\alpha|-2}\partial^\alpha\P f\|^2_{L^2_{v,x}}+\sum_{|\alpha|\le K}\|\psi_{|\alpha|-2}\partial^{\alpha}(\I-\P)f\|^2_{L^2_{v,x}}\\
	&\qquad+\sum_{\substack{|\alpha|+|\beta|\le K}}\|\psi_{|\alpha|+|\beta|-2}w^{l-|\alpha|-|\beta|}\partial^\alpha_\beta(\I-\P) f\|^2_{L^2_{v,x}}.
\end{align}
Also, we define the dissipation rate functional $\D_{K,l}$ by 
\begin{align}\label{Defd}
	\D_{K,l}(t) \notag&= \sum_{|\alpha|\le K-1}\|\psi_{|\alpha|-2}\partial^\alpha E(t)\|^2_{L^2_x}+\sum_{1\le|\alpha|\le K}\|\psi_{|\alpha|-2}\partial^\alpha\P f\|^2_{L^2_{v,x}}+\sum_{|\alpha|\le K}\|\psi_{|\alpha|-2}(\tilde{a}^{1/2})^w\partial^{\alpha}(\I-\P)f\|^2_{L^2_{v,x}}\\
	&\qquad+\sum_{\substack{|\alpha|+|\beta|\le K}}\|\psi_{|\alpha|+|\beta|-2}(\tilde{a}^{1/2})^ww^{l-|\alpha|-|\beta|}\partial^\alpha_\beta(\I-\P) f\|^2_{L^2_{v,x}}.
\end{align}
Here $E=E(t,x)$ is determined by $f(t,x,v)$ in terms of $E=-\nabla_x\phi$ and \eqref{8}. Notice that one can change the order of $(\tilde{a}^{1/2})^w$ and $w^{l-|\alpha|-|\beta|}$ due to Lemma \ref{inverse_bounded_lemma}. 
 The main results of this paper are stated as follows. 

\begin{Thm}\label{main1}
	Let $\gamma+2s\ge 0$, $0<s<1$. Define $i=1$ if $0<s<\frac{1}{2}$ and $i=2$ if $\frac{1}{2}\le s<1$. Let $K\ge i+1$, $l\ge \max\{K,\gamma/2+s+2,\gamma+2s+i+1\}$ and  $f_0(x,v)=(f_{0,+}(x,v),f_{0,-}(x,v))$ satisfying $F_\pm(0,x,v)=\mu(v)+(\mu(v))^{1/2}f_{0,\pm}(x,v)\ge 0$. Assume $\psi=1$. 
	If 
	\begin{align*}
		\epsilon_0 = (\E_{K,l}(0))^{1/2}+\|f_0\|_{Z_1}+\|E_0\|_{L^1_x},
	\end{align*}
is sufficiently small, where $E_0(x)=E(0,x)$, then there exists a unique global solution $f(t,x,v)$ to the Cauchy problem \eqref{7}-\eqref{9} of the Vlasov-Poisson-Boltzmann system such that $F_\pm(t,x,v)=\mu(v)+(\mu(v))^{1/2}f_\pm(t,x,v)\ge 0$ and 
\begin{equation}\label{15a}\begin{aligned}
		\E_{K,l}(t)&\lesssim \epsilon_0^2(1+t)^{-\frac{3}{2}},\\
		\E^h_{K,l}(t)&\lesssim \epsilon_0^2(1+t)^{-\frac{5}{2}},
	\end{aligned}
\end{equation}
for any $t\ge 0$. 
\end{Thm}
This gives the global existence to the Vlasov-Poisson-Boltzmann system with the optimal large time decay as in \cite{Duan2011}, where Duan and Strain discover the optimal large time decay for Vlasov-Maxwell-Boltzmann system. Notice that we only require $K\ge i+1$, which improve the index $K\ge 8$ in \cite{Duan2013}. 
In order to define the $a$ $priori$ assumption, for $0<T\le\infty$ and $t\in[0,T]$, we define the time-weighted energy norm $X(t)$ by 
\begin{align*}
	X(t) = \sup_{0\le\tau\le t}(1+\tau)^{3/2}\E_{K,l}(\tau)+\sup_{0\le\tau\le t}(1+\tau)^{5/2}\E^h_{K,l}(\tau).
\end{align*}
Here the high-order energy functional $\E^h_{K,l}$ has time decay rate $(1+t)^{-5/2}$ while $\E_{K,l}$ has time decay rate $(1+t)^{-3/2}$. They are all optimal as in the Boltzmann equation case \cite{Strain2012} and the Vlasov-Maxwell-Boltzmann system case \cite{Duan2011}. 
Let $\delta_0>0$ and the $a$ $priori$ assumption to be 
\begin{align}
	\label{priori}\sup_{0\le t\le T}X(t)\le \delta_0. 
\end{align}
Then we will obtain the following closed $a$ $priori$ estimate
\begin{align*}
	X(t)\lesssim \epsilon^2_0+X^{3/2}(t)+X^2(t).
\end{align*}

In order to extract the smoothing effect on $x$, we define a symbol $\tilde{b}$ by 
\begin{align}\label{11}
	\tilde{b}(v,y) = (1+|v|^2+|y|^2+|v\wedge y|^2)^{\delta_1}, 
\end{align}
where $\delta_1>0$ is defined by \eqref{107} and \eqref{108}. Notice that we will require $\gamma+2s>0$ here and in the next main result. 
\begin{Thm}\label{main2}Let $\gamma+2s>0$, $0<\tau<T\le \infty$.  
For any $l\ge K\ge 3$, assume $\psi=t^N$ with $N>0$ large and $l\ge K$. Let $f$ to be the solution to \eqref{7}-\eqref{9} satisfying that
	\begin{align}
		\epsilon_1 = (\E_{i+1,l}(0))^{1/2}+\|f_0\|_{Z_1}+\|E_0\|_{L^1_x} <\infty
	\end{align}is sufficiently small, where $i=1$ if $0<s<\frac{1}{2}$ and $i=2$ if $\frac{1}{2}\le s<1$. Then for $|\alpha|+|\beta|\le K$, 
\begin{align}\label{19a}
	\sup_{\tau\le t\le T}\|w^{l-|\alpha|-|\beta|}\partial^\alpha_\beta f\|^2_{L^2_{v,x}}\le C_\tau\epsilon^2_1<\infty,
\end{align}
where $C_\tau>0$ depends on $\tau$. Moreover, if additionally  
\begin{align*}
	\sup_{l_0\ge i+1}\E_{i+1,l_0}(0)<\infty, 
\end{align*}is sufficiently small, then for $|\alpha|+|\beta|\le K\le l$, $k\ge 0$,  
	\begin{align}\label{19b}
		\sup_{\tau\le t\le T}\|w^{l-|\alpha|-|\beta|}\partial^\alpha_\beta\partial^{k}_tf\|^2_{L^2_{v,x}}\le C_{\tau,k}<\infty,
	\end{align}where $C_{\tau,k}$ is a constant depending on $\tau$, $k$. 
Consequently, $f\in C^\infty(\R^+_t;C^\infty(\R^3_x;\mathscr{S}(\R^3_v)))$. 
\end{Thm}
This result is similar to the Boltzmann equation case; see \cite{Alexandre2010}. That is, whenever the initial data has exponential decay, the solution $f$ is Schwartz in $v$ and smooth in $(t,x)$ for any positive time $t$. 

In what follows let us point out several technical points in the proof of Theorem \ref{main1} and \ref{main2}. For Theorem \ref{main1}, firstly, we use $K\ge 2$ because of $H^2_x(\R^3)$ is a Banach algebra when controlling \eqref{12} and it's useful when dealing with the trilinear estimate. Secondly, the velocity weight $w^{l-|\alpha|-|\beta|}$ will help us deal with the term $v\cdot\nabla_x\phi f$ when bounding 
\begin{align*}
	\big(\sum_{\alpha_1\le\alpha}v\cdot\nabla_x\partial^{\alpha_1}\phi \partial^{\alpha-\alpha_1} f,e^{\pm\phi}w^{2l-2|\alpha|-2|\beta|}\partial^\alpha f\big)_{L^2_{v,x}},
\end{align*}where $|\alpha|\le K$. The case $\alpha_1=0$ will be eliminated by the similar term corresponding to $v\cdot\nabla_xf$ as in \cite{Guo2012}. This is what $e^{\pm\phi}$ designed for. The case $\alpha_1\neq 0$ can be bounded due to the weight $w^{l-|\alpha|-|\beta|}$. 
For the term $\nabla_x\phi\cdot\nabla_vf$, one will need to bound 
\begin{align*}
	\big(\sum_{\alpha_1\le\alpha}\nabla_x\partial^{\alpha_1}\phi \partial^{\alpha-\alpha_1}\nabla_v f,e^{\pm\phi}w^{2l-2|\alpha|-2|\beta|}\partial^\alpha f\big)_{L^2_{v,x}}.
\end{align*}
This term will transfer one derivative from $x$ to one derivative on $v$ and so one should require $|\alpha|+|\beta|\le K$. If $\alpha_1=0$, we can use integration by parts to move $\nabla_v$ to the weight $w^{2l-2|\alpha|-2|\beta|}$, while if $\alpha_1\neq 0$, the total order on the first $f$ is less or equal to $K$ and hence can be control by our energy functional $\E_{K,l}$ or $\D_{K,l}$. As in \cite{Duan2013}, one has to bound the term 
\begin{align*}
	\|\partial_t\phi\|_{L^\infty_x}\E_{K,l},
\end{align*}which cannot be absorbed by the energy dissipation norm. But observing that $\|\partial_t\phi\|_{L^\infty_x}$ is bounded by the high-order energy functional $\E^h_{K,l}$ and hence integrable as shown in \cite{Guo2012}, one can use the Gronwall's inequality to close the $a$ $priori$ estimate. 

The second technical point concerns the choice of $\psi=t^N$ in Theorem \ref{main2} and the usage of $\psi_{|\alpha|+|\beta|-2}$ is Section \ref{sec5}. Firstly, whenever $|\alpha|+|\beta|\ge 2$, $\psi_{|\alpha|+|\beta|-2} = t^{N(|\alpha|+|\beta|-2)}$ is equal to $0$ at $t=0$. Plugging this into energy estimate, one can easily deduce the smoothing effect locally in time, since the initial data becomes zero. By using the global energy control obtained in Theorem \ref{main1}, the local regularity becomes global-in-time regularity. Notice that we use $-2$ to eliminate the index arising from Sobolev embedding $\|\cdot\|_{L^\infty_x}\lesssim \|\cdot\|_{H^2_x}$. However, after adding the term $\psi_{|\alpha|+|\beta|-2}$, one need to deal with the term 
\begin{align*}
	\big(\partial_t(\psi_{|\alpha|+|\beta|-2})\partial^\alpha_\beta f,e^{\pm\phi}\partial^\alpha_\beta f\big)_{L^2_{v,x}}.
\end{align*}
Using the symbols \eqref{11a} and \eqref{11}, we can control this term by pseudo-differential norms with a little higher-order, where these pseudo-differential norms can be controlled by the functional $\E_{K,l}$ and $\D_{K,l}$. Hence, we can obtain a closed energy estimate locally. Together with the global energy control in Theorem \ref{main1}, one can deduce the regularity for any positive time $t>0$.

The rest of the paper is arranged as follows. In Section 2, we present some basic estimates for $L, \Gamma$, and tricks in energy estimates. In Section 3, we list the macroscopic energy estimates. In Section 4, we use the $a$ $priori$ estimate to perform proof of existence. In Section 5, we present the proof for regularity.

\section{Preliminaries}

In this section, we list several basic lemmas corresponding to the linearized Boltzmann collision term $L_\pm$ and the bilinear Boltzmann collision operator $\Gamma_\pm$. Recall $w^l = \<v\>^l$. The following lemma concerns with dissipation of $L_\pm$, whose proof can be found in \cite[Lemma 2.6 and Theorem 8.1]{Gressman2011}.

\begin{Lem}\label{lemmaL}For any $l\in\R$, multi-indices $\alpha,\beta$, we have the followings. 
	
	(i) It holds that \begin{equation*}
		(-Lg,g)_{L^2_v}\gtrsim \|(\tilde{a}^{1/2})^w(\I-\P)g\|^2_{L^2_v}.
	\end{equation*}

(ii) There exists $C>0$ such that 
\begin{align*}
	-(w^{2l}Lg,g)_{L^2_v}\gtrsim \|(\tilde{a}^{1/2})^ww^lg\|^2_{L^2_v}-C\|g\|^2_{L^2_v(B_C)}.
\end{align*}

(iii) For any $\eta>0$, 
\begin{align*}
	-(w^{2l-2|\alpha|-2|\beta|}\partial^\alpha_\beta Lg,\partial^\alpha_\beta g)_{L^2_v}&\gtrsim \|(\tilde{a}^{1/2})^ww^{l-|\alpha|-|\beta|}\partial^\alpha_\beta g\|^2_{L^2_v} - \eta\sum_{|\beta_1|\le|\beta|}\|(\tilde{a}^{1/2})^ww^{l-|\alpha|-|\beta_1|}\partial^\alpha_{\beta_1}g\|^2_{L^2_v}\\
	&\qquad\qquad\qquad\qquad\qquad\qquad\qquad\qquad-C_\eta\|\partial^\alpha g\|^2_{L^2(B_{C_\eta})}.
\end{align*}

\end{Lem}

The next lemmas concern the estimates on the nonlinear collision operator $\Gamma_\pm$. 
We will use the estimate in \cite[Lemma 2.2]{Duan2013} and the estimate from \cite[Proposition 3.1]{Strain2012}.

\begin{Lem}\label{Lem22}
	For any $ l\ge 0$, $m\ge 0$ and multi-index $\beta$, we have the upper bound 
	\begin{align}\label{12}
		|(w^{2l}&\notag\partial^\alpha_\beta\Gamma_\pm(f,g),\partial^\alpha_\beta h)_{L^2_{v,x}}|\\&\notag\lesssim \sum_{\substack{\alpha_1+\alpha_2=\alpha\\\beta_1+\beta_2\le\beta}}\int_{\R^3}\|\partial^{\alpha_1}_{\beta_1}f\|_{L^2_v}\|(\tilde{a}^{1/2})^ww^l\partial^{\alpha_2}_{\beta_2}g\|_{L^2_v}\|(\tilde{a}^{1/2})^ww^l\partial^\alpha_\beta h\|_{L^2_v}\,dx\\&\quad+ \sum_{\substack{\alpha_1+\alpha_2=\alpha\\\beta_1+\beta_2\le\beta}}\int_{\R^3}\|w^l\partial^{\alpha_1}_{\beta_1}f\|_{L^2_v}\|(\tilde{a}^{1/2})^w\partial^{\alpha_2}_{\beta_2}g\|_{L^2_v}\|(\tilde{a}^{1/2})^ww^l\partial^\alpha_\beta h\|_{L^2_v}\,dx.
	\end{align}
Let $i=1$ if $0<s<1/2$ and $i=2$ if $1/2\le s<1$, then 
\begin{align}\label{12a}
	\|w^l\Gamma(f,g)\|_{L^2_v}\lesssim \|w^l\<v\>^{\frac{\gamma+2s}{2}}f\|_{L^2_v}\|w^l\<v\>^{\gamma+2s}g\|_{H^i_v}
\end{align}
\end{Lem}
The estimate \eqref{12a} comes from \cite[Proposition 3.1]{Strain2012}, so we only give a short proof of \eqref{12}. 
As in \cite{Gressman2011}, we need some preparations as the followings. 
Notice that from Carleman representation \eqref{Carleman}, the derivative on $v$ will apply to $f,g$ and $\mu^{1/2}$ respectively. Then, 
\begin{equation*}
	\psi_{|\alpha|+|\beta|-2}\partial^\alpha_\beta\T(f,g) = \sum_{\alpha_1+\alpha_2=\alpha}\sum_{\beta_1+\beta_2+\beta_3=\beta}C^{\alpha_1,\alpha_2}_\alpha C^{\beta_1,\beta_2,\beta_3}_{\beta}\psi_{|\alpha|+|\beta|-2}\T_{\beta_3}(\partial^{\alpha_1}_{\beta_1}f,\partial^{\alpha_2}_{\beta_2}g). 
\end{equation*}
Let $\{\chi_k\}_{k=-\infty}^{k=+\infty}$ be a partition of unity on $(0,\infty)$ such that $|\chi_k|\le 1$ and supp$(\chi_k)\subset [2^{-k-1},2^{-k}]$. For each $k$, we define 
\begin{align*}
	B_k = B(v-v_*,\sigma)\chi_k(|v-v'|).
\end{align*}
Now we denote 
\begin{align*}
	T^{k,l}_+(f,g,h)=\int_{\R^3}dv\int_{\R^3}dv_*\int_{\mathbb{S}^{2}}d\sigma\,B_k(v-v_*,\sigma)f_*gh'w^{2l}(v')\partial_{\beta_3}(\mu^{1/2}(v_*')),\\
	T^{k,l}_-(f,g,h)=\int_{\R^3}dv\int_{\R^3}dv_*\int_{\mathbb{S}^{2}}d\sigma\,B_k(v-v_*,\sigma)f_*ghw^{2l}(v)\partial_{\beta_3}(\mu^{1/2}(v_*)).
\end{align*}
On the other hand, we can express the collision operator $Q$ by using its dual formulation as in \cite[A1]{Gressman2011}. Indeed, after a transformation, we can put cancellations on $g$ as 
\begin{align*}
	(w^{2l}\T(f,g),h)_{L^2_v}&=\int_{\R^3}dv\int_{\R^3}dv_*\int_{\mathbb{S}^{2}}d\sigma\,\tilde{B}(v-v_*,\sigma)f_*h'w^{2l}(v')\big(\mu^{1/2}(v_*')g(v)-\mu^{1/2}(v_*)g(v')\big)\\&\qquad+T^l_*(f,g,h),
\end{align*}
where 
\begin{align*}
	\tilde{B}(v-v_*,\sigma) = \frac{4B(v-v_*,\frac{2v'-v-v_*}{|2v'-v-v_*|})}{|v'-v_*||v-v_*|}
\end{align*}
and the operator $T^l_*(f,g,h)$ does not differentiate:
\begin{align*}
	T^l_*(f,g,h) = \int_{\R^3}dv'\int_{\R^3}dv_*\int_{E^{v'}_{v_*}}d\pi_v\,f_*g'h'w^{2l}(v')\partial_{\beta_3}(\mu^{1/2}(v_*))\tilde{B}\Big(1-\frac{|v'-v_*|^{3+\gamma}}{|v-v_*|^{3+\gamma}}\Big).
\end{align*}
Here $d\pi_v$ is Lebesgue measure on the $2$-dimensional plane $E^{v'}_{v_*}$ passing through $v'$ with normal $v'-v_*$, i.e. $E^{v'}_{v_*}=\{v\in\R^3:(v-v')\cdot(v_*-v')=0\}$, and $v$ is the variable of integration. With the observation above, we can use the following alternative representations for $T^{k,l}_+$ as well as a third trilinear operator $T^{k,l}_*$:
\begin{align*}
	T^{k,l}_+(f,g,h) = \int_{\R^3}dv'\int_{\R^3}dv_*\int_{E^{v'}_{v_*}}d\pi_v\,\tilde{B}_kf_*gh'w^{2l}(v')\partial_{\beta_3}(\mu^{1/2}(v'_*)),\\
	T^{k,l}_*(f,g,h) = \int_{\R^3}dv'\int_{\R^3}dv_*\int_{E^{v'}_{v_*}}d\pi_v\,\tilde{B}_kf_*g'h'w^{2l}(v')\partial_{\beta_3}(\mu^{1/2}(v_*)),
\end{align*}
where we use the notation 
\begin{align*}
	\tilde{B}_k = \frac{4B(v-v_*,\frac{2v'-v-v_*}{|2v'-v-v_*|})}{|v'-v_*||v-v_*|}\chi_k(|v-v'|).
\end{align*}
Then for $f,g,h\in\mathscr{S}(\R^3)$, we can use the pre-post collisional change of variables, the dual representation, and the previous calculation guarantee that 
\begin{align*}
	(w^{2l}\T_{\beta_3}(f,g),h)_{L^2_v} &= \sum^\infty_{k=-\infty}\big(T^{k,l}_+(f,g,h)-T^{k,l}_-(f,g,h)\big)\\
	&= T^l_*(f,g,h) + \sum^\infty_{k=-\infty}\big(T^{k,l}_+(f,g,h)-T^{k,l}_*(f,g,h)\big).
\end{align*}
Now we collect the estimates for the operators $T^{k,l}_+$, $T^{k,l}_-$ and $T^{k,l}_*$, which can be used to prove \eqref{12}. 
\begin{Prop}\label{prop}
	Let $k$ be an integer, $m\ge 0$, $l\in\R$. We have the following uniform estimates. 
	
	(i) \begin{equation*}
		|T^{k,l}_-(f,g,h)|\lesssim 2^{2sk}\|\<v\>^{-m}f\|_{L^2_v}\|\<v\>^{\gamma/2+s}w^lg\|_{L^2_v}\|\<v\>^{\gamma/2+s}w^lh\|_{L^2_v}.
	\end{equation*}

(ii)
\begin{equation*}
	|T^{k,l}_*(f,g,h)|\lesssim 2^{2sk}\|\<v\>^{-m}f\|_{L^2_v}\|\<v\>^{\gamma/2+s}w^lg\|_{L^2_v}\|\<v\>^{\gamma/2+s}w^lh\|_{L^2_v}.
\end{equation*}

(iii)
\begin{align*}
	|T^{k,l}_+(f,g,h)|&\lesssim 2^{2sk}\|f\|_{L^2_v}\|\<v\>^{\gamma/2+s}w^lg\|_{L^2_v}\|\<v\>^{\gamma/2+s}w^lh\|_{L^2_v}\\
	&\quad+2^{2sk}\|w^lf\|_{L^2_v}\|\<v\>^{\gamma/2+s}g\|_{L^2_v}\|\<v\>^{\gamma/2+s}w^lh\|_{L^2_v}.
\end{align*}
\end{Prop}
\begin{proof}
	First of all, notice that (i) and (ii) are the same as \cite[Proposition 3.1, 3.2]{Gressman2011}. So we only prove (iii). 
	The key point is to assign the velocity weight to $f$ and $g$ in a better way. 
	The following inequality will frequently be used:
	\begin{align}\label{aa}
		\int_{\S^2}B_k(v-v_*,\sigma)\,d\sigma\lesssim |v-v_*|^\gamma\int^{2^{-k}|v-v_*|^{-1}}_{2^{-k-1}|v-v_*|^{-1}}\theta^{-1-2s}\,d\theta\lesssim 2^{2sk}|v-v_*|^{\gamma+2s}. 
	\end{align}
	By Cauchy-Schwarz,
	\begin{align*}
		|T^{k,l}_+(f,g,h)|&\lesssim\Big(\int_{\R^3}dv\int_{\R^3}dv_*\int_{\mathbb{S}^{2}}d\sigma\,B_k|v-v_*|^{-\gamma-2s}|f_*g|\<v\>^{\gamma+2s}w^{2l}(v')|\partial_{\beta_3}(\mu^{1/2}(v_*'))|\Big)^{1/2}\\
		&\quad\times\Big(\int_{\R^3}dv\int_{\R^3}dv_*\int_{\mathbb{S}^{2}}d\sigma\,B_k|v-v_*|^{\gamma+2s}|h'|\<v\>^{-\gamma-2s}w^{2l}(v')|\partial_{\beta_3}(\mu^{1/2}(v_*'))|\Big)^{1/2}\\
		&:= I\times J.
	\end{align*}
For the term $I$, if $|v'|\le \frac{1}{2}(|v|^2+|v_*|^2)$, the collisional conservation laws imply $\mu^{1/4}(v_*')\le \mu^{1/8}(v_*)\mu^{1/8}(v)=\mu^{1/8}(v_*')\mu^{1/8}(v')$. It follows that 
\begin{align*}
	\<v\>^{\gamma+2s}w^{2l}(v')|\partial_{\beta_3}(\mu^{1/2}(v_*'))|\lesssim \mu^{1/16}(v_*)\mu^{1/16}(v).
\end{align*}
If $|v'|> \frac{1}{2}(|v|^2+|v_*|^2)$, then $|v'|^2\approx |v|^2+|v_*|^2$ and $w^{2l}(v')\approx w^{2l}(v)+w^{2l}(v_*)$. Hence,  
\begin{align*}
	\<v\>^{\gamma+2s}w^{2l}(v')|\partial_{\beta_3}(\mu^{1/2}(v_*'))|\lesssim(w^{2l}(v)+w^{2l}(v_*))\<v\>^{\gamma+2s}. 
\end{align*}
Thus, by using \eqref{aa}, 
\begin{align*}
	I&\lesssim \Big(\int_{\R^3}dv\int_{\R^3}dv_*2^{2sk}|f_*g|\<v\>^{\gamma+2s}(w^{2l}(v)+w^{2l}(v_*))|\Big)^{1/2}\\
	&\lesssim 2^{sk}\Big(\|w^lf\|_{L^2_v}\|\<v\>^{\gamma/2+s}g\|_{L^2_v}+\|f\|_{L^2_v}\|\<v\>^{\gamma/2+s}w^lg\|_{L^2_v}\Big).
\end{align*}
For the term $J$, since $|v|-|v_*|\le |v-v_*|=|v'-v_*|\frac{1}{\cos(\theta/2)}\lesssim |v'|+|v_*|$ with $\theta\in(0,\pi/2]$, it follows that $|v|\lesssim |v'|+|v_*|$ and hence, 
\begin{align*}
	\<v\>^{\gamma+2s}\lesssim \<v'\>^{\gamma+2s}\<v_*\>^{\gamma+2s}. 
\end{align*}
Thus, by using \eqref{aa} and pre-post change of variable,  
\begin{align*}
	J&=\Big(\int_{\R^3}dv\int_{\R^3}dv_*\int_{\mathbb{S}^{2}}d\sigma\,B_k|v-v_*|^{\gamma+2s}|h|\<v'\>^{-\gamma-2s}w^{2l}(v)|\partial_{\beta_3}(\mu^{1/2}(v_*))|\Big)^{1/2}\\
	&\lesssim 2^{sk}\Big(\int_{\R^3}dv\int_{\R^3}dv_*|v-v_*|^{2\gamma+4s}|h|\<v\>^{-\gamma-2s}\<v_*\>^{\gamma+2s}w^{2l}(v)|\partial_{\beta_3}(\mu^{1/2}(v_*))|\Big)^{1/2}\\
	&\lesssim 2^{sk}\|\<v\>^{\gamma/2+s}w^lh\|_{L^2_v}, 
\end{align*}
where we used the fact that 
\begin{align*}
	\int_{\R^3}\,dv_*\mu^{\lambda}(v_*)|v-v_*|^{2\gamma+4s}\lesssim \<v\>^{2\gamma+4s}, 
\end{align*}whenever $\gamma+2s>-\frac{3}{2}$ and $\lambda>0$. Together with the estimate of $I$, we complete the proof of Proposition \ref{prop}. 

\end{proof}

\begin{proof}
	[Proof of \eqref{12}]In terms of estimates obtained in Propositions \ref{prop}, by applying the cancellation inequalities constructed in \cite[Proposition 3.6, 3.7]{Gressman2011} and carrying out the similar procedure as that of \cite[Section 6.1]{Gressman2011}, one can prove \eqref{12} and the details are omitted for brevity. This completes the proof of Lemma \ref{Lem22}.
\end{proof}

In order to obtain a suitable norm estimate of $\Gamma$ on $x$. We shall write the following estimate, which is also very useful throughout our analysis. 

\begin{Lem}
	For any $u,v\in H^2_x$, we have 
	\begin{equation}\label{13}
		\begin{aligned}
			\|uv\|_{L^2_x}&\lesssim \|\nabla_xu\|_{H^1_x}\|v\|_{L^2_x},\\
			\|uv\|_{L^2_x}&\lesssim \|\nabla_xu\|_{L^2_x}\|v\|_{H^1_x}.
		\end{aligned}
	\end{equation}
Consequently, 
	\begin{align}\label{14}
		\|uv\|_{H^2_x}\le \|\nabla_xu\|_{H^1_x}\|\nabla_xv\|_{H^1_x}. 
	\end{align}
\end{Lem}
\begin{proof}
	The proof is straight forward. Notice that this lemma give that $H^2_x$ is a Banach algebra. 
	By Gagliardo-Nirenberg interpolation inequality and Sobolev embedding, we have 
	\begin{align*}
		\|u\|_{L^\infty}&\lesssim \|\nabla_xu\|^{1/2}\|\nabla^2_xu\|^{1/2}\lesssim \|\nabla_xu\|_{H^1},\\
		\|uv\|_{L^2}&\lesssim \|u\|_{L^6}\|v\|_{L^3}\lesssim \|\nabla_xu\|_{L^2}\|v\|_{H^1}.
	\end{align*}
Then \eqref{13} follows from H\"{o}lder's inequality. For \eqref{14}, 
	\begin{align*}
		\|uv\|_{H^2_x}\notag &= \sum_{|\alpha|\le 2}\|\partial^\alpha(uv)\|_{L^2}\\
		&\lesssim \sum_{|\alpha|=2}\|u\partial^\alpha v\|_{L^2}+\sum_{|\alpha|=|\beta|=1}\|\partial^\alpha u\partial^\beta v\|_{L^2} + \sum_{|\alpha|=2}\|\partial^\alpha uv\|_{L^2}\\
		&\lesssim \notag\|u\|_{L^\infty}\|\nabla_xv\|_{H^1}+\sum_{|\alpha|=|\beta|=1}\|\partial^\alpha u\|_{L^3}\|\partial^\beta v\|_{L^6} + \|\nabla_xu\|_{H^1}\|v\|_{L^\infty}.
	\end{align*}
Plugging the \eqref{13} estimate into this inequality, we have the desired control. 
\end{proof}

With the help of the above lemma, we can control the trilinear term $(\partial^\alpha_\beta\Gamma_\pm(f,g),\partial^\alpha_\beta h)_{L^2_{v,x}}$. 
\begin{Lem}\label{lemmat}
	Let $K\ge 2$. For any multi-indices $|\alpha|+|\beta|\le K$ and real number $l\ge K$, we have \begin{equation}\begin{aligned}\label{15}
		\Big|&(\psi_{2|\alpha|+2|\beta|-4}w^{2l-2|\alpha|-2|\beta|}\notag\partial^\alpha_\beta\Gamma_\pm(f,g),\partial^\alpha_\beta h)_{L^2_{v,x}}\Big|\\&\lesssim \bigg(\sum_{|\alpha|+|\beta|\le K}\|\psi_{|\alpha|+|\beta|-2}\partial^{\alpha}_\beta f\|_{L^2_{v,x}}\sum_{\substack{|\alpha|\ge 1\\ |\alpha|+|\beta|\le K}}\|\psi_{|\alpha|+|\beta|-2}(\tilde{a}^{1/2})^ww^{l-|\alpha|-|\beta|}\partial^{\alpha}_{\beta}g\|_{L^2_{v,x}}\\
		&\quad+\sum_{\substack{|\alpha|\ge 1\\ |\alpha|+|\beta|\le K}}\|\psi_{|\alpha|+|\beta|-2}\partial^{\alpha}_\beta f\|_{L^2_{v,x}}
		\sum_{|\alpha|+|\beta|\le K}\|\psi_{|\alpha|+|\beta|-2}(\tilde{a}^{1/2})^ww^{l-|\alpha|-|\beta|}\partial^{\alpha}_{\beta}g\|_{L^2_{v,x}}\\
		&\quad+\sum_{|\alpha|+|\beta|\le K}\|\psi_{|\alpha|+|\beta|-2}w^{l-|\alpha|-|\beta|}\partial^{\alpha}_\beta f\|_{L^2_{v,x}}\sum_{\substack{|\alpha|\ge 1\\ |\alpha|+|\beta|\le K}}\|\psi_{|\alpha|+|\beta|-2}(\tilde{a}^{1/2})^w\partial^{\alpha}_{\beta}g\|_{L^2_{v,x}}\\
		&\quad+\sum_{\substack{|\alpha|\ge 1\\ |\alpha|+|\beta|\le K}}\|\psi_{|\alpha|+|\beta|-2}w^{l-|\alpha|-|\beta|}\partial^{\alpha}_\beta f\|_{L^2_{v,x}}
		\sum_{|\alpha|+|\beta|\le K}\|\psi_{|\alpha|+|\beta|-2}(\tilde{a}^{1/2})^w\partial^{\alpha}_{\beta}g\|_{L^2_{v,x}}\bigg)
		\\
		&\qquad\qquad\qquad\qquad\qquad\qquad\qquad\qquad\qquad\times\|	\psi_{|\alpha|+|\beta|-2}(\tilde{a}^{1/2})^ww^{l-|\alpha|-|\beta|}\partial^{\alpha}_\beta h\|_{L^2_{v,x}},\end{aligned}
	\end{equation}
	where we restrict $t\in[0,1]$ when $\psi=t^N$ as in Theorem \ref{main2}.
\end{Lem}
\begin{proof}
Using the estimate \eqref{12}, we have 
\begin{align*}
&\quad\,\big|(\psi_{2|\alpha|+2|\beta|-4}w^{2l-2|\alpha|-2|\beta|}\partial^\alpha_\beta\Gamma_\pm(f,g),\partial^\alpha_\beta h)_{L^2_{v,x}}\big| \\
&\lesssim \sum_{\substack{\alpha_1+\alpha_2=\alpha\\\beta_1+\beta_2\le\beta}}\Big\|\psi_{|\alpha|+|\beta|-2}\|\partial^{\alpha_1}_{\beta_1}f\|_{L^2_v}\|(\tilde{a}^{1/2})^ww^{l-|\alpha|-|\beta|}\partial^{\alpha_2}_{\beta_2}g\|_{L^2_v}\Big\|_{L^2_x}\|\psi_{|\alpha|+|\beta|-2}(\tilde{a}^{1/2})^ww^{l-|\alpha|-|\beta|}\partial^\alpha_\beta h\|_{L^2_{v,x}}\\&\quad+ \sum_{\substack{\alpha_1+\alpha_2=\alpha\\\beta_1+\beta_2\le\beta}}\Big\|\psi_{|\alpha|+|\beta|-2}\|w^{l-|\alpha|-|\beta|}\partial^{\alpha_1}_{\beta_1}f\|_{L^2_v}\|(\tilde{a}^{1/2})^w\partial^{\alpha_2}_{\beta_2}g\|_{L^2_v}\Big\|_{L^2_x}\|\psi_{|\alpha|+|\beta|-2}(\tilde{a}^{1/2})^ww^{l-|\alpha|-|\beta|}\partial^\alpha_\beta h\|_{L^2_{v,x}},
\end{align*}by choosing $l-|\alpha|-|\beta|$ to be the $l$ in \eqref{12}. 
Here we divide the summation into several parts. For brevity we denote the first terms in the norm $\|\cdot\|_{L^2_x}$ inside the summation $\sum_{\substack{\alpha_1+\alpha_2=\alpha\\\beta_1+\beta_2\le\beta}}$ on the right hand side to be $I,J$ and discuss their value in several cases. 
If $2\le|\alpha_1|+|\beta_1|\le K$, then $|\alpha_2|+|\beta_2|\le |\alpha|+|\beta|-2$ and $l-|\alpha|-|\beta|\le l-|\alpha_2+\alpha'|-|\beta_2|$ for any $1\le|\alpha'|\le2$. Notice that here we will give $\psi_{|\alpha_1|+|\beta_1|-2}$ to $f$ and $\psi_{|\alpha_2|+|\beta_2|}$ to $g$. Also, $\psi_{|\alpha_2|+|\beta_2|}\le\psi_{|\alpha_2+\alpha'|+|\beta_2|-2}$. By using \eqref{13}$_1$, we have 
\begin{align*}
	I&\lesssim\psi_{|\alpha|+|\beta|-2}\|\partial^{\alpha_1}_{\beta_1}f\|_{L^2_{v,x}}\big\|\|(\tilde{a}^{1/2})^ww^{l-|\alpha|-|\beta|}\partial^{\alpha_2}_{\beta_2}g\|_{L^2_v}\big\|_{L^\infty_x}\\
	&\lesssim\|\psi_{|\alpha_1|+|\beta_1|-2}\partial^{\alpha_1}_{\beta_1}f\|_{L^2_{v,x}}\sum_{1\le|\alpha'|\le2}\|\psi_{|\alpha_2+\alpha'|+|\beta_2|-2}(\tilde{a}^{1/2})^ww^{l-|\alpha_2+\alpha'|-|\beta_2|}\partial^{\alpha_2+\alpha'}_{\beta_2}g\|_{L^2_{v,x}}\\
	&\lesssim\sum_{|\alpha|+|\beta|\le K}\|\psi_{|\alpha|+|\beta|-2}\partial^{\alpha}_\beta f\|_{L^2_{v,x}}\sum_{\substack{|\alpha|\ge 1\\ |\alpha|+|\beta|\le K}}\|\psi_{|\alpha|+|\beta|-2}(\tilde{a}^{1/2})^ww^{l-|\alpha|-|\beta|}\partial^{\alpha}_{\beta}g\|_{L^2_{v,x}}.
\end{align*}
Secondly, if $|\alpha_1|+|\beta_1|=1$, then $|\alpha_2|+|\beta_2|\le |\alpha|+|\beta|-1$ and by using \eqref{13}$_2$, we have 
\begin{align*}
	I&\lesssim \sum_{|\alpha'|= 1}\|\psi_{|\alpha_1+\alpha'|+|\beta_1|-2}\partial^{\alpha_1+\alpha'}_{\beta_1}f\|_{L^2_{v,x}}\sum_{|\alpha'|\le 1}\|\psi_{|\alpha_2+\alpha'|+|\beta_2|-2}(\tilde{a}^{1/2})^ww^{l-|\alpha_2+\alpha'|-|\beta_2|}\partial^{\alpha_2+\alpha'}_{\beta_2}g\|_{L^2_{v,x}}\\
	&\lesssim \sum_{\substack{|\alpha|\ge 1\\|\alpha|+|\beta|\le K}}\|\psi_{|\alpha|+|\beta|-2}\partial^{\alpha}_\beta f\|_{L^2_{v,x}}\sum_{|\alpha|+|\beta|\le K}\|\psi_{|\alpha|+|\beta|-2}(\tilde{a}^{1/2})^ww^{l-|\alpha|-|\beta|}\partial^{\alpha}_{\beta}g\|_{L^2_{v,x}}.
 \end{align*}
Here we used $\psi\le1$ and $\psi_{|\alpha|+|\beta|-2}\le \psi_{|\alpha_1+\alpha'_1|+|\beta_1|-2}\psi_{|\alpha_2+\alpha'_2|+|\beta_2|-2}$, for any $|\alpha'_1|= 1$, $|\alpha'_2|\le 1$.  
Thirdly, if $|\alpha_1|+|\beta_1|=0$, then by \eqref{13}$_1$, we have 
\begin{align*}
	I&\lesssim\sum_{1\le|\alpha'|\le2}\|\psi_{|\alpha_1+\alpha'|+|\beta_1|-2}\partial^{\alpha_1+\alpha'}_{\beta_1}f\|_{L^2_{v,x}}\|\psi_{|\alpha_2|+|\beta_2|-2}(\tilde{a}^{1/2})^ww^{l-|\alpha_2|-|\beta_2|}\partial^{\alpha_2}_{\beta_2}g\|_{L^2_{v,x}}\\
	&\lesssim \sum_{\substack{|\alpha|\ge 1\\ |\alpha|+|\beta|\le K}}\|\psi_{|\alpha|+|\beta|-2}\partial^{\alpha}_\beta f\|_{L^2_{v,x}}
	\sum_{|\alpha|+|\beta|\le K}\|\psi_{|\alpha|+|\beta|-2}(\tilde{a}^{1/2})^ww^{l-|\alpha|-|\beta|}\partial^{\alpha}_{\beta}g\|_{L^2_{v,x}}.
\end{align*}
Here we used $\psi_{|\alpha|+|\beta|-2}\le \psi_{|\alpha_1+\alpha'|+|\beta_1|-2}\psi_{|\alpha_2|+|\beta_2|-2}$, for any $|\alpha'|\le2$.  
Combining the above estimate, we have the desired result for $I$:
\begin{align*}
	I&\lesssim \sum_{|\alpha|+|\beta|\le K}\|\psi_{|\alpha|+|\beta|-2}\partial^{\alpha}_\beta f\|_{L^2_{v,x}}\sum_{\substack{|\alpha|\ge 1\\ |\alpha|+|\beta|\le K}}\|\psi_{|\alpha|+|\beta|-2}(\tilde{a}^{1/2})^ww^{l-|\alpha|-|\beta|}\partial^{\alpha}_{\beta}g\|_{L^2_{v,x}}\\
	&\qquad+\sum_{\substack{|\alpha|\ge 1\\ |\alpha|+|\beta|\le K}}\|\psi_{|\alpha|+|\beta|-2}\partial^{\alpha}_\beta f\|_{L^2_{v,x}}
	\sum_{|\alpha|+|\beta|\le K}\|\psi_{|\alpha|+|\beta|-2}(\tilde{a}^{1/2})^ww^{l-|\alpha|-|\beta|}\partial^{\alpha}_{\beta}g\|_{L^2_{v,x}}.
\end{align*}
Similarly, using the same discussion on $|\alpha_2|+|\beta_2|$ instead of $|\alpha_1|+|\beta_1|$, we have 
\begin{align*}
	J&\lesssim \sum_{|\alpha|+|\beta|\le K}\|\psi_{|\alpha|+|\beta|-2}w^{l-|\alpha|-|\beta|}\partial^{\alpha}_\beta f\|_{L^2_{v,x}}\sum_{\substack{|\alpha|\ge 1\\ |\alpha|+|\beta|\le K}}\|\psi_{|\alpha|+|\beta|-2}(\tilde{a}^{1/2})^w\partial^{\alpha}_{\beta}g\|_{L^2_{v,x}}\\
	&\qquad+\sum_{\substack{|\alpha|\ge 1\\ |\alpha|+|\beta|\le K}}\|\psi_{|\alpha|+|\beta|-2}w^{l-|\alpha|-|\beta|}\partial^{\alpha}_\beta f\|_{L^2_{v,x}}
	\sum_{|\alpha|+|\beta|\le K}\|\psi_{|\alpha|+|\beta|-2}(\tilde{a}^{1/2})^w\partial^{\alpha}_{\beta}g\|_{L^2_{v,x}}.
\end{align*}
Combining all the above estimate, we have the desired bound. 
Similar discussion on the indices $|\alpha_1|+|\beta_1|$ will be used frequently later and will not be mentioned for brevity. 


\end{proof}

A direct consequence of Lemma \ref{lemmat} is the following estimate; see \cite[Lemma 3.1]{Duan2013}.
\begin{Lem}\label{lemmag}
	Let $K\ge 2$, $|\alpha|+|\beta|\le K$, $l\ge K$. Then,
	\begin{equation*}
		|(\partial^\alpha\Gamma_\pm(f,f),\partial^\alpha f_\pm)_{L^2_{v,x}}|\lesssim\E^{1/2}_{K,l}\D_{K,l}(t),
	\end{equation*}
and
	\begin{equation*}
		|(w^{2l-2|\alpha|-2|\beta|}\partial^\alpha_\beta\Gamma_\pm(f,f),\partial^\alpha_\beta(\II-\PP)f)_{L^2_{v,x}}|\lesssim \E^{1/2}_{K,l}\D_{K,l}(t),
	\end{equation*}
In particular, when $|\alpha|\ge 1$, 
\begin{equation*}
|(w^{2l-2|\alpha|-2|\beta|}\partial^\alpha_\beta\Gamma_\pm(f,f),\partial^\alpha_\beta f)_{L^2_{v,x}}|\lesssim \E^{1/2}_{K,l}\D_{K,l}(t).
\end{equation*}
When $0\le |\alpha|\le K$, 
\begin{align*}
	|(w^{2l-2|\alpha|-2|\beta|}\partial^\alpha_\beta\Gamma_\pm(f,f),\partial^\alpha_\beta f)_{L^2_{v,x}}|\lesssim\E^{1/2}_{K,l}\D_{K,l}(t)+\E_{K,l}(t)\D^{1/2}_{K,l}(t)
\end{align*}
Also, for any function $\zeta(v)$ satisfying $|\zeta(v)|\approx e^{-\lambda|v|^2}$ for some $\lambda>0$, we have 
\begin{align*}
	(\partial^\alpha\Gamma_\pm(f,f),\zeta(v))_{L_{v,x}}\lesssim \E^{1/2}_{K,l}\D^{1/2}_{K,l}(t).
\end{align*}
\end{Lem}
\begin{proof}
	With Lemma \ref{lemmat}, these energy estimate can be verified directly by the definition of $\E$ and $\D$ as in \cite[Lemma 3.1]{Duan2013}, and details are omitted for brevity. Also, notice that $\Gamma_\pm$ is in $\ker L$ and . 
\end{proof}

\section{Macroscopic Estimate}
In this section, we assume $\psi=1$. We will analyze the macroscopic dissipation by taking the macroscopic projection on the equation \eqref{7}. Since we are dealing with Vlasov-Poisson-Boltzmann system, the idea here is similar to the Boltzmann equation case \cite{Gressman2011} and Vlasov-Maxwell-Boltzmann system case \cite{Duan2011}. But there's still some difference between these equations and Vlasov-Poisson-Boltzmann system and we will write a detailed proof for the sake of completeness. Notice that the calculation in this section is valid for both hard potential $\gamma+2s\ge 0$ and soft potential $\gamma+2s<0$. 

Recall the projection notation in \eqref{10}. By multiplying the equation \eqref{7} with $\mu^{1/2}, v_j\mu^{1/2}(j=1,2,3)$ and $\frac{1}{6}(|v|^2-3)\mu^{1/2}$ and then integrating them over the $\R^3_v$, we have 
\begin{equation}\label{17}\left\{\begin{aligned}
	&\partial_ta_\pm + \nabla\cdot b + \nabla_x\cdot(v\mu^{1/2},(\II-\PP)f)_{L^2_v} =0,\\
	&\partial_t\big(b_j+(v_j\mu^{1/2},(\II-\PP)f)_{L^2_v}\big)+\partial_j(a_\pm+2c)\mp E_j+(v_j\mu^{1/2},v\cdot\nabla_x(\II-\PP)f)_{L^2_v}\\&\qquad\qquad\qquad\qquad\qquad\qquad\qquad\qquad\qquad = (L_\pm f+g_\pm,v_j\mu^{1/2})_{L^2_v},\\
	&\partial_t\Big(c+\frac{1}{6}((|v|^2-3)\mu^{1/2},(\II-\PP)f)_{L^2_v}\Big)+\frac{1}{3}\nabla_x\cdot b + \frac{1}{6}((|v|^2-3)\mu^{1/2},v\cdot\nabla(\II-\PP)f)_{L^2_v} \\&\qquad\qquad\qquad\qquad\qquad\qquad\qquad\qquad\qquad= \frac{1}{6}(L_\pm f+g_\pm,(|v|^2-3)\mu^{1/2})_{L^2_v},
\end{aligned}\right.
\end{equation}
where for brevity, we denote $I=(I_+,I_-)$ with $I_\pm f=f_\pm$ and 
\begin{align}\label{22a}
		g_\pm = \pm\nabla_x\phi\cdot\nabla_vf_\pm\mp\frac{1}{2}\nabla_x\phi\cdot vf_\pm+\Gamma_\pm(f,f).  
\end{align}
Notice that $(\P_\pm f,v\mu^{1/2})_{L^2_v}$ and $(\P_\pm f,(|v|^2-3)\mu^{1/2})_{L^2_v}$ is not $0$ in general and similar for $\Gamma_\pm$. Also, we have used
\begin{align*}
	(\pm\nabla_x\phi\cdot\nabla_vf_\pm\mp\frac{1}{2}\nabla_x\phi\cdot vf_\pm,\mu^{1/2})_{L^2_v}=0,
\end{align*}which is obtained by integration by parts on $v$. In order to obtain the high-order moments, as in \cite{Duan2011}, we define for $1\le j,k\le 3$ that 
\begin{align*}
	\Theta_{jk}(f_\pm) = ((v_iv_j-1)\mu^{1/2},f_\pm)_{L^2_v},\ \ \Lambda_j(f_\pm) =\frac{1}{10}((|v|^2-5)v_j\mu^{1/2},f_\pm)_{L^2_v}. 
\end{align*}
Then multiplying the above high-order moments with equation \eqref{7}, we have 
\begin{equation}\label{18}\left\{
	\begin{aligned}
		&\partial_t\big(\Theta_{jj}((\II-\PP)f)+2c\big) + 2\partial_jb_j = \Theta_{jj}(g_\pm+h_\pm),\\
		&\partial_t\Theta_{jk}((\II-\PP)f)+\partial_jb_k+\partial_kb_j + \nabla_x\cdot(v\mu^{1/2},(\II-\PP)f)_{L^2_v} = \Theta_{jk}(g_\pm+h_\pm)+(\mu^{1/2},g_\pm)_{L^2_v},\ j\neq k,\\
		&\partial_t\Lambda_j((\II-\PP)f)+\partial_jc = \Lambda_j(g_\pm+h_\pm),
	\end{aligned}\right.
\end{equation}
where 
\begin{align*}
	h_\pm = -v\cdot\nabla_x(\II-\PP)f-L_\pm f. 
\end{align*}
By taking the mean value of every two equations with sign $\pm$ in \eqref{17}, we have 
\begin{equation}\label{19}\left\{
	\begin{aligned}
		&\partial_t\Big(\frac{a_++a_-}{2}\Big)+\nabla_x\cdot b = 0,\\
		&\partial_tb_j+\partial_j\Big(\Big(\frac{a_++a_-}{2}\Big)+2c\Big)+\frac{1}{2}\sum_{k=1}^3\partial_k\Theta_{jk}((\I-\P)f\cdot(1,1))
		= \frac{1}{2}(g_++g_-,v_j\mu^{1/2})_{L^2_v},\\
		&\partial_tc+\frac{1}{3}\nabla_x\cdot b + \frac{5}{6}\sum^3_{j=1}\partial_j\Lambda_j((\I-\P)f\cdot(1,1)) = \frac{1}{12}(g_++g_-,(|v|^2-3)\mu^{1/2})_{L^2_v},
	\end{aligned}\right.
\end{equation}for $1\le j\le3$. Similarly, taking the mean value with $\pm$ of the equation in \eqref{18}, we have 
\begin{equation}\label{20}
	\left\{\begin{aligned}
		&\partial_t\Big(\frac{1}{2}\Theta_{jk}((\II-\PP)f\cdot(1,1))+2c\delta_{jk}\Big) + \partial_jb_k+\partial_kb_j = \frac{1}{2}\Theta_{jk}(g_++g_-+h_++h_-),\\
		&\frac{1}{2}\partial_t\Lambda_j((\II-\PP)f\cdot(1,1))+\partial_jc = \frac{1}{2}\Lambda_j(g_++g_-+h_++h_-),
	\end{aligned}\right.
\end{equation}
for $1\le j,k\le 3$. $\delta_{jk}$ is the Kronecker delta. Moreover, for obtaining the dissipation of the electric field $E$, we take the difference with sign $\pm$ in the first two equations in \eqref{17}, we have 
\begin{equation}\label{21}\left\{
	\begin{aligned}
		&\partial_t(a_+-a_-)+\nabla_x\cdot G=0,\\
		&\partial_tG + \nabla_x(a_+-a_-)-2E+\nabla_x\cdot\Theta((\I-\P)f\cdot(1,-1))=(v\mu^{1/2},(g+Lf)\cdot(1,-1))_{L^2_v},
	\end{aligned}\right.
\end{equation}
where 
\begin{align}\label{27aa}
	G = (v\mu^{1/2},(\I-\P)f\cdot(1,-1))_{L^2_v}.
\end{align}
Recall that $E=-\nabla_x\phi$. Then by equation \eqref{8}, we have 
\begin{align}\label{16}
	\nabla_x\cdot E = a_+-a_-. 
\end{align}
In order to extract the dissipation rate of $a_\pm,b,c,E$, we would like to take the Fourier transform on the equation \eqref{19}, \eqref{20}, \eqref{21} and \eqref{16} with respect to $x$. Then
\begin{equation}\label{22}
	\left\{\begin{aligned}
		&\partial_t\Big(\frac{\widehat{a_+}+\widehat{a_-}}{2}\Big)+iy\cdot \widehat{b} = 0,\\
		&\partial_t\widehat{b_j}+iy_j\Big(\Big(\frac{\widehat{a_+}+\widehat{a_-}}{2}\Big)+2\widehat{c}\Big)+\frac{1}{2}\sum_{k=1}^3iy_k\Theta_{jk}((\I-\P)\widehat{f}\cdot(1,1))
		= \frac{1}{2}(\widehat{g_+}+\widehat{g_-},v_j\mu^{1/2})_{L^2_v},\\
		&\partial_t\widehat{c}+\frac{1}{3}iy\cdot \widehat{b} + \frac{5}{6}\sum^3_{j=1}iy_j\Lambda_j((\I-\P)\widehat{f}\cdot(1,1)) = \frac{1}{12}(\widehat{g_+}+\widehat{g_-},(|v|^2-3)\mu^{1/2})_{L^2_v},\\
		&\partial_t\Big(\frac{1}{2}\Theta_{jk}((\I-\P)\widehat{f}\cdot(1,1))+2\widehat{c}\delta_{jk}\Big) + iy_j\widehat{b_k}+iy_k\widehat{b_j} = \frac{1}{2}\Theta_{jk}(\widehat{g_+}+\widehat{g_-}+\widehat{h_+}+\widehat{h_-}),\\
		&\frac{1}{2}\partial_t\Lambda_j((\I-\P)\widehat{f}\cdot(1,1))+iy_j\widehat{c} = \frac{1}{2}\Lambda_j(\widehat{g_+}+\widehat{g_-}+\widehat{h_+}+\widehat{h_-}),
	\end{aligned}\right.
\end{equation}
\begin{equation}\label{23}\left\{
	\begin{aligned}
		&\partial_t(\widehat{a_+}-\widehat{a_-})+iy\cdot \widehat{G}=0,\\
		&\partial_t\widehat{G}+iy(\widehat{a_+}-\widehat{a_-})-2\widehat{E}+iy\cdot\Theta((\I-\P)\widehat{f}\cdot(1,-1))=(v\mu^{1/2},(g+Lf)\cdot(1,-1))_{L^2_v},\\
		&iy\cdot \widehat{E}=\widehat{a_+}-\widehat{a_-}.
	\end{aligned}\right.
\end{equation}

\begin{Lem}\label{Lemma31}
	Let $(f,E)$ be the solution to the Cauchy problem \eqref{7}-\eqref{9}. For any $K\ge 2$, there exists a functional $\E^{(1)}_{K}(t),\E^{(1)}_{K,h}(t)$ such that
	\begin{align}\label{27a}
		\E^{(1)}_{K}&\lesssim \sum_{|\alpha|\le K}\|\partial^\alpha f\|^2_{L^2_{x,v}}+\sum_{|\alpha|\le K-1}\|\partial^\alpha E\|^2_{L^2_x},\\
		\label{27b}\E^{(1)}_{K,h}&\lesssim \sum_{1\le|\alpha|\le K}\|\partial^\alpha\P f\|^2_{L^2_{x,v}}+\sum_{|\alpha|\le K}\|\partial^\alpha(\I-\P) f\|^2_{L^2_{x,v}}+\sum_{|\alpha|\le K-1}\|\partial^\alpha E\|^2_{L^2_x},
	\end{align}
and for any $t\ge 0$, 
\begin{equation}\label{24}\begin{aligned}
		\partial_t&\E^{(1)}_{K} + \lambda\sum_{|\alpha|\le K-1} \|\partial^\alpha\nabla_x(a_\pm,b,c)\|^2_{L^2_x}+\|a_+-a_-\|^2_{L^2_x}+\sum_{|\alpha|\le K-1}\|\partial^\alpha E\|^2_{L^2_x}\\
		&\qquad\lesssim\sum_{|\alpha|\le K} \|(\tilde{a}^{1/2})^w(\I-\P)\partial^{\alpha}{f}\|_{L^2_{v,x}}^2+\E_{K,l}(t)\D_{K,l}(t).
	\end{aligned}
\end{equation}
\begin{equation}\label{24a}\begin{aligned}
	\partial_t&\E^{(1)}_{K,h} + \lambda\sum_{1\le|\alpha|\le K-1} \|\partial^\alpha\nabla_x(a_\pm,b,c)\|^2_{L^2_x}+\|\nabla_x(a_+-a_-)\|^2_{L^2_x}+\sum_{|\alpha|\le K-1}\|\partial^\alpha E\|^2_{L^2_x}\\
	&\qquad\lesssim\sum_{|\alpha|\le K} \|(\tilde{a}^{1/2})^w(\I-\P)\partial^{\alpha}{f}\|_{L^2_{v,x}}^2+\E^h_{K,l}(t)\D_{K,l}(t).
\end{aligned}
\end{equation}
\end{Lem}
\begin{proof}We only need to prove the case of $|\alpha|=0$. Since equations \eqref{19}, \eqref{20} and \eqref{16} are linear in $a_\pm,b,c,f$, one can directly apply the derivative $\partial^\alpha$ to them. The results \eqref{24}\eqref{24a} follows similarly. 
	 Let $\zeta(v)$ be a function satisfying 
	 \begin{align*}
	 	|\zeta(v)|\approx e^{-\lambda|v|^2},
	 \end{align*} for some $\lambda>0$. Notice that we will use notation $\zeta$ for different function satisfying the above equivalence. 
	
	For the estimate on $\widehat{a_+}+\widehat{a_-}$, we use the equation \eqref{22}$_2$ to get 
	\begin{align*}
		|y|^2\Big|\frac{\widehat{a_+}+\widehat{a_-}}{2}\Big|^2
		&=\sum^{3}_{j=1}\Big(iy_j\frac{\widehat{a_+}+\widehat{a_-}}{2}\Big|iy_j\frac{\widehat{a_+}+\widehat{a_-}}{2}\Big)\\
		&=\sum^{3}_{j=1}\Big(iy_j\frac{\widehat{a_+}+\widehat{a_-}}{2}\Big|-\partial_t\widehat{b_j}-2iy_j\widehat{c}-\frac{1}{2}\sum_{k=1}^3iy_k\Theta_{jk}((\I-\P)\widehat{f}\cdot(1,1))+ \frac{1}{2}(\widehat{g_+}+\widehat{g_-},v_j\mu^{1/2})_{L^2_v}\Big)\\
		&=-\partial_t\sum^{3}_{j=1}\Big(iy_j\frac{\widehat{a_+}+\widehat{a_-}}{2}\Big|\widehat{b_j}\Big)+\sum^{3}_{j=1}\Big(iy_j\frac{\partial_t\widehat{a_+}+\partial_t\widehat{a_-}}{2}\Big|\widehat{b_j}\Big)\\&\qquad+\sum^{3}_{j=1}\Big(iy_j\frac{\widehat{a_+}+\widehat{a_-}}{2}\Big|-2iy_j\widehat{c}-\frac{1}{2}\sum_{k=1}^3iy_k\Theta_{jk}((\I-\P)\widehat{f}\cdot(1,1))+ \frac{1}{2}(\widehat{g_+}+\widehat{g_-},v_j\mu^{1/2})_{L^2_v}\Big)
	\end{align*}
To deal with the terms $\partial_t\widehat{a_+}+\partial_t\widehat{a_-}$, we use the equation \eqref{22}$_1$. Then by Cauchy-Schwarz inequality, we obtain 
\begin{align}\label{25}
	\partial_t\sum^{3}_{j=1}&\Re\Big(iy_j\frac{\widehat{a_+}+\widehat{a_-}}{2}\Big|\widehat{b_j}\Big)+\lambda|y|^2\Big|\frac{\widehat{a_+}+\widehat{a_-}}{2}\Big|^2\\
	&\lesssim |y\cdot\widehat{b}|^2+|y|^2|\widehat{c}|^2+|y|^2\|\zeta(\I-\P)\widehat{f}\|_{L^2_v}^2+|(\widehat{g_+}+\widehat{g_-},\zeta)_{L^2_v}|^2.\notag
\end{align} 

For the estimate of $\widehat{b}$, we use the equation \eqref{22}$_4$. 
\begin{align*}
	&\quad\,2|y|^2|\widehat{b}|^2+2|y\cdot\widehat{b}|^2\\
	&= \sum^{3}_{j,k=1}|iy_j\widehat{b_k}+iy_k\widehat{b_j}|^2\\
	&= \sum^{3}_{j,k=1}\Big(iy_j\widehat{b_k}+iy_k\widehat{b_j}\Big|-\partial_t\Big(\frac{1}{2}\Theta_{jk}((\I-\P)\widehat{f}\cdot(1,1))+2\widehat{c}\delta_{jk}\Big) + \frac{1}{2}\Theta_{jk}(\widehat{g_+}+\widehat{g_-}+\widehat{h_+}+\widehat{h_-})\Big)\\
	&=-\partial_t\sum^{3}_{j,k=1}\Big(iy_j\widehat{b_k}+iy_k\widehat{b_j}\Big|\frac{1}{2}\Theta_{jk}((\I-\P)\widehat{f}\cdot(1,1))+2\widehat{c}\delta_{jk}\Big)\\&\qquad+\sum^{3}_{j,k=1}\Big(iy_j\partial_t\widehat{b_k}+iy_k\partial_t\widehat{b_j}\Big|\frac{1}{2}\Theta_{jk}((\I-\P)\widehat{f}\cdot(1,1))+2\widehat{c}\delta_{jk}\Big)\\&\qquad+\sum^{3}_{j,k=1}\Big(iy_j\widehat{b_k}+iy_k\widehat{b_j}\Big|\frac{1}{2}\Theta_{jk}(\widehat{g_+}+\widehat{g_-}+\widehat{h_+}+\widehat{h_-})\Big)
\end{align*}
To eliminate the terms $\partial_t\widehat{b_k}$ and $\partial_t\widehat{b_j}$, we will use equation \eqref{22}$_2$. Thus, by Cauchy-Schwarz inequality, 
\begin{equation}\label{26}\begin{aligned}
		\partial_t&\sum^{3}_{j,k=1}\Re\Big(iy_j\widehat{b_k}+iy_k\widehat{b_j}\Big|\frac{1}{2}\Theta_{jk}((\I-\P)\widehat{f}\cdot(1,1))+2\widehat{c}\delta_{jk}\Big)+\lambda|y|^2|\widehat{b}|^2+2|y\cdot\widehat{b}|^2\\
		&\lesssim \delta|y|^2|\widehat{a_+}+\widehat{a_-}|^2+C_\delta|y|^2|\widehat{c}|^2+C_\delta|y|^2\|\zeta(\I-\P)\widehat{f}\|_{L^2_v}^2+|(\widehat{g_+}+\widehat{g_-}+\widehat{h_+}+\widehat{h_-},\zeta)_{L^2_v}|^2.
	\end{aligned}
\end{equation}

For the estimate of $\widehat{c}$, we use the equation \eqref{22}$_5$. 
\begin{align*}
	|y|^2|\widehat{c}|^2&=\sum^{3}_{j=1}\Big(iy_j\widehat{c}\Big|iy_j\widehat{c}\Big)\\
	&=\sum^{3}_{j=1}\Big(iy_j\widehat{c}\Big|-\frac{1}{2}\partial_t\Lambda_j((\I-\P)\widehat{f}\cdot(1,1))+ \frac{1}{2}\Lambda_j(\widehat{g_+}+\widehat{g_-}+\widehat{h_+}+\widehat{h_-})\Big)\\
	&=-\frac{1}{2}\partial_t\sum^{3}_{j=1}\Big(iy_j\widehat{c}\Big|\Lambda_j((\I-\P)\widehat{f}\cdot(1,1))\Big)+\frac{1}{2}\sum^{3}_{j=1}\Big(iy_j\partial_t\widehat{c}\Big|\Lambda_j((\I-\P)\widehat{f}\cdot(1,1))\Big)\\&\qquad+\frac{1}{2}\sum^{3}_{j=1}\Big(iy_j\widehat{c}\Big| \Lambda_j(\widehat{g_+}+\widehat{g_-}+\widehat{h_+}+\widehat{h_-})\Big).
\end{align*}
To eliminate the term $\partial_t\widehat{c}$, we use the equation \eqref{22}$_3$ to get 
\begin{equation}\label{27}\begin{aligned}
		\frac{1}{2}\partial_t&\Re\sum^{3}_{j=1}\Big(iy_j\widehat{c}\Big|\Lambda_j((\I-\P)\widehat{f}\cdot(1,1))\Big)+\lambda|y|^2|\widehat{c}|^2\\
		&\lesssim \delta|y|^2|\widehat{b}|^2 + |y|^2\|\zeta(\I-\P)\widehat{f}\|_{L^2_v}^2+|(\widehat{g_+}+\widehat{g_-}+\widehat{h_+}+\widehat{h_-},\zeta)_{L^2_v}|^2,
	\end{aligned}
\end{equation}for any $\delta>0$.

To obtain the dissipation rate, we use linear combination $\kappa_1\times\eqref{25}+\kappa_2\times\eqref{26}+\eqref{27}$ and choose $\delta<<\kappa_1<<\kappa_2$ sufficiently small to get 
\begin{equation}\label{28}\begin{aligned}
		&\quad\,\partial_t\Bigg(\kappa_1\Re\sum^{3}_{j=1}\Big(iy_j\frac{\widehat{a_+}+\widehat{a_-}}{2}\Big|\widehat{b_j}\Big)+\kappa_2\Re\sum^{3}_{j,k=1}\Big(iy_j\widehat{b_k}+iy_k\widehat{b_j}\Big|\frac{1}{2}\Re\Theta_{jk}((\I-\P)\widehat{f}\cdot(1,1))+2\widehat{c}\delta_{jk}\Big)\\&\qquad+	\frac{1}{2}\sum^{3}_{j=1}\Big(iy_j\widehat{c}\Big|\Lambda_j((\I-\P)\widehat{f}\cdot(1,1))\Big)\Bigg)+\lambda|y|^2\Big|\frac{\widehat{a_+}+\widehat{a_-}}{2}\Big|^2+\lambda|y|^2|\widehat{b}|^2+\lambda|y|^2|\widehat{c}|^2\\
		&\lesssim (1+|y|^2)\|(\tilde{a}^{1/2})^w(\I-\P)\widehat{f}\|_{L^2_v}^2+|(\widehat{g_+}+\widehat{g_-},\zeta)_{L^2_v}|^2.
	\end{aligned}
\end{equation}
Here we use the fact that 
\begin{align*}
|(\widehat{h_+}+\widehat{h_-},\zeta)_{L^2_v}|^2\lesssim (1+|y|^2)\|(\tilde{a}^{1/2})^w(\I-\P)f\|^2_{L^2_v}.
\end{align*}
which follows from the definition of $h_\pm$ and Lemma \ref{lemmaL}. Notice that $\widehat{a_+}+\widehat{a_-},\widehat{b},\widehat{c}$ have coefficient $|y|$. It means that the above estimate has one order of derivative on $x$. 
In order to obtain the dissipation property of $\widehat{a_\pm}$, we will discuss the dissipation of $\widehat{a_+}-\widehat{a_-}$, since 
\begin{align*}
	|\widehat{a_+}|^2+|\widehat{a_-}|^2=\frac{|\widehat{a_+}-\widehat{a_-}|^2}{2}+\frac{|\widehat{a_+}-\widehat{a_-}|^2}{2}.
\end{align*}
Now we observe from equation \eqref{23}$_3$ and \eqref{23}$_2$ that 
\begin{align*}
	(|y|^2+2)|\widehat{a_+}-\widehat{a_-}|^2 &= \Big(iy(\widehat{a_+}-\widehat{a_-})\Big|iy(\widehat{a_+}-\widehat{a_-})-2\widehat{E}\Big)\\
	&=\Big(iy(\widehat{a_+}-\widehat{a_-})\Big|-\partial_t\widehat{G}-iy\cdot\Theta((\I-\P)\widehat{f}\cdot(1,-1))+(v\mu^{1/2},(\widehat{g}+L\widehat{f})\cdot(1,-1))_{L^2_v}\Big)\\
	&=-\partial_t\Big(iy(\widehat{a_+}-\widehat{a_-})\Big|\widehat{G}\Big)+\Big(iy(\partial_t\widehat{a_+}-\partial_t\widehat{a_-})\Big|\widehat{G}\Big)\\&\qquad+\Big(iy(\widehat{a_+}-\widehat{a_-})\Big|-iy\cdot\Theta((\I-\P)\widehat{f}\cdot(1,-1))+(v\mu^{1/2},(\widehat{g}+L\widehat{f})\cdot(1,-1))_{L^2_v}\Big).
\end{align*}
By using the \eqref{23}$_1$, we obtain by Cauchy-Schwarz inequality that 
\begin{equation}\label{29}\begin{aligned}
		\partial_t\Re\Big(iy(\widehat{a_+}-\widehat{a_-})\Big|\widehat{G}\Big)+\lambda(|y|^2+2)|\widehat{a_+}-\widehat{a_-}|^2 &\lesssim |y\cdot\widehat{G}|^2+|y|^2\|(\tilde{a}^{1/2})\zeta(\I-\P)\widehat{f}\|_{L^2_v}^2+|(\widehat{g},\zeta)_{L^2_v}|^2\\&\lesssim |y|^2\|(\tilde{a}^{1/2})\zeta(\I-\P)\widehat{f}\|_{L^2_v}^2+|(\widehat{g},\zeta)_{L^2_v}|^2,
	\end{aligned}
\end{equation}
by using the inequality \eqref{15} and $L\in S(\tilde{a})$. Recall that $G$ is defined as $G = (v\mu^{1/2},(\I-\P)f\cdot(1,-1))_{L^2_v}$. 
Moreover, we need the dissipation rate on $\widehat{E}$. Hence, by using equation \eqref{23}$_2$, 
\begin{align*}
	2|\widehat{E}|^2 &= \big(\widehat{E}\big|2\widehat{E}\big)\\
	&= \big(\widehat{E}\big|\partial_t\widehat{G}+iy(\widehat{a_+}-\widehat{a_-})+iy\cdot\Theta((\I-\P)\widehat{f}\cdot(1,-1))-(v\mu^{1/2},(\widehat{g}+L\widehat{f})\cdot(1,-1))_{L^2_v}\big)\\
	&=\partial_t\big(\widehat{E}\big|\widehat{G}\big)-\big(\partial_t\widehat{E}\big|\widehat{G}\big)+\big(\widehat{E}\big|iy(\widehat{a_+}-\widehat{a_-})+iy\cdot\Theta((\I-\P)\widehat{f}\cdot(1,-1))-(v\mu^{1/2},(\widehat{g}+L\widehat{f})\cdot(1,-1))_{L^2_v}\big)
\end{align*}
Here by \eqref{8} and \eqref{23}$_1$, 
\begin{align*}
	\partial_tE = -\partial_t\nabla_x\phi = -\nabla_x\Delta^{-1}_x\partial_t(a_+-a_-) = -\nabla_x\Delta^{-1}_x\nabla_x\cdot G. 
\end{align*}
Thus by Fourier transform, 
\begin{align*}
	|\partial_t\widehat{E}|^2\lesssim |\widehat{G}|^2\lesssim \|\mu^{1/2}(\I-\P)\widehat{f}\|^2_{L^2}. 
\end{align*}
Plugging this into the above estimate, we have,
\begin{align}\label{30}
	\partial_t\Re\big(-\widehat{E}\big|\widehat{G}\big)+\lambda|\widehat{E}|^2
	&\lesssim|y|^2|\widehat{a_+}-\widehat{a_-}|^2+ (1+|y|^2)\|(\tilde{a}^{1/2})^w(\I-\P)f\|^2_{L^2_v}+|(\widehat{g},\zeta)_{L^2_v}|^2,
\end{align}where $\zeta\in S(\tilde{a}^{1/2})$ and Lemma \ref{inverse_bounded_lemma} are applied. 
Taking the combination $\eqref{29} +\kappa_3\times \eqref{30}$, we have 
\begin{align}\label{31}
	\partial_t\big(iy(\widehat{a_+}-\widehat{a_-})-\kappa_3\widehat{E}\big|\widehat{G}\big)+\lambda(1+|y|^2)|\widehat{a_+}-\widehat{a_-}|^2 +\lambda|\widehat{E}|^2
	&\lesssim (1+|y|^2)\|(\tilde{a}^{1/2})\zeta(\I-\P)\widehat{f}\|_{L^2_v}^2+|(\widehat{g},\zeta)_{L^2_v}|^2.
\end{align}
Now we take the integral on \eqref{28} and \eqref{31} with respect to $y$, use the Plancherel's Theorem and sum this two inequality together. Then,
\begin{align*}
	\partial_t\E^{(1)}_{1}+\lambda\|\nabla_x(a_\pm,b,c)\|^2_{L^2_x}+\lambda\|a_+-a_-\|^2_{L^2_x}+\lambda\|E\|^2_{L^2_x}\lesssim \|(\tilde{a}^{1/2})\zeta(\I-\P)\widehat{f}\|_{L^2_vH^1_x}^2+|(\widehat{g},\zeta)_{L^2_v}|^2,
\end{align*}
where 
\begin{align*}
	\E^{(1)}_{1} &=\kappa_1\sum^{3}_{j=1}\Re\Big(\nabla_x\frac{{a_+}+{a_-}}{2},b\Big)_{L^2_x}+\kappa_2\sum^{3}_{j,k=1}\Re\Big(\partial_j{b_k}+\partial_k{b_j},\frac{1}{2}\Theta_{jk}((\II-\PP){f}\cdot(1,1))+2{c}\delta_{jk}\Big)_{L^2_x}\\&\qquad+	\frac{1}{2}\Re\sum^{3}_{j=1}\Big(\partial_j{c},\Lambda_j((\II-\PP){f}\cdot(1,1))\Big)_{L^2_x}+\Re\big(\nabla_x({a_+}-{a_-})-\kappa_3{E},{G}\big)_{L^2_x}.
\end{align*}
As we mentioned at the beginning of this proof, the equations \eqref{19} and \eqref{20} are linear in $a_\pm,b,c,f$, one can directly apply the derivative $\partial^\alpha$ to \eqref{19} and \eqref{20} with $|\alpha|\le K-1$. Then we can get the high-order estimate. For any $m\ge0$, we define 
\begin{align*}
	&\E^{(1)}_{K}
	=\sum_{|\alpha|\le K-1}\bigg(\kappa_1\sum^{3}_{j=1}\Big(\nabla_x\frac{{\partial^\alpha a_+}+{\partial^\alpha a_-}}{2},\partial^\alpha b\Big)_{L^2_x}\\
	&+\kappa_2\sum^{3}_{j,k=1}\Re\Big(\partial^\alpha\partial_j{b_k}+\partial^\alpha\partial_k{b_j},\frac{1}{2}\Theta_{jk}((\II-\PP)\partial^\alpha{f}\cdot(1,1))+2\partial^\alpha{c}\delta_{jk}\Big)_{L^2_x}\\
	&+	\frac{1}{2}\sum^{3}_{j=1}\Big(\partial^\alpha\partial_j{c},\Lambda_j((\II-\PP)\partial^\alpha{f}\cdot(1,1))\Big)_{L^2_x}\\
	&+\big(\nabla_x(\partial^\alpha{a_+}-\partial^\alpha{a_-})-\kappa_3{\partial^\alpha E},\partial^\alpha{G}\big)_{L^2_x}\bigg).
\end{align*}
Then, 
\begin{align*}
	\partial_t&\E^{(1)}_{K} + \lambda\sum_{|\alpha|\le K-1} \|\partial^\alpha\nabla_x(a_\pm,b,c)\|^2_{L^2_x}+\lambda\|a_+-a_-\|^2_{L^2_x}+\lambda\sum_{|\alpha|\le K-1}\|\partial^\alpha E\|^2_{L^2_x}\\
	&\qquad\lesssim\sum_{|\alpha|\le K} \|(\tilde{a}^{1/2})^w(\I-\P)\partial^{\alpha}\widehat{f}\|_{L^2_{v,x}}^2+\sum_{|\alpha|\le K-1}\|(\partial^{\alpha}g,\zeta)_{L^2_v}\|^2_{L^2_x}.
\end{align*}
Also, $\E^{(1)}_{K}\lesssim \sum_{|\alpha|\le K}\|\partial^\alpha f\|^2_{L^2_{x,v}}+\sum_{|\alpha|\le K-1}\|\partial^\alpha E\|^2_{L^2_x}$ can be easily verified by direct calculation. 
Finally, we only need to estimate $\|(\partial^{\alpha}g,\zeta)_{L^2_v}\|^2_{L^2_x}$ for $|\alpha|\le K-1$. 
By Lemma \ref{lemmag}, 
\begin{align*}
	\|(\partial^{\alpha}g,\zeta)_{L^2_v}\|_{L^2_x}^2&\lesssim \|(\partial^{\alpha}(\pm\nabla_x\phi\cdot\nabla_vf_\pm\mp\frac{1}{2}\nabla_x\phi\cdot vf_\pm+\Gamma_\pm(f,f)),\zeta)_{L^2_v}\|_{L^2_x}^2\\
	&\lesssim \E_{K,l}(t)\D_{K,l}(t). 
\end{align*}
This completes the proof of \eqref{24}. The proof of \eqref{24a} is similar, which is by directly applying the derivative $\partial^\alpha$ to \eqref{19}\eqref{21}$_1$ with $1\le|\alpha|\le K-1$ instead of $|\alpha|\le K-1$. On the other hand, we still apply $\partial^\alpha$ to \eqref{21}$_2$ with $|\alpha|\le K-1$. Then we will obtain \eqref{24a}.
\end{proof}

Consider the homogeneous linearized system 
\begin{equation}\label{39}
	\left\{\begin{aligned}
		&\partial_tf_\pm +v\cdot\nabla_x f_\pm\pm \mu^{1/2}v\cdot\nabla_x\phi + Lf_\pm = 0,\\
		&-\Delta_x\phi = \int_{\R^3}(f_+-f_-)\mu^{1/2}\,dv,\quad \phi\to 0\text{ as }|x|\to\infty,\\
		&f_\pm|_{t=0}=f_{0,\pm},
	\end{aligned} \right.
\end{equation}which is \eqref{7}-\eqref{9} with $g_\pm=0$ in \eqref{22a}. We write the formal solution to Cauchy problem \eqref{39} to be 
\begin{equation}\label{42b}
	f = e^{tB}f_0,
\end{equation}where $e^{tB}$ denotes the solution operator. For later use, we will analyze the large time behavior of system \eqref{39}. The idea here follows from \cite{Strain2012}.
\begin{Thm}\label{homogen}
	Let $f= e^{tB}f_0$ be the solution to \eqref{39}, $m\ge 0$ be an integer and time decay rate index to be 
	\begin{equation*}
		\sigma_m = \frac{3}{4}+\frac{m}{2}.
	\end{equation*}
Then for $l\ge 0$, $t\ge 0$, 
\begin{equation}
	\|w^l\nabla^m_xf(t)\|_{L^2_{v,x}}+\|\nabla_x^mE(t)\|_{L^2_x} \lesssim (1+t)^{-\sigma_m}\big(\|w^lf_0\|_{Z_1}+\|E_0\|_{L_1}+\|w^l\nabla_x^mf_0\|_{L^2_{v,x}}\big).
\end{equation}
\end{Thm}

Before proving this result, we shall need the following lemma. 
\begin{Lem}
	Let $f$ be the solution to \eqref{39}. Then the followings are valid. 
	
	(1) There exists a time-frequency interactive functional $\E^{(2)}$ such that 
	\begin{equation*}
		\E^{(2)} \approx \|\widehat{f}\|_{L^2_v}^2+|\widehat{E}|^2,
	\end{equation*}and for $t\ge 0$, $y\in\R^3$, 
	\begin{equation}\label{44a}
		\partial_t\E^{(2)}(t,y)+\frac{\lambda|y|^2}{1+|y|^2}(\|(\tilde{a}^{1/2})^w\widehat{f}\|^2_{L^2_{v}}+|E|^2)\le 0. 
	\end{equation}

(2)
There exists a time-frequency interactive functional $\E^{(2)}_l$ such that 
\begin{align}\label{444}
	\E^{(2)}_l \approx \|w^l\widehat{f}\|_{L^2_v}^2+|\widehat{E}|^2,
\end{align}and for $t\ge 0$, $y\in\R^3$, 
\begin{align}\label{45aa}
	\partial_t\E^{(2)}_l(t,y)+\frac{\lambda|y|^2}{1+|y|^2}(\|(\tilde{a}^{1/2})^ww^l\widehat{f}\|^2_{L^2_v}+|\widehat{E}|^2)\le 0. 
\end{align}
\end{Lem}
\begin{proof}
	Using the calculation from Lemma \ref{Lemma31} with $g_\pm=0$ therein, we apply the combination $\frac{\eqref{28}}{1+|y|^2}+\frac{\eqref{29}}{1+|y|^2}+\frac{|y|^2\eqref{30}}{1+|y|^2}$ to get 
	\begin{align}\label{44}
		\partial_t\E^{(2)}_{int}+\lambda\frac{|y|^2}{1+|y|^2}\big(|{\widehat{a_+}+\widehat{a_-}}|^2+|\widehat{b}|^2+|\widehat{c}|^2+|\widehat{a_+}-\widehat{a_-}|^2+|\widehat{E}|^2\big)
		&\lesssim \|(\tilde{a}^{1/2})^w(\I-\P)\widehat{f}\|_{L^2_v}^2.
	\end{align}
where
\begin{align*}
	\E^{(2)}_{int}&=\Bigg(\frac{\kappa_1}{1+|y|^2}\Re\sum^{3}_{j=1}\Big(iy_j\frac{\widehat{a_+}+\widehat{a_-}}{2}\Big|\widehat{b_j}\Big)+	\frac{1}{2(1+|y|^2)}\Re\sum^{3}_{j=1}\Big(iy_j\widehat{c}\Big|\Lambda_j((\I-\P)\widehat{f}\cdot(1,1))\Big)\\&\qquad+\frac{\kappa_2}{1+|y|^2}\Re\sum^{3}_{j,k=1}\Big(iy_j\widehat{b_k}+iy_k\widehat{b_j}\Big|\frac{1}{2}\Theta_{jk}((\I-\P)\widehat{f}\cdot(1,1))+2\widehat{c}\delta_{jk}\Big)\\
	&\qquad+
	\frac{1}{1+|y|^2}\Re\Big(iy(\widehat{a_+}-\widehat{a_-})\Big|\widehat{G}\Big)+\frac{\kappa_3|y|^2}{1+|y|^2}\Re\big(-\widehat{E}\big|\widehat{G}\big)\Bigg).
\end{align*}
In order to eliminate the right-hand term of \eqref{44} and obtain the $\|\widehat{f}\|_{L^2_v}^2$ on the left hand side, we take the Fourier transform of \eqref{39} over $x$ and take the inner product with $f_\pm$ over $\R^3_v$. Summing on $\pm$ and taking the real part, we have 
\begin{equation*}
	\frac{1}{2}\partial_t\|\widehat{f}\|^2_{L^2_v}+ \Re(\mu^{1/2}iv\cdot y\widehat{\phi},\widehat{f}_+-\widehat{f}_-)_{L^2_v} + \sum_{\pm}(L\widehat{f}_\pm,\widehat{f}_\pm)_{L^2_v} = 0.
\end{equation*}
Recall the definition \eqref{27aa} and using \eqref{21}$_1$, we have 
\begin{equation*}
	\frac{1}{2}\partial_t|\widehat{E}|^2=\Re(\frac{iy}{|y|^2}y\cdot\widehat{G}|y\widehat{\phi})
	=\Re(i\widehat{G}|y\widehat{\phi}) = \Re(\mu^{1/2}iv\cdot y\widehat{\phi},f_+-f_-)_{L^2_v}.
\end{equation*}
On the other hand, by Lemma \ref{lemmaL}, we have 
\begin{align*}
	\sum_{\pm}(L\widehat{f}_\pm,\widehat{f}_\pm)_{L^2_v}\ge \lambda\|(\tilde{a}^{1/2})^w(\I-\P)\widehat{f}\|^2_{L^2_{v}}.
\end{align*}
Thus, 
\begin{equation}\label{45a}
	\frac{1}{2}\partial_t\big(\|\widehat{f}\|^2_{L^2_v}+|\widehat{E}|^2)+\lambda\|(\tilde{a}^{1/2})^w(\I-\P)\widehat{f}\|^2_{L^2_{v}}\le 0.
\end{equation}
Taking the combination $\kappa\times\eqref{44}+\eqref{45a}$ with $\kappa<<1$, we have 
\begin{equation*}
	\partial_t\E^{(2)} + \frac{\lambda|y|^2}{1+|y|^2}(\|(\tilde{a}^{1/2})^w(\I-\P)f\|^2_{L^2_{v}}+|\widehat{E}|^2)\le 0. 
\end{equation*}
where $\E^{(2)} = \kappa\E^{(2)}_{int} + \|f\|^2_{L^2_v}+|\widehat{E}|^2$. It's direct to check that $\E^{(2)}
\approx \|f\|^2_{L^2_v}+|\widehat{E}|^2$ by using $\kappa<<1$.

In order to obtain \eqref{45aa}, we write \eqref{39}$_1$ to be 
\begin{align*}
	\partial_t(\II&-\PP)\widehat{f} + iv\cdot y(\II-\PP)\widehat{f} + L(\II-\PP)\widehat{f} \\
	&= \mp(\II-\PP)(\widehat{E}\cdot v)\mu^{1/2} - (\II-\PP)(iv\cdot y\P\widehat{f}) -\PP(iv\cdot y(\I-\P)\widehat{f}). 
\end{align*}
Taking inner product with $w^{2l}(\I-\P)\widehat{f}$ over $\R^3_v$ and using Lemma \ref{lemmaL}, one has 
\begin{align}\label{48aaa}
	\partial_t\|w^l(\I-\P)\widehat{f}\|^2_{L^2_v}+\lambda \|(\tilde{a}^{1/2})^ww^l(\I-\P)\widehat{f}\|^2_{L^2_v}\lesssim |\widehat{E}|^2+|y|^2\|(\tilde{a}^{1/2})^w\widehat{f}\|^2_{L^2_v}.
\end{align}
Using a similar argument without taking projection $\II-\PP$, one has 
\begin{align}\label{48aab}
	\partial_t\|w^l\widehat{f}\|^2_{L^2_v}+\lambda \|(\tilde{a}^{1/2})^ww^l\widehat{f}\|^2_{L^2_v}\lesssim |\widehat{E}|^2+\|\widehat{f}\|^2_{L^2(B_C)}.
\end{align}
Taking combination $\eqref{44a}+\kappa(\eqref{48aaa}\chi_{|y|\le 1}+\eqref{48aab}\chi_{|y|\ge 1})$ with $\kappa<<1$, one has 
\begin{align*}
	\partial_t\E^{(2)}_l(t,y)+\frac{\lambda|y|^2}{1+|y|^2}\big(\|(\tilde{a}^{1/2})^w\widehat{f}\|^2_{L^2_v}+|\widehat{E}|^2\big)\le 0,
\end{align*}
where 
\begin{align*}
	\E^{(2)}_l(t,y) = \E^{(2)}(t,y) + \kappa(\|w^l(\I-\P)\widehat{f}\|^2_{L^2_v}\chi_{|y|\le 1}+\|w^l\widehat{f}\|^2_{L^2_v}\chi_{|y|\ge 1}).
\end{align*}
It's direct to compute \eqref{444} since $\kappa$ is suitably small. 
\end{proof}

Now we are in a position to prove the large time behavior of the homogeneous system \eqref{39}. 
\begin{proof}[Proof of Theorem \ref{homogen}]
By noticing $\|\cdot\|_{L^2_v}\lesssim \|(\tilde{a}^{1/2})^w(\cdot)\|_{L^2_v}$ for hand potential, \eqref{45aa} gives that 
\begin{align*}
	\partial_t\E^{(2)}_l(t,y)+\frac{\lambda|y|^2}{1+|y|^2}\E^{(2)}_l(t,y)\le 0. 
\end{align*}
Then by solving this ODE, we have 
\begin{align*}
	\E^{(2)}_l(t,y)\lesssim e^{-\frac{\lambda|y|^2t}{1+|y|^2}}\E^{(2)}_l(0,y).
\end{align*}
By using \eqref{444}, 
\begin{align}\label{51}
	\|\nabla^m_xw^lf\|^2_{L^2_{v,x}}&+\|\nabla^m_xE\|^2_{L^2_x}\approx \int_{\R^3}|y|^{2m}\E^{(2)}_l(t,y)\,dy\\
	&\lesssim \int_{|y|\le 1}|y|^{2m}e^{-\lambda|y|^2t}\E^{(2)}_l(0,y)\,dy+e^{-\lambda t}\int_{|y|\ge 1}|y|^{2m}\E^{(2)}_l(0,y)\,dy.\notag
\end{align}
By H\"{o}lder's inequality and scaling on $y$, one has 
\begin{align*}
	\int_{|y|\le 1}|y|^{2m}e^{-\lambda|y|^2t}\E^{(2)}_l(0,y)\,dy
	&\lesssim \min\{1,t^{-3/2-m}\}\|\E^{(2)}_l(0,y)\|_{L^\infty_y}.
\end{align*}
For the case $|y|\ge 1$, noticing \eqref{8}, we have 
\begin{align*}
	\||y|^m\widehat{E_0}\|_{L^2_y(|y|\ge 1)}\lesssim \||y|^{m-1}\widehat{f_0}\|_{L^2_y(|y|\ge 1)}\lesssim \||y|^{m}\widehat{f_0}\|_{L^2_y(|y|\ge 1)},
\end{align*}
which yields that 
\begin{align*}
	e^{-\lambda t}\int_{|y|\ge 1}|y|^{2m}\E^{(2)}_l(0,y)\,dy\lesssim e^{-\lambda t}\|w^l\nabla^m_xf_0\|^2_{L^2_{v,x}}.
\end{align*}
Thus, \eqref{51} becomes 
\begin{align*}
	\|\nabla^m_xw^lf\|^2_{L^2_{v,x}}+\|\nabla^m_xE\|^2_{L^2_x}\lesssim (1+t)^{-3/2-m}\big(\|w^lf_0\|^2_{Z_1}+\|E_0\|^2_{L_1}+\|w^l\nabla^m_xf_0\|^2_{L^2_{v,x}})
\end{align*}
This completes the proof. 
\end{proof}

\section{Global Existence}\label{sec4}
In this section, we are going to prove the main Theorem \ref{main1}, the global-in-time existence of the solution to the following system. 
\begin{equation}\label{16?}
\left\{\begin{aligned}
	&\partial_tf_\pm + v_i\partial^{e_i}f_\pm \pm \frac{1}{2}\partial^{e_i}\phi v_if_\pm  \mp\partial^{e_i}\phi\partial_{e_i}f_\pm \pm \partial^{e_i}\phi  v_i\mu^{1/2} - L_\pm f = \Gamma_{\pm}(f,f),\\
	&-\Delta_x \phi = \int_{\Rd}(f_+-f_-)\mu^{1/2}\,dv, \quad \phi\to 0\text{ as }|x|\to\infty,\\
	&f_\pm|_{t=0} = f_{0,\pm}. 
\end{aligned}\right.
\end{equation}The index appearing in both superscript and subscript means the summation. Our goal is to obtain the $a$ $priori$ from this equation. 
For this, we suppose that the Cauchy problem \eqref{16?} admits a smooth solution $f(t,x,v)$ over $0\le t\le T$ for $0<T\le\infty$, and the solution $f(t,x,v)$ satisfies 
\begin{align}
	\sup_{0\le t\le T}\E_{K,l}(t)\le \delta_0,
\end{align} where $\delta_0$ is a suitably small constant. 
Under this assumption, we can derive a simple fact that 
\begin{align*}
	\|\phi\|_{L^\infty}\lesssim\|\phi\|_{H^2_x}\le \delta_0, \quad \|e^{\pm\phi}\|_{L^\infty}\approx 1.
\end{align*}
Also, by equation \eqref{21}$_1$, we have 
\begin{equation}\label{34a}
	\partial_t\phi = -\Delta_x^{-1}\partial_t(a_+-a_-)=\Delta_x^{-1}\nabla\cdot G,
\end{equation}
\begin{equation}\label{34}
	\|\partial_t\phi\|_{L^\infty}\lesssim  \|\nabla_x\partial_t\phi\|^{1/2}_{L^2_x}\|\nabla^2_x\partial_t\phi\|^{1/2}_{L^2_x}\lesssim \|\nabla_x G\|_{H^1}\lesssim \|(\I-\P)f\|_{L^2_vH^2_x} \lesssim (\E^h_{K,l})^{1/2}(t). 
\end{equation}

\begin{Thm}\label{thm41}Define $i=1$ if $0<s<\frac{1}{2}$ and $i=2$ if $\frac{1}{2}\le s<1$. For any $l\ge K\ge i+1$, there is $\E_{K,l}$ satisfying \eqref{Defe} such that for $0\le t\le T$,
\begin{align}\label{42a}
	\partial_t\E_{K,l}(t)+\lambda D_{K,l}(t)\lesssim \|\partial_t\phi\|_{L^\infty_x}\E_{K,l}(t), 
\end{align}where $D_{K,l}$ is defined by \eqref{Defd}. 
\end{Thm}
\begin{proof}
	For later use and brevity of the proof, we define a useful function
	$\psi = \psi(t)$ equal to $1$ in this section and equal to $t^{N}(0\le t\le 1)$ in the next section. In any case, we have 
	\begin{align*}
		0\le \psi\le 1.
	\end{align*}
	In this proof, we will carry the function $\psi$ for brevity of the proof in next section. 
	
	For any $K\ge i+1$ being the total derivative of $v,x$, we let $|\alpha|+|\beta|\le K$.
	On one hand, we apply $\partial^\alpha$ to equation \eqref{7}$_1$ to get 
	\begin{equation}\begin{aligned}\label{35}
		&\quad\,\partial_t\partial^{\alpha} f_\pm + v_i\partial^{e_i+\alpha} f_\pm \pm \frac{1}{2}\sum_{\substack{\alpha_1\le\alpha}}\partial^{e_i+\alpha_1}\phi v_i\partial^{\alpha-\alpha_1}f_\pm \\ &\qquad\mp\sum_{\substack{\alpha_1\le\alpha}}\partial^{e_i+\alpha_1}\phi\partial^{\alpha-\alpha_1}_{e_i} f_\pm \pm \partial^{e_i+\alpha}\phi v_i\mu^{1/2} - \partial^{\alpha} L_\pm  f =
		\partial^{\alpha} \Gamma_{\pm}(f,f).\end{aligned}
\end{equation}
On the other hand, we apply $\partial^{\alpha}_\beta$ to equation \eqref{7}$_1$ and decompose $f_\pm=\PP f+(\II-\PP)f$. Then, 
	\begin{align}\label{36}
		&\quad\,\partial_t\partial^{\alpha}_\beta (\II-\PP)f + \sum_{\beta_1\le \beta}C^{\beta_1}_{\beta}\partial_{\beta_1}v_i\partial^{e_i+\alpha}_{\beta-\beta_1}(\II-\PP)f \notag\\
		&\notag\qquad\pm \frac{1}{2}\sum_{\substack{\alpha_1\le\alpha}}\sum_{\beta_1\le\beta}\partial^{e_i+\alpha_1}\phi \partial_{\beta_1}v_i\partial^{\alpha-\alpha_1}_{\beta-\beta_1}(\II-\PP)f \\ &\qquad\mp\sum_{\substack{\alpha_1\le\alpha}}\partial^{e_i+\alpha_1}\phi\partial^{\alpha-\alpha_1}_{\beta+e_i}(\II-\PP)f \pm \partial^{e_i+\alpha}\phi \partial_\beta(v_i\mu^{1/2}) - \partial^{\alpha}_\beta L_\pm (\I-\P)f \\
		&\notag= -\partial_t\partial^{\alpha}_\beta \PP f + \sum_{\beta_1\le \beta}C^{\beta_1}_{\beta}\partial_{\beta_1}v_i\partial^{e_i+\alpha}_{\beta-\beta_1}\PP f \mp \frac{1}{2}\sum_{\substack{\alpha_1\le\alpha}}\sum_{\beta_1\le\beta}\partial^{e_i+\alpha_1}\phi \partial_{\beta_1}v_i\partial^{\alpha-\alpha_1}_{\beta-\beta_1}\PP f \\ &\notag\qquad\mp\sum_{\substack{\alpha_1\le\alpha}}\partial^{e_i+\alpha_1}\phi\partial^{\alpha-\alpha_1}_{\beta+e_i}\PP f
		+
		\partial^{\alpha}_\beta \Gamma_{\pm}(f,f).
	\end{align}

\paragraph{Step 1. Estimate without weight.}
For the estimate without weight, we take the case $|\alpha|\le K$ and $\beta=0$. This case is for obtaining the term $\|\partial^\alpha\nabla_x\phi\|^2_{L^2_x}$ on the left hand side of the energy inequality. Taking inner product of equation \eqref{35} with $\psi_{2|\alpha|-4}e^{\pm\phi}\partial^{\alpha} f_\pm$ over $\R^3_v\times\R^3_x$, we have   
\begin{align}\label{45}
	&\notag\quad\,\Big(\partial_t\partial^{\alpha} f_\pm,\psi_{2|\alpha|-4}e^{\pm\phi}\partial^{\alpha} f_\pm\Big)_{L^2_{v,x}}
	+ \Big(v_i\partial^{e_i+\alpha}f_\pm,\psi_{2|\alpha|-4}e^{\pm\phi}\partial^{\alpha} f_\pm\Big)_{L^2_{v,x}}\\ 
	&\notag\pm \Big(\frac{1}{2}\sum_{\substack{\alpha_1\le\alpha}}C^{\alpha_1}_{\alpha}\partial^{e_i+\alpha_1}\phi v_i\partial^{\alpha-\alpha_1}f_\pm,\psi_{2|\alpha|-4}e^{\pm\phi}\partial^{\alpha} f_\pm\Big)_{L^2_{v,x}} \\ 
	&\mp
	\Big(\sum_{\substack{\alpha_1\le\alpha}}C^{\alpha_1}_{\alpha}\partial^{e_i+\alpha_1}\phi\partial^{\alpha-\alpha_1}_{e_i}f_\pm,\psi_{2|\alpha|-4}e^{\pm\phi}\partial^{\alpha} f_\pm\Big) _{L^2_{v,x}}\\
	&\notag\pm \Big(\partial^{e_i+\alpha}\phi v_i\mu^{1/2},\psi_{2|\alpha|-4}e^{\pm\phi}\partial^{\alpha} f_\pm\Big)_{L^2_{v,x}} 
	- \Big(\partial^{\alpha} L_\pm f,\psi_{2|\alpha|-4}e^{\pm\phi}\partial^{\alpha} f_\pm\Big)_{L^2_{v,x}}\\ 
	&\notag= \Big(\partial^{\alpha} \Gamma_{\pm}(f,f),\psi_{2|\alpha|-4}e^{\pm\phi}\partial^{\alpha} f_\pm\Big)_{L^2_{v,x}}.
\end{align}
Now we denote these terms with summation $\sum_{\pm}$ by $I_1$ to $I_7$ and estimate them term by term. 

For the first term $I_1$ on the left hand side.  
\begin{align*}
	&\quad\,\frac{1}{2}\partial_t\|\psi_{|\alpha|-2}e^{\frac{\pm\phi}{2}}\partial^{\alpha} f_\pm\|^2_{L^2_{v,x}} \\&= \Re\big(\partial_t(\psi_{|\alpha|-2}e^{\frac{\pm\phi}{2}}\partial^{\alpha} f_\pm),\psi_{|\alpha|-2}e^{\frac{\pm\phi}{2}}\partial^{\alpha} f_\pm)_{L^2_{v,x}}\\
	&= \Re(\partial_t\partial^{\alpha} f_\pm\psi_{|\alpha|-2} \pm\frac{1}{2}\partial_t\phi\psi_{|\alpha|-2}\partial^{\alpha}f_\pm +\partial_t(\psi_{|\alpha|-2})\partial^{\alpha} f_\pm,\psi_{|\alpha|-2} e^{\pm\phi}\partial^{\alpha} f_\pm)_{L^2_{v,x}}.
\end{align*}Then,
\begin{align}\label{37}
	I_1 &= \frac{1}{2}\partial_t\sum_{\pm}\|e^{\frac{\pm\phi}{2}}\psi_{|\alpha|-2}\partial^{\alpha} f_\pm\|^2_{L^2_{v,x}} \mp \Re\sum_{\pm}\frac{1}{2}(\partial_t\phi e^{\pm\phi}\partial^{\alpha} f_\pm, \psi_{2|\alpha|-4}\partial^{\alpha} f_\pm)_{L^2_{v,x}}\\&\qquad-\Re\sum_{\pm}(\partial_t(\psi_{|\alpha|-2})\partial^{\alpha} f_\pm,\psi_{|\alpha|-2} e^{\pm\phi}\partial^{\alpha} f_\pm)_{L^2_{v,x}}. \notag
\end{align}When $\psi=1$, the third term on the right hand side is $0$. The second term on the right hand side of \eqref{37} is estimated as 
\begin{align}\label{46}
	\Big|\frac{1}{2}(\partial_t\phi \psi_{2|\alpha|-4}e^{\pm\phi}\partial^{\alpha} f_\pm, \partial^{\alpha} f_\pm)_{L^2_{v,x}}\Big|\lesssim \|\partial_t\phi\|_{L^\infty}\|\psi_{|\alpha|-2}\partial^\alpha f_\pm\|^2_{L^2_{v,x}}\lesssim \|\partial_t\phi\|_{L^\infty}\E_{K,l}(t)
\end{align}

For the second term $I_2$, we will compose it with $I_3$ with $\alpha_1=0$ in $I_3$. It turns out that the sum is zero. This is what $e^{\pm\phi}$ designed for, cf. \cite{Guo2012}. By taking integration by parts on $x$, one has  
\begin{align}\label{48aa}
	&\quad\,\Big(v_i\partial^{e_i+\alpha}f_\pm,\psi_{2|\alpha|-4}e^{\pm\phi}\partial^{\alpha} f_\pm\Big)_{L^2_{v,x}}
	\pm \Big(\frac{1}{2}\partial^{e_i}\phi v_i\partial^{\alpha}f_\pm,\psi_{2|\alpha|-4}e^{\pm\phi}\partial^{\alpha} f_\pm\Big)_{L^2_{v,x}}=0.
\end{align}

For the left terms in $I_3$, the weight will be used. In this case, $\alpha_1$ is not zero. (If $\alpha=0$, then there's already no left terms in $I_3$.) Then $|\alpha|\ge 1$ and the second $f_\pm$ in the following must have at least one order derivative. Notice that $\psi\le 1$. 
\begin{equation}\label{40}\begin{aligned}
	&\quad\, \Big|\pm \Big(\frac{1}{2}\sum_{\substack{\alpha_1\le\alpha}}C^{\alpha_1}_{\alpha}\partial^{e_i+\alpha_1}\phi v_i\partial^{\alpha-\alpha_1}f_\pm,\psi_{2|\alpha|-4}e^{\pm\phi}\partial^{\alpha} f_\pm\Big)_{L^2_{v,x}}\Big|\\
	&\lesssim \sum_{\substack{\alpha_1\le\alpha}}\Big|\Big(\psi_{|\alpha|-2}\partial^{e_i+\alpha_1}\phi w\partial^{\alpha-\alpha_1}f_\pm,\psi_{|\alpha|-2}e^{\pm\phi}\partial^{\alpha} f_\pm\Big)_{L^2_{v,x}}\Big|\\
	&\lesssim \Big(\sum_{|\alpha_1|\le K}\|\psi_{|\alpha_1|-2}\partial^{e_i+\alpha_1}\phi\|_{L^2_{v,x}}\Big)\Big(\sum_{1\le|\alpha|\le K-1}\|\psi_{|\alpha|-2}w^{l-|\alpha|}\partial^{\alpha}f_\pm e^{\frac{\pm\phi}{2}}\|_{L^2_{v,x}}\Big)\|\psi_{|\alpha|-2}\partial^{\alpha} f_\pm e^{\frac{\pm\phi}{2}}\|_{L^2_{v,x}}\\
	&\lesssim \E^{1/2}_{K,l}(t)\D_{K,l}(t).
\end{aligned}\end{equation}
Here we used \eqref{13} for $\phi$ and the first $f_\pm$. When the number of derivatives on $\phi$ and $f_\pm$ are both less than $K$, we used \eqref{13}$_2$ to give one order of derivative to them and the total number of derivatives are less or equal to $K$. When one of $\phi$ and $f_\pm$ has $K$ derivatives, then we use \eqref{13}$_1$ to give two derivatives to the other one. Then the total number of derivatives for them are still less or equal to $K$. The technique is the same as Lemma \ref{lemmat}. 

For the term $I_4$, when $\alpha_1=0$, by integration by parts on $\partial_{e_i}$, we have 
\begin{align}\label{50}
	I_4&=\sum_{\pm}\mp
	\Big(\partial^{e_i}\phi\partial^{\alpha}_{e_i}f_\pm,\psi_{2|\alpha|-4}e^{\pm\phi}\partial^{\alpha} f_\pm\Big)_{L^2_{v,x}}=0,
\end{align}
When $\alpha_1\neq 0$, then $|\alpha|\ge1$ and the total order of derivatives on the first $f_\pm$ is less or equal to $K$ and is controllable. 
\begin{align}\label{42}
	\notag|I_4|&=\Big|\sum_{\pm}\mp
	\Big(\sum_{\substack{0\neq\alpha_1\le\alpha}}C^{\alpha_1}_{\alpha}\partial^{e_i+\alpha_1}\phi\partial^{\alpha-\alpha_1}_{e_i}f_\pm,\psi_{2|\alpha|-4}e^{\pm\phi}\partial^{\alpha} f_\pm\Big) _{L^2_{v,x}}\Big|\\	
	&\lesssim \E^{1/2}_{K,l}(t)\D_{K,l}(t).
\end{align}
Here we used \eqref{13} for $\phi$ and the first $f_\pm$ as the followings. If $|\alpha_1|=1$, we use \eqref{13}$_1$ to give two derivatives to $\phi$ on $x$. If $|\alpha_1|=2$, we use \eqref{13}$_2$ to give one derivative to both $\phi$ and the first $f_\pm$ on $x$. If $K\ge|\alpha|\ge 3$, then we use \eqref{13}$_1$ to give two $x$ derivatives to the first $f_\pm$. The idea is similar to the proof of Lemma \ref{lemmat}. We also used that for $m\ge 2$, 
\begin{align*}
	\|\psi_{m-2}\nabla_x^{m+1}\phi\|_{L^2_x}\lesssim \|\psi_{m-2}\nabla_x^{m+1}\nabla_x^{-1}(a_+-a_-)\|^2_{L^2_x}\lesssim \|\psi_{m-3}\nabla_x^{m-1}(a_+,a_-)\|^2_{L^2_x}\lesssim \D_{K,l},
\end{align*}which follows from \eqref{16?}$_2$.

For the term $I_5$, we will divide $e^{\pm\phi}$ into $(e^{\pm\phi}-1)$ and $1$. Recall equation \eqref{16} and \eqref{21}. For the part of $1$, 
\begin{align}\label{43}\notag
	\sum_{\pm}\pm\Big(\partial^{e_i+\alpha}\phi v_i\mu^{1/2},\psi_{2|\alpha|-4}\partial^{\alpha} f_\pm\Big)_{L^2_{v,x}} 
    &=\notag -\Big(\partial^{\alpha}\phi,\psi_{2|\alpha|-4}\partial^{\alpha} \nabla_x\cdot G\Big)_{L^2_{x}}\\
    &=\notag \Big(\partial^{\alpha}\phi,\psi_{2|\alpha|-4}\partial^{\alpha} \partial_t(a_+-a_-)\Big)_{L^2_{x}}\\
    &= \frac{1}{2}\partial_t\|\psi_{|\alpha|-2}\partial^{\alpha}\nabla_x\phi\|_{L^2_x}^2.
\end{align}
For the part of $(e^{\pm\phi}-1)$, notice that 
\begin{align*}
	|e^{\pm\phi}-1|\lesssim \|\phi\|_{L^\infty}\lesssim \|\nabla_x\phi\|_{H^1_x}.
\end{align*}Then, 
\begin{align}
	\Big|&\sum_{\pm}\pm\Big(\partial^{e_i+\alpha}\phi v_i\mu^{1/2},(e^{\pm\phi}-1)\psi_{2|\alpha|-4}\partial^{\alpha} f_\pm\Big)_{L^2_{v,x}}\Big|\notag\\
	&\lesssim \|\nabla_x\phi\|_{H^1_x}\sum_{|\alpha|\le K}\|\partial^\alpha\nabla_x\phi\|_{L^2_{v,x}}\sum_{|\alpha|\le K}\|\partial^\alpha(\II-\PP)f\|_{L^2_{v,x}}\\
	&\lesssim \E^{1/2}_{K,l}(t)\D_{K,l}(t).\notag
\end{align}

For the term $I_6$, since $L_\pm$ commutes with $\partial^{\alpha}$ and $e^{\pm\phi}$, by Lemma \ref{lemmaL}, we have 
\begin{align}
	I_6 = - \sum_{\pm}\Big(\partial^{\alpha} L_\pm f,\psi_{2|\alpha|-4}e^{\pm\phi}\partial^{\alpha} f_\pm\Big)_{L^2_{v,x}}\ge \lambda \sum_{\pm}\|\psi_{|\alpha|-2}e^{\frac{\pm\phi}{2}}(\tilde{a}^{1/2})^w\partial^{\alpha}(\II-\PP) f\|_{L^2_{v,x}}^2. 
\end{align}

For the term $I_7$, by Lemma \ref{lemmag}, we have 
\begin{align}
	|I_7|&= \Big|\sum_{\pm}\Big(\partial^{\alpha} \Gamma_{\pm}(f,f),\psi_{2|\alpha|-4}e^{\pm\phi}\partial^{\alpha} f_\pm\Big)_{L^2_{v,x}}\Big|\lesssim\E^{1/2}_{K,l}(t)\D_{K,l}(t).
\end{align}

Therefore, combining all the estimate above and take the summation on $\pm$, $|\alpha|\le K$, noticing that $|e^{\frac{\pm\phi}{2}}|\approx 1$, we conclude that, when $\psi=1$, 
\begin{equation}\label{47}
	\begin{aligned}
		&\quad\,\frac{1}{2}\partial_t\sum_{\pm}\sum_{|\alpha|\le K}\Big(\|\psi_{|\alpha|-2}e^{\frac{\pm\phi}{2}}\partial^{\alpha} f_\pm\|_{L^2_{v,x}} +
		\|\psi_{|\alpha|-2}\partial^{\alpha}\nabla_x\phi\|_{L^2_x}^2\Big)\\&\qquad + \lambda \sum_{\pm}\sum_{|\alpha|\le K}\|\psi_{|\alpha|-2}e^{\frac{\pm\phi}{2}}(\tilde{a}^{1/2})^w\partial^{\alpha} (\II-\PP)f\|_{L^2_{v,x}}^2\\
 &\lesssim \|\partial_t\phi\|_{L^\infty}\E_{K,l}(t)+\E^{1/2}_{K,l}(t)\D_{K,l}(t).
	\end{aligned}
\end{equation}
Taking the combination $\eqref{47}+\kappa\times\eqref{24}$ with $0<\kappa<<1$, we have that when $\psi=1$, 
\begin{align}\label{48a}\notag
	&\quad\,\frac{1}{2}\partial_t\sum_{\pm}\sum_{|\alpha|\le K}\Big(\|\psi_{|\alpha|-2}e^{\frac{\pm\phi}{2}}\partial^{\alpha} f_\pm\|_{L^2_{v,x}} +
	\|\psi_{|\alpha|-2}\partial^{\alpha}\nabla_x\phi\|_{L^2_x}^2+\kappa\E^{(1)}_{K}\Big)\\\notag&\qquad + \lambda \sum_{\pm}\sum_{|\alpha|\le K}\|\psi_{|\alpha|-2}e^{\frac{\pm\phi}{2}}(\tilde{a}^{1/2})^w\partial^{\alpha} (\II-\PP)f\|_{L^2_{v,x}}^2\\
	&\notag\qquad + \lambda\sum_{|\alpha|\le K-1} \|\partial^\alpha\nabla_x(a_\pm,b,c)\|^2_{L^2_x}+\lambda\|a_+-a_-\|^2_{L^2_x}+\lambda\sum_{|\alpha|\le K-1}\|\partial^\alpha E\|^2_{L^2_x}\\
	&\lesssim \|\partial_t\phi\|_{L^\infty}\E_{K,l}(t)+ (\E^{1/2}_{K,l}(t)+\E_{K,l}(t))\D_{K,l}(t)
\end{align}
The term $\|(\tilde{a}^{1/2})\psi_{|\alpha|-2}(\I-\P)\partial^{\alpha}\widehat{f}\|_{L^2_{v,x}}^2$ in \eqref{24} is eliminated. 

\paragraph{Step 2. Estimate with weight on $x$ derivatives}
This case is particularly for $|\alpha|=K$. Let $1\le |\alpha|\le K$ and take inner product of \eqref{35} with $\psi_{2|\alpha|-4}e^{\pm\phi}w^{2l-2|\alpha|}\partial^\alpha f_\pm$ over $\R^3_v\times\R^3_x$.
\begin{align}\label{77}
	&\notag\quad\,\Big(\partial_t\partial^{\alpha} f_\pm,\psi_{2|\alpha|-4}w^{2l-2|\alpha|}e^{\pm\phi}\partial^{\alpha} f_\pm\Big)_{L^2_{v,x}}
	+ \Big(v_i\partial^{e_i+\alpha}f_\pm,\psi_{2|\alpha|-4}w^{2l-2|\alpha|}e^{\pm\phi}\partial^{\alpha} f_\pm\Big)_{L^2_{v,x}}\\ 
	&\notag\pm \Big(\frac{1}{2}\sum_{\substack{\alpha_1\le\alpha}}C^{\alpha_1}_{\alpha}\partial^{e_i+\alpha_1}\phi v_i\partial^{\alpha-\alpha_1}f_\pm,\psi_{2|\alpha|-4}w^{2l-2|\alpha|}e^{\pm\phi}\partial^{\alpha} f_\pm\Big)_{L^2_{v,x}} \\ 
	&\mp
	\Big(\sum_{\substack{\alpha_1\le\alpha}}C^{\alpha_1}_{\alpha}\partial^{e_i+\alpha_1}\phi\partial^{\alpha-\alpha_1}_{e_i}f_\pm,\psi_{2|\alpha|-4}w^{2l-2|\alpha|}e^{\pm\phi}\partial^{\alpha} f_\pm\Big) _{L^2_{v,x}}\\
	&\notag\pm \Big(\partial^{e_i+\alpha}\phi v_i\mu^{1/2},\psi_{2|\alpha|-4}w^{2l-2|\alpha|}e^{\pm\phi}\partial^{\alpha} f_\pm\Big)_{L^2_{v,x}} 
	- \Big(\partial^{\alpha} L_\pm f,\psi_{2|\alpha|-4}w^{2l-2|\alpha|}e^{\pm\phi}\partial^{\alpha} f_\pm\Big)_{L^2_{v,x}}\\ 
	&\notag= \Big(\partial^{\alpha} \Gamma_{\pm}(f,f),\psi_{2|\alpha|-4}w^{2l-2|\alpha|}e^{\pm\phi}\partial^{\alpha} f_\pm\Big)_{L^2_{v,x}}.
\end{align}
As in the Step 1, taking summation on $\pm$, we estimate it term by term. The proof is similar to $I_1$ to $I_7$. The first term on the left hand is 
\begin{align}\label{99c}\notag &\quad\,\frac{1}{2}\partial_t\sum_{\pm}\|e^{\frac{\pm\phi}{2}}w^{l-|\alpha|}\psi_{|\alpha|-2}\partial^{\alpha} f_\pm\|_{L^2_{v,x}} \mp \Re\sum_{\pm}\frac{1}{2}(\partial_t\phi e^{\pm\phi}\partial^{\alpha} f_\pm, \psi_{2|\alpha|-4}w^{2l-2|\alpha|}\partial^{\alpha} f_\pm)_{L^2_{v,x}}\\&\qquad-\Re\sum_{\pm}(\partial_t(\psi_{|\alpha|-2})\partial^{\alpha} f_\pm,\psi_{|\alpha|-2} e^{\pm\phi}w^{2l-2|\alpha|}\partial^{\alpha} f_\pm)_{L^2_{v,x}}. 
\end{align}
The second term and the third term with $\alpha_1=0$ are canceled by using integration by parts. The left case $\alpha_1\neq 0$ in the third term is bounded above by $\E^{1/2}_{K,l}(t)\D_{K,l}(t)$. 
For the fourth term when $\alpha_1=0$, by integration by parts on $\partial_{e_i}$, we have 
\begin{align*}
	\sum_{\pm}\mp
	\Big(\partial^{e_i}\phi\partial^{\alpha}_{e_i}f_\pm,\psi_{2|\alpha|-4}w^{2l-2|\alpha|}e^{\pm\phi}\partial^{\alpha} f_\pm\Big)_{L^2_{v,x}}\lesssim \E^{1/2}_{K,l}(t)\D_{K,l}(t).
\end{align*}
When $\alpha_1\neq 0$, then $|\alpha|\ge1$ and the total order of derivatives on the first $f_\pm$ is less or equal to $K$ and the fourth term is bounded above by $\E^{1/2}_{K,l}(t)\D_{K,l}(t)$. 
For the fifth term, we write a upper bound: for any $\eta>0$,
\begin{align*}
	\Big|\Big(\partial^{e_i+\alpha}\phi v_i\mu^{1/2},\psi_{2|\alpha|-4}w^{2l-2|\alpha|}e^{\pm\phi}\partial^{\alpha} f_\pm\Big)_{L^2_{v,x}}\Big| &\lesssim \eta\|\psi_{|\alpha|-2}(\tilde{a}^{1/2})^ww^{l-|\alpha|}\partial^{\alpha} f_\pm\|^2_{L^2_{v,x}}+C_\eta\|\partial^\alpha\nabla_x\phi\|^2_{L^2_x}.
\end{align*}
For the sixth term, since $L_\pm$ commutes with $\partial^{\alpha}$ and $e^{\pm\phi}$, by Lemma \ref{lemmaL}, we have 
\begin{align*}
	&- \sum_{\pm}\Big(\partial^{\alpha}w^{2l-2|\alpha|} L_\pm f,\psi_{2|\alpha|-4}e^{\pm\phi}\partial^{\alpha} f_\pm\Big)_{L^2_{v,x}}\\&\ge \lambda \|\psi_{|\alpha|-2}e^{\frac{\pm\phi}{2}}(\tilde{a}^{1/2})^ww^{l-|\alpha|}\partial^{\alpha} f\|_{L^2_{v,x}}^2-C\|\psi_{|\alpha|-2}(\tilde{a}^{1/2})^w\partial^\alpha f\|^2_{L^2_{v,x}}. 
\end{align*}
By using Lemma \ref{lemmag}, the first term on the right hand of \eqref{77} is bounded above by $\E^{1/2}_{K,l}(t)\D_{K,l}(t)$.
Taking $\psi=1$, combining the above estimate, taking summation on $1\le |\alpha|\le K$ and letting $\eta$ suitably small, we have 
\begin{align}\label{74}
	&\quad\,\notag\frac{1}{2}\partial_t\sum_{\pm}\sum_{1\le|\alpha|\le K}\|e^{\frac{\pm\phi}{2}}\psi_{|\alpha|-2}w^{l-|\alpha|}\partial^{\alpha} f_\pm\|_{L^2_{v,x}} + \lambda\sum_{\pm}\sum_{1\le|\alpha|\le K}\|\psi_{|\alpha|-2}e^{\frac{\pm\phi}{2}}(\tilde{a}^{1/2})^ww^{l-|\alpha|}\partial^{\alpha} f_\pm\|_{L^2_{v,x}}^2\\
	&\lesssim \|\partial_t\phi\|_{L^2_{v,x}}\E_{K,l}(t)+\sum_{|\alpha|\le K}\|\partial^\alpha\nabla_x\phi\|^2_{L^2_x}+\sum_{1\le|\alpha|\le K}\|\psi_{|\alpha|-2}(\tilde{a}^{1/2})^w\partial^\alpha f_\pm\|^2_{L^2_{v,x}} + \E^{1/2}_{K,l}(t)\D_{K,l}(t).
\end{align}

\paragraph{Step 3. Estimate with weight on the mixed derivatives.}

Let $K\ge i+1$ with $i=1$ if $0<s<\frac{1}{2}$ and $i=2$ if $\frac{1}{2}\le s<1$.  $|\alpha|\le K-1$ and $|\alpha|+|\beta|\le K$. Taking inner product of equation \eqref{36} with  $\psi_{2|\alpha|+2|\beta|-4}e^{\pm\phi}w^{2l-2|\alpha|-2|\beta|}\partial^{\alpha}_\beta (\II-\PP)f$ over $\R^3_v\times\R^3_x$, one has 
\begin{align*}
		&\quad\,\Big(\partial_t\partial^{\alpha}_\beta (\II-\PP)f,e^{\pm\phi}\psi_{2|\alpha|+2|\beta|-4}w^{2l-2|\alpha|-2|\beta|}\partial^{\alpha}_\beta (\II-\PP)f\Big)_{L^2_{v,x}}\\
		&\qquad + \Big(\sum_{\beta_1\le \beta}C^{\beta_1}_{\beta}\partial_{\beta_1}v_i\partial^{e_i+\alpha}_{\beta-\beta_1}(\II-\PP)f,e^{\pm\phi}\psi_{2|\alpha|+2|\beta|-4}w^{2l-2|\alpha|-2|\beta|}\partial^{\alpha}_\beta (\II-\PP)f\Big)_{L^2_{v,x}} \\
		&\qquad\pm \Big(\frac{1}{2}\sum_{\substack{\alpha_1\le\alpha\\\beta_1\le\beta}}\partial^{e_i+\alpha_1}\phi \partial_{\beta_1}v_i\partial^{\alpha-\alpha_1}_{\beta-\beta_1}(\II-\PP)f,e^{\pm\phi}\psi_{2|\alpha|+2|\beta|-4}w^{2l-2|\alpha|-2|\beta|}\partial^{\alpha}_\beta (\II-\PP)f\Big)_{L^2_{v,x}} \\ &\qquad\mp\Big(\sum_{\substack{\alpha_1\le\alpha}}\partial^{e_i+\alpha_1}\phi\partial^{\alpha-\alpha_1}_{\beta+e_i}(\II-\PP)f,e^{\pm\phi}\psi_{2|\alpha|+2|\beta|-4}w^{2l-2|\alpha|-2|\beta|}\partial^{\alpha}_\beta (\II-\PP)f\Big)_{L^2_{v,x}}\\
		&\qquad \pm \Big(\partial^{e_i+\alpha}\phi \partial_\beta(v_i\mu^{1/2}),e^{\pm\phi}\psi_{2|\alpha|+2|\beta|-4}w^{2l-2|\alpha|-2|\beta|}\partial^{\alpha}_\beta (\II-\PP)f\Big)_{L^2_{v,x}}\\
		&\qquad - \Big(\partial^{\alpha}_\beta L_\pm (\I-\P)f,e^{\pm\phi}\psi_{2|\alpha|+2|\beta|-4}w^{2l-2|\alpha|-2|\beta|}\partial^{\alpha}_\beta (\II-\PP)f\Big)_{L^2_{v,x}} \\
		&= -\Big(\partial_t\partial^{\alpha}_\beta \PP f,e^{\pm\phi}\psi_{2|\alpha|+2|\beta|-4}w^{2l-2|\alpha|-2|\beta|}\partial^{\alpha}_\beta (\II-\PP)f\Big)_{L^2_{v,x}}\\
		&\qquad + \Big(\sum_{\beta_1\le \beta}C^{\beta_1}_{\beta}\partial_{\beta_1}v_i\partial^{e_i+\alpha}_{\beta-\beta_1}\PP f,e^{\pm\phi}\psi_{2|\alpha|+2|\beta|-4}w^{2l-2|\alpha|-2|\beta|}\partial^{\alpha}_\beta (\II-\PP)f\Big)_{L^2_{v,x}} \\
		&\qquad\mp \Big(\frac{1}{2}\sum_{\substack{\alpha_1\le\alpha}}\sum_{\beta_1\le\beta}\partial^{e_i+\alpha_1}\phi \partial_{\beta_1}v_i\partial^{\alpha-\alpha_1}_{\beta-\beta_1}\PP f,e^{\pm\phi}\psi_{2|\alpha|+2|\beta|-4}w^{2l-2|\alpha|-2|\beta|}\partial^{\alpha}_\beta (\II-\PP)f\Big)_{L^2_{v,x}} \\ &\qquad\mp\Big(\sum_{\substack{\alpha_1\le\alpha}}\partial^{e_i+\alpha_1}\phi\partial^{\alpha-\alpha_1}_{\beta+e_i}\PP f,e^{\pm\phi}\psi_{2|\alpha|+2|\beta|-4}w^{2l-2|\alpha|-2|\beta|}\partial^{\alpha}_\beta (\II-\PP)f\Big)_{L^2_{v,x}}\\&\qquad
		+
		\Big(\partial^{\alpha}_\beta \Gamma_{\pm}(f,f),e^{\pm\phi}\psi_{2|\alpha|+2|\beta|-4}w^{2l-2|\alpha|-2|\beta|}\partial^{\alpha}_\beta (\II-\PP)f\Big)_{L^2_{v,x}}.
\end{align*}
Now we denote these terms with summation $\sum_{\pm}$ by $J_1$ to $J_{11}$ and estimate them term by term. 
The estimate of $J_1$ to $J_3$ are similar to $I_1$ to $I_3$. 
That is 
\begin{align*}
	J_1 &\ge \partial_t\sum_{\pm}\|e^{\frac{\pm\phi}{2}}\psi_{|\alpha|+|\beta|-2}w^{l-|\alpha|-|\beta|}\partial^{\alpha}_\beta (\II-\PP)f\|_{L^2_{v,x}} - C\|\partial_t\phi\|_{L^\infty}\E^h_{K,l}(t)\\& -\sum_{\pm}\big|(\partial_t(\psi_{|\alpha|+|\beta|-2})\partial^{\alpha}_\beta (\II-\PP)f,\psi_{|\alpha|+|\beta|-2}\notag w^{2l-2|\alpha|-2|\beta|}e^{\pm\phi}\partial^{\alpha}_\beta (\II-\PP)f)_{L^2_{v,x}}\big|  
\end{align*}
\begin{equation*}
	|J_2 + J_3|\lesssim  \E^{1/2}_{K,l}(t)\D_{K,l}(t).
\end{equation*}

For the term $J_4$, when $\alpha_1=0$, by integration by parts on $\partial_{e_i}$, we have 
\begin{align*}
	&\quad\,\big|\sum_{\pm}\mp
	(\partial^{e_i}\phi\partial^{\alpha}_{\beta+e_i}(\II-\PP)f,e^{\pm\phi}\psi_{2|\alpha|+2|\beta|-4}w^{2l-2|\alpha|-2|\beta|}\partial^{\alpha}_\beta (\II-\PP)f)_{L^2_{v,x}}\big|\\
	&=\frac{1}{2}|\sum_{\pm}(\partial^{e_i}\phi\partial_{e_i}(w^{2l-2|\alpha|-2|\beta|})\partial^{\alpha}_\beta(\II-\PP)f,e^{\pm\phi}\psi_{2|\alpha|+2|\beta|-4}\partial^{\alpha}_\beta (\II-\PP)f)_{L^2_{v,x}}|\\
	&\lesssim \sum_{\pm}\||\nabla_x\phi|\psi_{|\alpha|+|\beta|-2}w^{l-|\alpha|-|\beta|}\partial^{\alpha}_\beta(\II-\PP)f\|_{L^2_{v,x}}\|\psi_{|\alpha|+|\beta|-2}w^{l-|\alpha|-|\beta|}\partial^\alpha (\II-\PP)f\|_{L^2_{v,x}}\\
	&\lesssim \E^{1/2}_{K,l}(t)\D_{K,l}(t),
\end{align*}
by using \eqref{13}. 
If $\alpha_1\neq 0$, then $\alpha\neq 0$ and there's at least one derivative on $x$. The order of total derivatives on the first $(\II-\PP)f$ is less or equal to $K$ and its order of derivatives on $v$ is less or equal to $K-1$, and hence $J_4$ is bounded above by $\E^{1/2}_{K,l}(t)\D_{K,l}(t)$. 
For the term $J_5$, we only need to have a upper bound.
\begin{align*}\notag
	|J_5| &= 
	\Big|\sum_{\pm}\pm \Big(\partial^{e_i+\alpha}\phi \partial_\beta(v_i\mu^{1/2}),\psi_{2|\alpha|+2|\beta|-4}e^{\pm\phi}w^{2l-2|\alpha|-2|\beta|}\partial^{\alpha}_\beta (\II-\PP)f\Big)_{L^2_{v,x}}\Big|\\
	&\lesssim\sum_{|\alpha|\le K} \|\psi_{|\alpha|-2}\partial^\alpha\nabla_x\phi\|_{L^2_{v,x}}\sum_{\substack{|\alpha|+|\beta|\le K}}\|\psi_{|\alpha|+|\beta|-2}(\tilde{a}^{1/2})^ww^{l-|\alpha|-|\beta|}\partial^{\alpha}_{\beta}(\I-\P)f\|_{L^2_{v,x}}\\
	&\lesssim \eta\sum_{\substack{|\alpha|+|\beta|\le K}}\|\psi_{|\alpha|+|\beta|-2}(\tilde{a}^{1/2})^ww^{l-|\alpha|-|\beta|}\partial^{\alpha}_{\beta}(\I-\P)f\|^2_{L^2_{v,x}}+C_\eta\sum_{|\alpha|\le K} \|\psi_{|\alpha|-2}\partial^\alpha\nabla_x\phi\|_{L^2_{v,x}}^2.
\end{align*}

For the term $J_6$, since $L_\pm$ commutes with $e^{\pm\phi}$, by Lemma \ref{lemmaL}, we have 
\begin{align*}
	J_6 &= - \sum_{\pm}\Big(\partial^{\alpha}_\beta L_\pm(\I-\P) f,\psi_{2|\alpha|+2|\beta|-4}e^{\pm\phi}w^{2l-2|\alpha|-2|\beta|}\partial^{\alpha}_\beta (\II-\PP)f\Big)_{L^2_{v,x}}\\&\ge \lambda\sum_{\pm} \|\psi_{|\alpha|+|\beta|-2}e^{\frac{\pm\phi}{2}}(\tilde{a}^{1/2})^ww^{l-|\alpha|-|\beta|}\partial^{\alpha}_\beta (\II-\PP)f\|_{L^2_{v,x}}^2-C_\eta\sum_{\pm}\|(\tilde{a}^{1/2})^w\partial^\alpha(\II-\PP)f\|^2_{L^2_v}\\
	&\qquad -\eta\sum_{\pm}\sum_{|\beta_1|\le|\beta|}\|\psi_{|\alpha|+|\beta|-2}e^{\frac{\pm\phi}{2}}(\tilde{a}^{1/2})^ww^{l-|\alpha|-|\beta_1|}\partial^{\alpha}_{\beta_1}(\II-\PP)f\|^2_{L^2_v},
\end{align*}for any $\eta>0$. 
Here we use the fact that $\|w^{l-|\alpha|-|\beta|}(\cdot)\|_{L^2(B_{C_\eta})}\lesssim \|(\tilde{a}^{1/2})^w(\cdot)\|_{L^2}$. 
For $J_7$, $J_8$, using the exponential decay in $v$ and the conservation laws \eqref{17}, we have 
\begin{align*}
	J_7+J_8 &\lesssim \eta\sum_{\pm}\|\psi_{|\alpha|+|\beta|-2}(\tilde{a}^{1/2})^ww^{l-|\alpha|-|\beta|}\partial^\alpha_\beta(\II-\PP)f\|_{L^2_{v,x}}+C_\eta\Big(\sum_{|\alpha|\le K}\|\partial^\alpha\nabla_x\phi\|^2_{L^2_x}\\&\qquad+\sum_{|\alpha|\le K-1}\|\partial^\alpha\nabla_x(a_\pm,b,c)\|^2_{L^2_x}+\sum_{|\alpha|\le K}\|\psi_{|\alpha|-2}(\tilde{a}^{1/2})^w\partial^{\alpha}(\I-\P)f\|^2_{L^2_{v,x}}+\E_{K,l}(t)\D_{K,l}(t)\Big).
\end{align*}
Notice that here we used $|\alpha|\le K-1$. 
Similar to $I_3$, the term $J_9$ can be controlled by using \eqref{13}. $J_{10}$ is similar, since there's exponential decay in $v$. Then one can derive 
\begin{align*}
	|J_9+J_{10}|\lesssim \E^{1/2}_{K,l}(t)\D_{K,l}(t)
\end{align*} 
For the term $J_{11}$, by Lemma \ref{lemmag}, we have 
\begin{align*}
	|J_{11}|&= \Big|\sum_{\pm}\Big(\partial^{\alpha}_\beta \Gamma_{\pm}(f,f),\psi_{2|\alpha|+2|\beta|-4}w^{2l-2|\alpha|-2|\beta|}e^{\pm\phi}\partial^{\alpha}_\beta (\II-\PP)f\Big)_{L^2_{v,x}}\Big|\lesssim\E^{1/2}_{K,l}\D_{K,l}.
\end{align*}

Therefore, combining all the estimate above and take the summation on $|\alpha|\le K-1$, $|\alpha|+|\beta|\le K$, noticing that $|e^{\frac{\pm\phi}{2}}|\approx 1$, and letting $\eta$ sufficiently small, we conclude that, when $\psi=1$, 
 we conclude that when $\psi=1$, 
\begin{align}\label{48}\notag
	&\quad\,\partial_t\sum_{\pm}\sum_{\substack{|\alpha|\le K-1\\|\alpha|+|\beta|\le K}}\|e^{\frac{\pm\phi}{2}}\psi_{|\alpha|+|\beta|-2}w^{l-|\alpha|-|\beta|}\partial^{\alpha}_\beta(\II-\PP)f\|^2_{L^2_{v,x}}\\
	&\qquad
	+\lambda \sum_{\pm}\sum_{\substack{|\alpha|\le K-1\\|\alpha|+|\beta|\le K}}\|\psi_{|\alpha|+|\beta|-2}e^{\frac{\pm\phi}{2}}(\tilde{a}^{1/2})^ww^{l-|\alpha|-|\beta|}\partial^{\alpha}_\beta (\I-\P)f\|_{L^2_{v,x}}^2\notag\\
	&\lesssim \sum_{|\alpha|\le K-1} \|\partial^\alpha\nabla_x\phi\|_{L^2_{x}}^2 +(\E^{1/2}_{K,l}(t)+\E_{K,l}(t))\D_{K,l}(t)+\|\partial_t\phi\|_{L^\infty}\E^h_{K,l}(t)\\
	&\qquad+\sum_{|\alpha|\le K-1}\|\partial^\alpha\nabla_x(a_\pm,b,c)\|^2_{L^2_x}+\sum_{|\alpha|\le K}\|\psi_{|\alpha|-2}(\tilde{a}^{1/2})^w\partial^{\alpha}(\I-\P)f\|^2_{L^2_{v,x}}.\notag
\end{align}
The redundant terms on the right hand side will be eliminated by using \eqref{48a}.

\paragraph{Step 5.}

We are able to prove this theorem by taking the proper linear combination of those estimates obtained in the above steps as follows. Taking combination $C_1\times\eqref{48a}+\eqref{74}+\eqref{48}$ with sufficiently large $C_1>0$, we have 
\begin{align}\label{58}
	&\notag\quad\,\partial_t\E_{K,l}(t)
\notag + C_1\lambda \sum_{\pm}\sum_{|\alpha|\le K}\|\psi_{|\alpha|-2}e^{\frac{\pm\phi}{2}}(\tilde{a}^{1/2})^w\partial^{\alpha} f_\pm\|_{L^2_{v,x}}^2\\
&\notag\qquad + C_1\lambda\sum_{|\alpha|\le K-1} \|\partial^\alpha\nabla_x(a_\pm,b,c)\|^2_{L^2_x}+C_1\lambda\|a_+-a_-\|^2_{L^2_x}+C_1\lambda\sum_{|\alpha|\le K-1}\|\partial^\alpha E\|^2_{L^2_x}\\
&\qquad
+\lambda \sum_{\pm}\sum_{\substack{|\alpha|+|\beta|\le K}}\|\psi_{|\alpha|+|\beta|-2}e^{\frac{\pm\phi}{2}}(\tilde{a}^{1/2})^ww^{l-|\alpha|-|\beta|}\partial^{\alpha}_\beta (\II-\PP)f\|_{L^2_{v,x}}^2\notag\\
&\lesssim \|\partial_t\phi\|_{L^\infty}\E_{K,l}(t)+ (\E^{1/2}_{K,l}(t)+\E_{K,l}(t))\D_{K,l}(t). 
\end{align}
where 
\begin{align}
	\E_{K,l}(t)&=\notag \frac{C_1}{2}\sum_{\pm}\sum_{|\alpha|\le K}\|\psi_{|\alpha|-2}e^{\frac{\pm\phi}{2}}\partial^{\alpha} f_\pm\|_{L^2_{v,x}} +\sum_{\pm}\sum_{\substack{|\alpha|+|\beta|\le K}}\|e^{\frac{\pm\phi}{2}}w^{l-|\alpha|-|\beta|}\partial^{\alpha}_\beta(\II-\PP)f\|^2_{L^2_{v,x}}\\&\qquad+
	C_1\sum_{|\alpha|\le K}\|\psi_{|\alpha|-2}\partial^{\alpha}\nabla_x\phi\|_{L^2_x}^2+\kappa\E^{(1)}_{K}.
\end{align}The second to fourth terms on the left hand side of \eqref{58} is larger than $D_{K,l}$.
Notice that here for the term $\sum_{|\alpha|= K}\|\partial^\alpha\nabla_x\phi\|$, we use the fact that 
\begin{align}\label{78}
	\|\psi_{K-2}\nabla^{K+1}_x\phi\|^2_{L^2_x}&\lesssim \|\psi_{K-2}\nabla^{K+1}_x\Delta^{-1}_x(a_+-a_-)\|^2_{L^2_x}
	\lesssim \|\psi_{K-3}\nabla_x^{K-1}(a_+,a_-)\|^2_{L^2_x},
\end{align}and hence can be eliminated by using $C_1\times\eqref{48a}$. 
Noticing \eqref{27a} and $\kappa<<1$, it's direct to see that 
\begin{align*}
	\E_{K,l}(t)\notag &\approx \sum_{|\alpha|\le K}\|\psi_{|\alpha|-2}\partial^\alpha E(t)\|^2_{L^2_x}+\sum_{|\alpha|\le K}\|\psi_{|\alpha|-2}\partial^\alpha\P f\|^2_{L^2_{v,x}}+\sum_{|\alpha|\le K}\|\psi_{|\alpha|-2}\partial^{\alpha}(\I-\P)f\|^2_{L^2_{v,x}}\\
	&\qquad+\sum_{\substack{|\alpha|+|\beta|\le K}}\|\psi_{|\alpha|+|\beta|-2}w^{l-|\alpha|-|\beta|}\partial^\alpha_\beta(\I-\P) f\|^2_{L^2_{v,x}}
\end{align*}
Recalling the $a$ $priori$ assumption \eqref{priori}, the desired estimate \eqref{42a} follows directly from \eqref{58}. 
\end{proof}

For the higher order instant energy, we have the following theorem. 
\begin{Thm}\label{thm42}For any $l\ge K$, there is $\E^h_{K,l}(t)$ satisfying \eqref{Defeh} such that for any $0\le t\le T$,
	\begin{align}
		\partial_t\E^h_{K,l}+\lambda \D_{K,l}(t)\lesssim \|\partial_t\phi\|_{L^\infty}\E^h_{K,l}(t)+\|\nabla_x(a_\pm,b,c)\|^2_{L^2_x}, 
	\end{align}where $\D_{K,l}$ is defined by \eqref{Defd}. 
\end{Thm}
\begin{proof}
	By letting $|\alpha|\ge1$ in \eqref{45}, repeating the calculations from \eqref{45} to \eqref{47}, we can instead obtain 
	\begin{align}\label{63}
				&\notag\quad\,\frac{1}{2}\partial_t\sum_{\pm}\sum_{1\le|\alpha|\le K}\Big(\|\psi_{|\alpha|-2}e^{\frac{\pm\phi}{2}}\partial^{\alpha} f_\pm\|_{L^2_{v,x}} +
		\|\psi_{|\alpha|-2}\partial^{\alpha}\nabla_x\phi\|_{L^2_x}^2\Big)\\&\qquad + \lambda \sum_{\pm}\sum_{1\le|\alpha|\le K}\|\psi_{|\alpha|-2}e^{\frac{\pm\phi}{2}}(\tilde{a}^{1/2})^w\partial^{\alpha} (\II-\PP)f\|_{L^2_{v,x}}^2
		\lesssim \|\partial_t\phi\|_{L^\infty}\E^h_{K,l}(t)+\E^{1/2}_{K,l}(t)\D_{K,l}(t),
	\end{align}
Notice that here the first right-hand term contains $\E^h_{K,l}$ since there's at least one derivative on $x$. 
In order to eliminate the term $\|(\tilde{a}^{1/2})^w(\I-\P)f\|^2_{L^2_{v,x}}$ in \eqref{48}, we shall take the inner product of \eqref{36} with $e^{\pm\phi}(\II-\PP)f$ over $\R^3_v\times\R^3_x$ and $\alpha=\beta=0$. 
\begin{align*}
		&\quad\,\big(\partial_t(\II-\PP)f,e^{\pm\phi}(\II-\PP)f\big)_{L^2_{v,x}} + \big(v_i\partial^{e_i}(\II-\PP)f,e^{\pm\phi}(\II-\PP)f\big)_{L^2_{v,x}} \\
	&\qquad\pm \big(\frac{1}{2}\partial^{e_i}\phi v_i(\II-\PP)f,e^{\pm\phi}(\II-\PP)f\big)_{L^2_{v,x}} 
	\mp\big(\partial^{e_i}\phi\partial_{e_i}(\II-\PP)f,e^{\pm\phi}(\II-\PP)f\big)_{L^2_{v,x}}\\&\qquad \pm \big(\partial^{e_i}\phi v_i\mu^{1/2},e^{\pm\phi}(\II-\PP)f\big)_{L^2_{v,x}} -  \big(L_\pm (\I-\P)f,e^{\pm\phi}(\II-\PP)f\big)_{L^2_{v,x}} \\
	&= -\big(\partial_t\PP f,e^{\pm\phi}(\II-\PP)f\big)_{L^2_{v,x}} + \big(v_i\partial^{e_i}\PP f,e^{\pm\phi}(\II-\PP)f\big)_{L^2_{v,x}} \mp \big(\frac{1}{2}\partial^{e_i}\phi v_i\PP f,e^{\pm\phi}(\II-\PP)f\big)_{L^2_{v,x}} \\ &\qquad\mp\big(\partial^{e_i}\phi\partial^{}_{e_i}\PP f,e^{\pm\phi}(\II-\PP)f\big)_{L^2_{v,x}}
	+
	\big(\partial^{\alpha}_\beta \Gamma_{\pm}(f,f),e^{\pm\phi}(\II-\PP)f\big)_{L^2_{v,x}}.
\end{align*}
As before, we denote these terms with summation over $\pm$ by $L_1,\dots,L_{11}$ and estimate them term by term. 
\begin{equation*}
	L_1\ge \frac{1}{2}\partial_t\sum_{\pm}\|e^{\frac{\pm\phi}{2}}(\II-\PP)f\|^2_{L^2_{v,x}}-C\|\partial_t\phi\|_{L^\infty_x}\E^h_{K,l}(t). 
\end{equation*}
The same as \eqref{48aa}, by integration by parts on $x$, $L_2+L_3=0$. Same as \eqref{50}, $L_4=0$. 
Similar to \eqref{43}, 
\begin{align*}
	L_5 = \frac{1}{2}\partial_t\|\nabla_x\phi\|^2_{L^2_x}. 
\end{align*}
By Lemma \ref{lemmaL}, 
\begin{align*}
	L_6 \ge \lambda\|(\tilde{a}^{1/2})^w(\I-\P)f\|^2_{L^2_{v,x}}. 
\end{align*}
Recalling the conservation laws \eqref{19}, one has 
\begin{align*}
	L_7 \le \frac{\lambda}{4}\|(\tilde{a}^{1/2})^w(\I-\P)f\|^2_{L^2_{v,x}}+C\big(\|\nabla_x(a_\pm,b,c)\|^2_{L^2_x}+\|\nabla_x\phi\|^2_{L^2_x}+\|(\tilde{a}^{1/2})^w\nabla_x(\I-\P)f\|^2_{L^2_{v,x}}+\E_{K,l}\D_{K,l}\big).
\end{align*}
By Cauchy-Schwarz inequality, 
\begin{align*}
	L_8\le \frac{\lambda}{4}\|(\tilde{a}^{1/2})^w(\I-\P)f\|^2_{L^2_{v,x}}+C\|\nabla_x(a_\pm,b,c)\|^2_{L^2_{v,x}}. 
\end{align*}
Similar to the calculation on $J_9,J_{10},J_{11}$, we have that $L_9,L_{10},L_{11}$ are bounded above by $\E^{1/2}_{K,l}\D_{K,l}$. 
Combining the above estimate, we have 
\begin{align}\label{64}
	&\notag\frac{1}{2}\partial_t\sum_{\pm}\|e^{\frac{\pm\phi}{2}}(\II-\PP)f\|^2_{L^2_{v,x}}+\frac{1}{2}\partial_t\|\nabla_x\phi\|^2_{L^2_x}+\lambda\|(\tilde{a}^{1/2})^w(\I-\P)f\|^2_{L^2_{v,x}}\\
	&\lesssim \|\partial_t\phi\|_{L^\infty_x}\E^h_{K,l}(t)+\|\nabla_x(a_\pm,b,c)\|^2_{L^2_{v,x}}+\|(\tilde{a}^{1/2})^w\nabla_x(\I-\P)f\|^2_{L^2_{v,x}}+(\E^{1/2}_{K,l}+\E_{K,l})\D_{K,l}
\end{align}

Now we use combination $C_2\times(C_1\times(\eqref{63}+\kappa\times\eqref{24a})+\eqref{64})+\eqref{74}+\eqref{48}$ with $\kappa<<1$. Taking $C_1>>\frac{1}{\kappa}$ sufficiently large then taking $C_2$ sufficiently large, we obtain that when $\psi=1$, 
\begin{align*}
			&\partial_t\E^h_{K,l}(t) + C_1C_2\kappa\lambda\sum_{1\le|\alpha|\le K-1} \|\partial^\alpha\nabla_x(a_\pm,b,c)\|^2_{L^2_x}+C_1C_2\kappa\|\nabla_x(a_+-a_-)\|^2_{L^2_x}\\
			&\qquad +C_1C_2\sum_{|\alpha|\le K-1}\|\partial^\alpha E\|^2_{L^2_x}+ C_1C_2\lambda \sum_{\pm}\sum_{1\le|\alpha|\le K}\|\psi_{|\alpha|-2}e^{\frac{\pm\phi}{2}}(\tilde{a}^{1/2})^w\partial^{\alpha} f_\pm\|_{L^2_{v,x}}^2\\
			&\qquad+C_2\lambda\|(\tilde{a}^{1/2})^w(\I-\P)f\|^2_{L^2_{v,x}}
			+\lambda \sum_{\pm}\sum_{\substack{|\alpha|+|\beta|\le K}}\|\psi_{|\alpha|+|\beta|-2}e^{\frac{\pm\phi}{2}}(\tilde{a}^{1/2})^ww^{l-|\alpha|-|\beta|}\partial^{\alpha}_\beta (\I-\P)f\|_{L^2_{v,x}}^2\notag\\
			&\lesssim\|\partial_t\phi\|_{L^\infty}\E^h_{K,l}(t)+(\E^{1/2}_{K,l}(t)+\E_{K,l}(t))\D_{K,l}(t)+\|\nabla_x(a_\pm,b,c)\|^2_{L^2_x},
\end{align*}
	where left-hand terms except the first one adding $\|\nabla_x(a_\pm,b,c)\|^2_{L^2_x}$ is larger than $\D_{K,l}$ and 
	\begin{align*}
		\E^h_{K,l}(t) &= C_1C_2\kappa\E^{(1)}_{K,h}+\frac{C_1C_2}{2}\sum_{\pm}\sum_{1\le|\alpha|\le K}\Big(\|\psi_{|\alpha|-2}e^{\frac{\pm\phi}{2}}\partial^{\alpha} f_\pm\|_{L^2_{v,x}} +
		C_1C_2\|\psi_{|\alpha|-2}\partial^{\alpha}\nabla_x\phi\|_{L^2_x}^2\Big)\\
		&\qquad+\frac{C_2}{2}\sum_{\pm}\|e^{\frac{\pm\phi}{2}}(\II-\PP)f\|^2_{L^2_{v,x}}+\frac{C_2}{2}\|\nabla_x\phi\|^2_{L^2_x}\\
		&\qquad+\sum_{\pm}\sum_{\substack{|\alpha|+|\beta|\le K}}\|e^{\frac{\pm\phi}{2}}\psi_{|\alpha|+|\beta|-2}w^{l-|\alpha|-|\beta|}\partial^{\alpha}_\beta(\II-\PP)f\|^2_{L^2_{v,x}}.
	\end{align*}
Noticing $\kappa<<1$ is sufficiently small, it's direct to verify \eqref{Defeh}. 
At last, by using the $a$ $priori$ assumption \eqref{priori}, we obtain the desired estimate. 
\end{proof}

The idea of the following proof is similar to the Boltzmann case \cite{Strain2012}.
\begin{Thm}Let $T\in(0,\infty]$.
	Consider the solution $f$ to Cauchy problem \eqref{7}-\eqref{9}.
Let $i=1$ if $0<s<\frac{1}{2}$ and $i=2$ if $\frac{1}{2}\le s<1$. Assume $K\ge i+1$, $l\ge\max\{K,\gamma/2+s+2,\gamma+2s+i+1\}$. Define 
	\begin{equation}\label{79}
		X(t) = \sup_{0\le\tau\le t}(1+\tau)^{3/2}\E_{K,l}(\tau)+\sup_{0\le\tau\le t}(1+t)^{5/2}\E^h_{K,l}(\tau),
	\end{equation}
and
\begin{equation*}
	\epsilon_0 = (\E_{K,l}(0))^{1/2}+\|f_0\|_{Z_1}+\|E_0\|_{L^1}.
\end{equation*}
Then under the $a$ $priori$ assumption \eqref{priori} for $\delta_0>0$ sufficiently small, we have 
	\begin{equation}\label{80}
		X(t)\lesssim \epsilon^2_0+X^{3/2}(t)+X^2(t),
	\end{equation}
for any $0\le t\le T$. 
\end{Thm}
\begin{proof}
{Step 1}.	From \eqref{34} and \eqref{42a}, we have that for $l\ge 0$, 
	\begin{align}\label{81}
		\|\partial_t\phi\|_{L^\infty_x}&\lesssim (\E^h_{K,l}(t))^{1/2}\lesssim(1+t)^{-5/4}X^{1/2}(t)\lesssim\delta^{1/2}_0(1+t)^{-5/4},\\
		&\partial_t\E_{K,l}(t)+D_{K,l}(t)\lesssim \|\partial_t\phi\|_{L^\infty_x}\E_{K,l}(t).\notag
	\end{align}
By using Gronwall's inequality, for $0\le t\le T$, 
\begin{align}\label{81a}
	\E_{K,l}(t)\lesssim \E_{K,l}(0)e^{C\int^t_0\|\partial_t\phi(\tau)\|_{L^\infty_x}}\,d\tau\lesssim \E_{K,l}(0)\lesssim \epsilon_0^2.
\end{align}
	
	Step 2. To prove the decay of $\E^h_{K,l}$, we use Theorem \ref{thm42} to get 
	\begin{align*}
		\partial_t\E^h_{K,l}+\lambda \D_{K,l}(t)\lesssim \|\partial_t\phi\|_{L^\infty}\E^h_{K,l}(t)+\|\nabla_x(a_\pm,b,c)\|^2, 
	\end{align*}
We will use the trick in \cite{Strain2012}.
Noticing the term $\|\partial^K E\|_{L^2_x}^2$ inside $\E^{h}_{K,l}$ is bounded above by 
\begin{align*}
	\|\nabla^{K+1}_x\phi\|^2_{L^2_x}=\|\nabla_x^{K+1}\Delta^{-1}_x(a_+-a_-)\|^2_{L^2_x}
	\le \|\nabla_x^{K-1}(a_\pm,b,c)\|^2_{L^2_x},
\end{align*}
which is a term in $\D_{K,l}$. Also, for hard potential,  $\|\cdot\|_{L^2_v}\lesssim\|(\tilde{a}^{1/2})^w(\cdot)\|_{L^2_v}$. Hence,
\begin{align*}
	\partial_t\E^h_{K,l}+\lambda \E^h_{K,l}(t)\lesssim \|\partial_t\phi\|_{L^\infty}\E^h_{K,l}(t)+\|\nabla_x(a_\pm,b,c)\|^2.
\end{align*}
By Gronwall's inequality, 
\begin{align}\label{82a}
	\E^h_{K,l}(t)\lesssim e^{-\lambda t}\E^h_{K,l}(0)+\int^t_0\,d\tau e^{-\lambda (t-\tau)}\big(\|\partial_t\phi(\tau)\|_{L^\infty}\E^h_{K,l}(\tau)+\|\nabla_x(a_\pm,b,c)(\tau)\|^2\big).
\end{align}
We will need to deal with the terms inside the time integral. By \eqref{81}, for $0\le \tau\le t$, 
\begin{align*}
	\|\partial_t\phi(\tau)\|_{L^\infty}\E^h_{K,l}(\tau)\lesssim (\E^h_{K,l}(t))^{3/2}\lesssim (1+t)^{-15/4}X^{3/2}(t).
\end{align*}
We claim that for $0\le t\le T$, 
\begin{align}\label{82}
	\|\nabla_x(a_\pm,b,c)(t)\|^2\lesssim (\epsilon_0+X(t))(1+t)^{-5/4}
\end{align}
Recalling \eqref{42b}, by Duhamel's principle, we can write the solution to \eqref{7} as 
\begin{align*}
	f(t) = e^{tB}f_0+\int^t_0e^{(t-\tau)B}g(\tau)\,d\tau,
\end{align*}where $g=(g_+,g_-)$ is defined by \eqref{22a}. Applying Theorem \ref{homogen} with $m=1$, $l=0$, $\sigma_1= \frac{5}{4}$ therein, we have 
\begin{align}\label{83}
	\|\nabla_x\P f(t)\|_{L^2_{v,x}} \lesssim\|\nabla_xf(t)\|_{L^2_{v,x}} \lesssim \epsilon_0(1+t)^{-5/4}+\int^t_0(1+t-\tau)^{-5/4}\big(\|g(\tau)\|_{Z_1}+\|\nabla_xg(\tau)\|_{L^2_{v,x}}\big)\,d\tau,
\end{align}where we use the fact that $(g_\pm,\mu^{1/2})=0$.
By using \eqref{12a} and Young's inequality, we have 
\begin{align*}
	\|\Gamma(f,f)(\tau)\|_{Z_1}&\lesssim \int\,dx\|\<v\>^{\gamma/2+s}f(\tau)\|_{L^2_v}\|\<v\>^{\gamma+2s}f(\tau)\|_{H^2_v}\\
	&\lesssim \|\<v\>^{\gamma/2+s}f(\tau)\|_{L^2_{v,x}}\|\<v\>^{\gamma+2s}f(\tau)\|_{H^i_vL^2_x}\\
	&\lesssim \E_{K,l}(\tau),
\end{align*}since $l\ge i+\gamma+2s$. 
On the other hand,
\begin{align*}
	\|\nabla_x\phi\cdot\nabla_vf_\pm(\tau)\|_{Z_1}&\lesssim \|\nabla_x\phi(\tau)\|_{L^2_x}\|\nabla_vf(\tau)\|_{L^2_{v,x}}\lesssim\E_{K,l}(\tau),\\
	\|v\cdot\nabla_x\phi f_\pm(\tau)\|_{Z_1}&\lesssim\|\nabla_x\phi(\tau)\|_{L^2_x}\|vf_\pm(\tau)\|_{L^2_{v,x}}\lesssim\E_{K,l}(\tau).
\end{align*}
Similarly, by using \eqref{13}, 
\begin{align*}\notag
	\|\nabla_x\Gamma(f,f)(\tau)\|_{L^2_{v,x}}&\lesssim \Big\|\|\<v\>^{\gamma/2+s}\nabla_xf(\tau)\|_{L^2_v}\|\<v\>^{\gamma+2s}f(\tau)\|_{H^i_v}+\|\<v\>^{\gamma/2+s}f(\tau)\|_{L^2_v}\|\<v\>^{\gamma+2s}\nabla_xf(\tau)\|_{H^i_v}\Big\|_{L^2_x}\\
	&\lesssim \|\<v\>^{\gamma/2+s}f(\tau)\|_{L^2_{v}H^2_{x}}\|\<v\>^{\gamma+2s}f(\tau)\|_{H^i_vH^1_x}\\
	&\lesssim \E_{K,l}(\tau),\notag
\end{align*}
by \eqref{13}, since $K\ge i+1$, $l\ge \max\{\gamma/2+s+2,\gamma+2s+i+1\}$. 
\begin{align*}
	\|\nabla_x(\nabla_x\phi\cdot\nabla_vf_\pm(\tau))\|_{L^2_{v,x}}&\lesssim \|\nabla_x\phi(\tau)\|_{H^2_x}\|f_\pm(\tau)\|_{H^1_{v}H^1_x}\lesssim\E_{K,l}(\tau)\\
	\|\nabla_x(v\cdot\nabla_x\phi f_\pm)(\tau)\|_{L^2_{v,x}}&\lesssim \|\nabla_x\phi(\tau)\|_{H^1_x}\|vf_\pm(\tau)\|_{L^2_{v}H^1_x}\lesssim\E_{K,l}(\tau).
\end{align*}
Plugging the above estimate into \eqref{83} and using $\E_{K,l}(t)\le (1+t)^{-3/2}X(t)$, we have 
\begin{align}\label{84}
	\|\nabla_xf(t)\|_{L^2_{v,x}} &\lesssim\notag \epsilon_0(1+t)^{-5/4}+X(t)\int^t_0(1+t-\tau)^{-5/4}(1+\tau)^{-3/2}\,d\tau\\
	&\lesssim (\epsilon_0+X(t))(1+t)^{-5/4}.
\end{align}This proves the claim. 
Now \eqref{82a} gives that for $0\le t\le T$, 
\begin{align}\label{85}
	\E^h_{K,l}(t)\notag&\lesssim e^{-\lambda t}\epsilon_0+\int^t_0\,d\tau e^{-\lambda (t-\tau)}\big((1+\tau)^{-15/4}X^{3/2}(t)+(\epsilon_0^2+X^2(t))(1+\tau)^{-5/2}\big)\\
	&\lesssim (\epsilon_0^2+X^{3/2}(t)+X^2(t))(1+t)^{-5/2}. 
\end{align}

Step 3. By using the same way proving \eqref{84} and applying $m=0$, $\sigma_0=\frac{3}{4}$ instead, we can obtain 
\begin{align}\label{86}
	\|f\|_{L^2_{v,x}}\lesssim (\epsilon_0+X(t))(1+t)^{-3/4}.
\end{align}
Since $\E_{K,l}(t)\approx\|\P f\|^2_{L^2_x}+\E^h_{K,l}(t)$, we have from \eqref{85} and \eqref{86} that 
\begin{align}\label{87}
	\E_{K,l}(t)\approx (\epsilon_0^2+X^{3/2}(t)+X^2(t))(1+t)^{-3/2}. 
\end{align}
Now the desired estimate \eqref{80} follows from \eqref{81a}, \eqref{85} and \eqref{87}. This completes the proof. 
\end{proof}

\begin{proof}[Proof of Theorem \ref{main1}]
	It follows immediate from the $a$ $priori$ estimate \eqref{80} that $X(t)\lesssim \epsilon^2_0$ holds true for any $t\ge 0$, whenever $\epsilon_0$ is sufficiently small. The rest is to prove the local existence and uniqueness of solutions in terms of the energy norm $\E_{K,l}$ and the non-negativity of $\F_\pm=\mu+\mu^{1/2}f$. One can use the iteration on system 
	\begin{equation}
		\left\{\begin{aligned}
			&\partial_tf^{n+1}_\pm+v\cdot\nabla_xf^{n+1}_\pm\mp\nabla_x\phi^n\cdot\nabla_xf^{n+1}_\pm\pm\frac{1}{2}\nabla_x\phi^n\cdot vf^{n+1}_\pm\pm v\mu^{1/2}\cdot\nabla_x\phi^n-L_\pm f=\Gamma_\pm(f^n,f^{n+1}),\\
			&-\Delta_x\phi^{n+1}=\int_{\R^3}\mu^{1/2}(f^{n+1}_+-f^{n+1}_-)\,dv,\\
			&f^{n+1}|_{t=0}=f_0,
		\end{aligned}\right.
	\end{equation}
and the details of proof are omitted for brevity; see \cite{Guo2012, Strain2013} and \cite{Gressman2011}.
Therefore, the unique global-in-time solution to \eqref{7}-\eqref{9} exists by using continuity argument. The estimate \eqref{15a} follows from $X(t)\lesssim \epsilon^2_0$ directly. 
\end{proof}

\section{Regularity}\label{sec5}
In this section, we will prove that the global-in-time solution we found in the last section is actually smooth in $v,x,t$. 
Let $N>0$ be a large number chosen later. Assume $T\in(0,1]$, $t\in[0,T]$ and 
\begin{equation}\label{93}
	\psi=t^{N},\quad\psi_k=\left\{\begin{aligned}
		1, \text{  if $k\le 0$},\\
		\psi^k, \text{ if $k> 0$}. 
	\end{aligned}\right.
\end{equation}
is this section. Then $|\partial_t\psi_k|\lesssim \psi_{k-1/N}$. Let $f$ be the smooth solution to \eqref{7}-\eqref{9} over $0\le t\le T$ and assume the $a$ $priori$ assumption 
\begin{align}\label{priori1}
	\sup_{0\le t\le T}\E_{K,l}(t)\le \delta_0,
\end{align}where $\delta_0\in(0,1)$ is a suitably small constant. 
\begin{Lem}\label{lem51}Assume $\gamma+2s>0$, $K\ge 3$, $l\ge K$.
	Let $f$ to be the solution to \eqref{7}-\eqref{9} satisfying 
	\begin{align}\label{100}
		\epsilon^2_1 = \E_{2,l}(0) <\infty
	\end{align} is sufficiently small. Then there exists $t_0>0$ independent of $T$ such that for $0< t\le t_0$, 
	\begin{align}
		\E_{K,l}(t)\lesssim \epsilon^2_1.
	\end{align}
\end{Lem}
\begin{proof}
		Assume $|\alpha|+|\beta|\le K$ in this proof. 
	The reason of choosing such $\psi_{|\alpha|+|\beta|-2}$ is that whenever $l\ge K\ge 2$, the initial value $\E_{K,l}(0)=\E_{2,l}(0)$, since $\psi_{|\alpha|+|\beta|-2}|_{t=0}=0$ whenever $|\alpha|+|\beta|\ge 3$. 

Step 1. We claim that when $\psi$ is defined by \eqref{93}, 
\begin{align}
	\partial_t\E_{K,l}(t)+\lambda\D_{K,l}(t) \lesssim \|\partial_t\phi\|_{L^\infty_x}\E_{K,l}(t)+\E_{K,l} + \sum_{|\alpha|+|\beta|\le K}\|\psi_{|\alpha|+|\beta|-2-\frac{1}{2N}}w^{l-|\alpha|-|\beta|}\partial^\alpha_\beta f\|^2_{L^2_{v,x}}.
\end{align}
We will apply the calculation from \eqref{35} to \eqref{74}. Notice that the only difference is that $\psi=t^{N}$ and the term $\partial_t(\psi_{|\alpha|+|\beta|-2})$ is not zero any more, which is in the third term of \eqref{37}, \eqref{99c}. They are bounded above by 
	\begin{align}\label{98a}
		&\notag\sum_{\pm}\sum_{|\alpha|+|\beta|\le K}\big|(\partial_t(\psi_{|\alpha|+|\beta|-2})\partial^{\alpha}_\beta (\II-\PP)f,\psi_{|\alpha|+|\beta|-2}\notag w^{2l-2|\alpha|-2|\beta|}e^{\pm\phi}\partial^{\alpha}_\beta (\II-\PP)f)_{L^2_{v,x}}\big|\\&\qquad
		+\sum_{\pm}\sum_{|\alpha|\le K}\big|(\partial_t(\psi_{|\alpha|-2})\partial^{\alpha} f_\pm,\psi_{|\alpha|-2} e^{\pm\phi}w^{2l-2|\alpha|}\partial^{\alpha} f_\pm)_{L^2_{v,x}}\big|\\
		&\lesssim \sum_{|\alpha|+|\beta|\le K}\|\psi_{|\alpha|+|\beta|-2-\frac{1}{2N}}w^{l-|\alpha|-|\beta|}\partial^\alpha_\beta f\|^2_{L^2_{v,x}}.\notag
	\end{align}
Together with \eqref{47}, we have 
\begin{align}\label{111}
	&\notag\quad\,\frac{1}{2}\partial_t\sum_{\pm}\sum_{|\alpha|\le K}\Big(\|\psi_{|\alpha|-2}e^{\frac{\pm\phi}{2}}\partial^{\alpha} f_\pm\|_{L^2_{v,x}} +
	\|\psi_{|\alpha|-2}\partial^{\alpha}\nabla_x\phi\|_{L^2_x}^2\Big)\\&\qquad + \lambda \sum_{\pm}\sum_{|\alpha|\le K}\|\psi_{|\alpha|-2}e^{\frac{\pm\phi}{2}}(\tilde{a}^{1/2})^w\partial^{\alpha} f_\pm\|_{L^2_{v,x}}^2\\
	&\lesssim \|\partial_t\phi\|_{L^\infty}\E_{K,l}(t)+\E^{1/2}_{K,l}(t)\D_{K,l}(t)+\sum_{|\alpha|+|\beta|\le K}\|\psi_{|\alpha|+|\beta|-2-\frac{1}{2N}}w^{l-|\alpha|-|\beta|}\partial^\alpha_\beta f\|^2_{L^2_{v,x}}.\notag
\end{align}
Now it suffices to compute the energy estimate with weight and mixed derivatives. The idea is similar to Step 3 in Theorem \ref{thm41}. For any $|\alpha|+|\beta|\le K$, we apply $\partial^\alpha_\beta$ to equation \eqref{7} and take the inner product with $\psi_{2|\alpha|+2|\beta|-4}e^{\pm\phi}w^{2l-2|\alpha|-2|\beta|}\partial^\alpha_\beta f_\pm$. Then 
\begin{align*}
		&\quad\,\Big(\partial_t\partial^{\alpha}_\beta  f,e^{\pm\phi}\psi_{2|\alpha|+2|\beta|-4}w^{2l-2|\alpha|-2|\beta|}\partial^{\alpha}_\beta  f\Big)_{L^2_{v,x}}\\
&\qquad + \Big(\sum_{\beta_1\le \beta}C^{\beta_1}_{\beta}\partial_{\beta_1}v_i\partial^{e_i+\alpha}_{\beta-\beta_1} f,e^{\pm\phi}\psi_{2|\alpha|+2|\beta|-4}w^{2l-2|\alpha|-2|\beta|}\partial^{\alpha}_\beta  f\Big)_{L^2_{v,x}} \\
&\qquad\pm \Big(\frac{1}{2}\sum_{\substack{\alpha_1\le\alpha\\\beta_1\le\beta}}\partial^{e_i+\alpha_1}\phi \partial_{\beta_1}v_i\partial^{\alpha-\alpha_1}_{\beta-\beta_1} f,e^{\pm\phi}\psi_{2|\alpha|+2|\beta|-4}w^{2l-2|\alpha|-2|\beta|}\partial^{\alpha}_\beta  f\Big)_{L^2_{v,x}} \\ &\qquad\mp\Big(\sum_{\substack{\alpha_1\le\alpha}}\partial^{e_i+\alpha_1}\phi\partial^{\alpha-\alpha_1}_{\beta+e_i} f,e^{\pm\phi}\psi_{2|\alpha|+2|\beta|-4}w^{2l-2|\alpha|-2|\beta|}\partial^{\alpha}_\beta  f\Big)_{L^2_{v,x}}\\
&\qquad \pm \Big(\partial^{e_i+\alpha}\phi \partial_\beta(v_i\mu^{1/2}),e^{\pm\phi}\psi_{2|\alpha|+2|\beta|-4}w^{2l-2|\alpha|-2|\beta|}\partial^{\alpha}_\beta  f\Big)_{L^2_{v,x}}\\
&\qquad - \Big(\partial^{\alpha}_\beta L_\pm f,e^{\pm\phi}\psi_{2|\alpha|+2|\beta|-4}w^{2l-2|\alpha|-2|\beta|}\partial^{\alpha}_\beta  f\Big)_{L^2_{v,x}} \\
&= \Big(\partial^{\alpha}_\beta \Gamma_{\pm}(f,f),e^{\pm\phi}\psi_{2|\alpha|+2|\beta|-4}w^{2l-2|\alpha|-2|\beta|}\partial^{\alpha}_\beta  f\Big)_{L^2_{v,x}}.
\end{align*}
We denote these terms by $N_1$ to $N_7$. Similar to \eqref{37} and \eqref{46}, using \eqref{98a}, we have 
\begin{align*}
	N_1 \ge  \frac{1}{2}\partial_t\|e^{\frac{\pm\phi}{2}}\psi_{|\alpha|+|\beta|-2}\partial^\alpha_\beta f_\pm\|^2_{L^2_{v,x}} - C\|\partial_t\phi\|_{L^\infty_x}\E_{K,l}(t) - C\sum_{|\alpha|+|\beta|\le K}\|\psi_{|\alpha|+|\beta|-2-\frac{1}{2N}}w^{l-|\alpha|-|\beta|}\partial^\alpha_\beta f\|^2_{L^2_{v,x}}.
\end{align*}
The term $N_2$ is canceled by $N_3$ with $\alpha_1=0$ in $N_3$. For the left terms in $N_3$, $\alpha_1\neq 0$ and hence $|\alpha-\alpha_1|+1\le |\alpha|$. Using the same argument in \eqref{40}, the left terms in $N_3$ are bounded above by $\E^{1/2}_{K,l}(t)\D_{K,l}(t)$. 
For $N_4$, as \eqref{50} and \eqref{42}, by taking integration by parts and noticing the total order of derivatives, we have $|N_4|\lesssim  \E^{1/2}_{K,l}(t)\D_{K,l}(t)$. For $N_5$, there's exponential decay in $v$ and hence, 
\begin{align}\label{105a}
	|N_5|\notag&\lesssim \|\psi_{|\alpha|-2}\partial^{e_i+\alpha}\phi\|_{L^2_x} \|\psi_{|\alpha|+|\beta|-2}w^{l-|\alpha|-|\beta|}\partial^{\alpha}_\beta  f\|_{L^2_{v,x}}\\
	&\lesssim \eta\|\psi_{|\alpha|+|\beta|-2}(\tilde{a}^{1/2})^ww^{l-|\alpha|-|\beta|}\partial^{\alpha}_\beta  f\|_{L^2_{v,x}}^2+C_\eta\|\psi_{|\alpha|-2}\partial^{\alpha}\nabla_x\phi\|_{L^2_x}^2.
\end{align}The second term on the right hand of \eqref{105a} is bounded above by $\E_{K,l}$. 
For $N_6$, using Lemma \ref{lemmaL} and $\psi\le 1$, we have 
\begin{align*}
	&\quad\,-\sum_\pm\big(\partial^{\alpha}_\beta L_\pm f,e^{\pm\phi}\psi_{2|\alpha|+2|\beta|-4}w^{2l-2|\alpha|-2|\beta|}\partial^{\alpha}_\beta  f\big)_{L^2_{v,x}}\\
	&\gtrsim \|\psi_{|\alpha|+|\beta|-2}(\tilde{a}^{1/2})^ww^{l-|\alpha|-|\beta|}\partial^\alpha_\beta f\|^2_{L^2_v} - \eta\sum_{|\beta_1|\le|\beta|}\|\psi_{|\alpha|+|\beta_1|-2}(\tilde{a}^{1/2})^ww^{l-|\alpha|-|\beta_1|}\partial^\alpha_{\beta_1}f_\pm\|^2_{L^2_v}\\
	&\qquad\qquad\qquad\qquad\qquad\qquad\qquad\qquad-C_\eta\|\psi_{|\alpha|-2}\partial^\alpha f\|^2_{L^2(B_{C_\eta})}
\end{align*}
By Lemma \ref{lemmag}, $N_7$ is bounded above by $\E^{1/2}_{K,l}\D_{K,l}(t)+\E_{K,l}(t)\D^{1/2}_{K,l}(t)\lesssim \E_{K,l}(t)+(\E^{1/2}_{K,l}+\E_{K,l}(t))\D_{K,l}(t)$. Combining the above estimate, taking summation on $\pm$, $|\alpha|+|\beta|\le K$ and letting $\eta$ sufficiently small, we have 
\begin{align}\label{111a}
	&\quad\,\frac{1}{2}\partial_t\sum_\pm\sum_{|\alpha|+|\beta|\le K}\|e^{\frac{\pm\phi}{2}}\psi_{|\alpha|+|\beta|-2}\partial^\alpha_\beta f_\pm\|^2_{L^2_{v,x}}+\lambda \sum_\pm\sum_{|\alpha|+|\beta|\le K}\|\psi_{|\alpha|+|\beta|-2}(\tilde{a}^{1/2})^ww^{l-|\alpha|-|\beta|}\partial^\alpha_\beta f_\pm\|^2_{L^2_v} \notag\\&\lesssim \|\partial_t\phi\|_{L^\infty_x}\E_{K,l}(t) + \sum_{|\alpha|+|\beta|\le K}\|\psi_{|\alpha|+|\beta|-2-\frac{1}{2N}}w^{l-|\alpha|-|\beta|}\partial^\alpha_\beta f\|^2_{L^2_{v,x}}+\E_{K,l}+\E_{K,l}(t)\D_{K,l}(t).
\end{align}
Taking combination $\eqref{111}+\eqref{111a}$ and noticing \eqref{78}\eqref{15a}, we have 
\begin{align}
	\partial_t\E_{K,l}(t)+\lambda\D_{K,l}(t) \lesssim \|\partial_t\phi\|_{L^\infty_x}\E_{K,l}(t)+\E_{K,l} + \sum_{|\alpha|+|\beta|\le K}\|\psi_{|\alpha|+|\beta|-2-\frac{1}{2N}}w^{l-|\alpha|-|\beta|}\partial^\alpha_\beta f\|^2_{L^2_{v,x}}.
\end{align}
This proves the claim. So it suffices to control the last term.

Step 2. 
To deal with it, as in \cite{Deng2020b}, we define $\tilde{b}(v,y)$ and constants $\delta_1,\delta_2>0$ as the followings. 
Let $c = \max\{0,-\frac{\gamma}{2}\}$. If $\gamma+3s-c\le 1$, we let 
\begin{align}\label{107}
	\delta_1 = \frac{-s+c}{\gamma+s+c-2},\qquad\delta_2 = \frac{-1}{\gamma+s+c-2}.
\end{align}
If $\gamma+3s-c\ge 1$, we let
\begin{align}\label{108}
	\delta_1 = \frac{s-c}{2s-2c+1},\qquad\delta_2 = \frac{1}{2s-2c+1}.
\end{align}
Then in each case, since $s-c\in(0,1)$, by direct calculation we have the following estimates. 
\begin{align*}
	\delta_1\le 1,\quad\delta_1\le \delta_2,\quad\delta_1-\frac{1}{2}+\frac{\delta_2}{2}\le 0,\quad   \frac{\delta_1}{\delta_2}+c\le s,\quad \frac{\delta_1-1}{\delta_2}\le\gamma+2s-2.
\end{align*} 
Let $\chi_0$ be a smooth cutoff function such that $\chi_0(z)$ equal to $1$ when $|z|<\frac{1}{2}$ and equal to $0$ when $|z|\ge 1$. Define 
\begin{align}\label{98}
	\tilde{b}(v,y) &= (1+|v|^2+|y|^2+|v\wedge y|^2)^{\delta_1},\\
	\chi(v,\eta) &= \chi_0\bigg(\frac{1+|v|^2+|\eta|^2+|v\wedge\eta|^2}{(1+|v|^2+|y|^2+|v\wedge y|^2)^{\delta_2}}\bigg),\notag
\end{align}and
\begin{align*}
	\theta(v,\eta) = (1+|v|^2+|y|^2+|v\wedge y|^2)^{\delta_1-1} (y\cdot\eta+(v\wedge y)\cdot(v\wedge\eta))\chi(v,\eta).
\end{align*}
Using the definition \eqref{98} of $\tilde{b}$, by Young's inequality, if $|\alpha|+|\beta|>2$, we have 
\begin{align}\notag
	\psi_{|\alpha|+|\beta|-2-\frac{1}{2N}}&\lesssim \delta \big((\tilde{b}^{1/2})^{\frac{|\alpha|+|\beta|-2-\frac{1}{2N}}{|\alpha|+|\beta|-2}}\psi_{|\alpha|+|\beta|-2-\frac{1}{2N}}\big)^{\frac{|\alpha|+|\beta|-2}{|\alpha|+|\beta|-2-\frac{1}{2N}}}+C_{0,\delta}\big((\tilde{b}^{-1/2})^{\frac{|\alpha|+|\beta|-2-\frac{1}{2N}}{|\alpha|+|\beta|-2}}\big)^{2N(|\alpha|+|\beta|-2)}\\
	&\lesssim \delta\, \tilde{b}^{1/2}\psi_{|\alpha|+|\beta|-2} + C_{0,\delta}\<v\>^{-N_0}\<y\>^{-N_0},\label{100a}
\end{align}where $C_{0,\delta}$ is a large constant depending on $\delta>0$,
for any $N_0>0$, with $N>>N_0$ sufficiently large. If $|\alpha|+|\beta|\le 2$, we have 
\begin{align*}
	1\lesssim \delta\,\tilde{b}^{1/2} + C_{0,\delta}\<v\>^{-N_0}\<y\>^{-N_0}.
\end{align*} 
When $|\beta|> 2$, noticing $\psi_{|\alpha|+|\beta|-2-\frac{1}{2N}}=\psi_{|\alpha|-\frac{1}{2N}}\psi_{|\beta|-2}$, we use 
\begin{align*}
	\psi_{|\alpha|+|\beta|-2-\frac{1}{2N}}\lesssim\delta\, \tilde{b}^{1/2}\psi_{|\alpha|+|\beta|-2} +  C_{0,\delta}\psi_{|\beta|-2}\<v\>^{-N_0}\<y\>^{-N_0}.
\end{align*}
Then by using Lemma \ref{inverse_bounded_lemma} or more precisely by \cite[Lemma 2.4]{Deng2020a}, we have 
\begin{align}\label{101}
	&\quad\,\|\psi_{|\alpha|+|\beta|-2-\frac{1}{2N}}w^{l-|\alpha|-|\beta|}\partial^\alpha_\beta f\|^2_{L^2_{v,x}}\lesssim \delta^2\|\psi_{|\alpha|+|\beta|-2}(\tilde{b}^{1/2})^ww^{l-|\alpha|-|\beta|}\partial^\alpha_\beta f\|^2_{L^2_{v,x}} + C^2_{0,\delta} \|\psi_{|\beta|-2}\partial_\beta f\|^2_{L^2_{v,x}}.
\end{align}
In a similar way, by noticing 
\begin{align*}
	1\lesssim \frac{\delta}{C_{0,\delta}}\tilde{a}^{1/2} + C_\delta \<v\>^{-N_0}\<\eta\>^{-N_0},
\end{align*}where $C_{0,\delta}$ comes from \eqref{100a}, we have 
\begin{align*}
	\|\psi_{|\beta|-2}\partial_\beta f\|^2_{L^2_{v,x}}\lesssim \frac{\delta^2}{C^2_{0,\delta}}\|\psi_{|\beta|-2}(\tilde{a}^{1/2})^w\partial_\beta f\|^2_{L^2_{v,x}}+C_\delta\|f\|^2_{L^2_{v,x}}\lesssim \frac{\delta^2}{C^2_{0,\delta}}\D_{K,l}+C_\delta\E_{0,0}. \end{align*}
Plugging this into \eqref{101}, we have 
\begin{align}\label{101a}
	\|\psi_{|\alpha|+|\beta|-2-\frac{1}{2N}}w^{l-|\alpha|-|\beta|}\partial^\alpha_\beta f\|^2_{L^2_{v,x}}\lesssim \delta^2\|\psi_{|\alpha|+|\beta|-2}(\tilde{b}^{1/2})^ww^{l-|\alpha|-|\beta|}\partial^\alpha_\beta f\|^2_{L^2_{v,x}}+\delta^2\D_{K,l}+C_\delta\E_{0,0}.
\end{align}
Now it suffices to eliminate the first term on the right hand. 

Step 3. 
By the calculation in \cite[Theorem 3.3]{Deng2020b}, we have 
\begin{equation*}
	\theta\in S(1),\ \text{ and }\ \tilde{b}(v,y) = \{\theta,v\cdot y\} + R,
\end{equation*}where $R\in S(\tilde{a})$. Thus, $\theta^w(v,y,D_v,D_y)$ is a bounded operator on $L^2_{v,y}$. Noticing $\tilde{b}$ is a symbol only on $v$, $y$ and $\theta$ is a symbol on $y,v,\eta$ and the second derivative of $v\cdot y$ with respect to $v$ is zero, we have by \eqref{compostion} that 
\begin{align}\label{102}\notag
	\|(\tilde{b}^{1/2})^w(v,x,D_v,D_x)g\|^2_{L^2_{v,x}} &=\big(b(v,y)\widehat{g},\widehat{g}\big)_{L^2_{v,y}} \\\notag
	&= \Re\big(\{\theta,v\cdot y\}^w(v,y,D_v,D_y)\widehat{g},\widehat{g}\big)_{L^2_{v,y}} + \Re(R^w\widehat{g},\widehat{g})_{L^2_{v,y}}\\\notag
	&\lesssim\Re\big(\mathbf{i}v\cdot y\widehat{g},\theta^w\widehat{g}\big)_{L^2_{v,y}}+\|(\tilde{a}^{1/2})^wg\|^2_{L^2_{v,x}}\\
	&\lesssim\Re\big(v\cdot \nabla_x{g},(\theta^w\widehat{g})^\vee\big)_{L^2_{v,x}}+\|(\tilde{a}^{1/2})^wg\|^2_{L^2_{v,x}},
\end{align}
for any $g$ in a suitable smooth space. Here and after, we write $(\tilde{b}^{1/2})^w=(\tilde{b}^{1/2})^w(v,x,D_v,D_x)$ and $\theta^w=\theta^w(v,y,D_v,D_y)$.
Now we let $g = \psi_{|\alpha|+|\beta|-2}w^{l-|\alpha|-|\beta|}\partial^\alpha_\beta f_\pm e^{\frac{\pm\phi}{2}}$ in \eqref{102}, then 
\begin{align}\label{104}
	&\notag\quad\,\|(\tilde{b}^{1/2})^w\psi_{|\alpha|+|\beta|-2}w^{l-|\alpha|-|\beta|}\partial^\alpha_\beta f_\pm e^{\frac{\pm\phi}{2}}\|_{L^2_{v,x}}\\
	&\lesssim\big(v\cdot \nabla_x{\psi_{|\alpha|+|\beta|-2}w^{l-|\alpha|-|\beta|}\partial^\alpha_\beta f}e^{\frac{\pm\phi}{2}},(\theta^w\psi_{|\alpha|+|\beta|-2}w^{l-|\alpha|-|\beta|}{(\partial^\alpha_\beta f_\pm e^{\frac{\pm\phi}{2}})^\wedge})^\vee\big)_{L^2_{v,x}}+\D_{K,l}.
\end{align}
We denote the first term on the right hand by $M$ for brevity. Then, by equation \eqref{7}, 
\begin{align*}
	&\quad\,v\cdot\nabla_x(\partial^\alpha_\beta f_\pm e^{\frac{\pm\phi}{2}})\\
	&= v_i\partial^{\alpha+e_i}_\beta f_\pm e^{\frac{\pm\phi}{2}} \pm \frac{1}{2}v_i\partial^{e_i}\phi e^{\frac{\pm\phi}{2}}\partial^\alpha_\beta f_\pm\\
	&= \partial_\beta\big(v_i\partial^{\alpha+e_i}f_\pm e^{\frac{\pm\phi}{2}}\big)-\sum_{0\neq \beta_1\le \beta}\partial_{\beta_1}v_i\partial^{\alpha+e_i}_{\beta-\beta_1}f_\pm e^{\frac{\pm\phi}{2}}\pm \frac{1}{2}v_i\partial^{e_i}\phi e^{\frac{\pm\phi}{2}}\partial^\alpha_\beta f_\pm\\
	&=-\partial_t\partial^\alpha_\beta f_\pm e^{\frac{\pm\phi}{2}} \pm \sum_{\alpha_1\le\alpha}C^{\alpha_1}_\alpha\partial^{e_i+\alpha_1}\phi\partial^{\alpha-\alpha_1}_{\beta+e_i}f_\pm e^{\frac{\pm\phi}{2}} \mp \frac{1}{2}\sum_{\alpha_1\le\alpha}\sum_{\beta_1\le\beta}\partial^{e_i+\alpha_1}\phi\partial_{\beta_1}v_i\partial^{\alpha-\alpha_1}_{\beta-\beta_1}f_\pm e^{\frac{\pm\phi}{2}}\\
	&\quad\mp \partial^{e_i+\alpha}\phi\partial_\beta(v_i\mu^{1/2})e^{\frac{\pm\phi}{2}}+\partial^\alpha_\beta L_\pm fe^{\frac{\pm\phi}{2}}+\partial^\alpha_\beta\Gamma_\pm(f,f)e^{\frac{\pm\phi}{2}}  -\sum_{0\neq \beta_1\le \beta}\partial_{\beta_1}v_i\partial^{\alpha+e_i}_{\beta-\beta_1}f_\pm e^{\frac{\pm\phi}{2}}\pm \frac{1}{2}v_i\partial^{e_i}\phi e^{\frac{\pm\phi}{2}}\partial^\alpha_\beta f_\pm
\end{align*}
Thus, 
\begin{align*}
 M &=
 \Re\big(-\psi_{2|\alpha|+2|\beta|-4}w^{l-|\alpha|-|\beta|}\partial_t\partial^\alpha_\beta f_\pm e^{\frac{\pm\phi}{2}},(\theta^ww^{l-|\alpha|-|\beta|}{(\partial^\alpha_\beta f_\pm e^{\frac{\pm\phi}{2}})^\wedge})^\vee\big)_{L^2_{v,x}} \\
&\quad \pm \Re\big(\psi_{2|\alpha|+2|\beta|-4}w^{l-|\alpha|-|\beta|}\sum_{\alpha_1\le\alpha}C^{\alpha_1}_\alpha\partial^{e_i+\alpha_1}\phi\partial^{\alpha-\alpha_1}_{\beta+e_i}f_\pm e^{\frac{\pm\phi}{2}},(\theta^ww^{l-|\alpha|-|\beta|}{(\partial^\alpha_\beta f_\pm e^{\frac{\pm\phi}{2}})^\wedge})^\vee\big)_{L^2_{v,x}} \\
&\quad\mp \Re\big(\psi_{2|\alpha|+2|\beta|-4}w^{l-|\alpha|-|\beta|}\frac{1}{2}\sum_{\alpha_1\le\alpha}\sum_{\beta_1\le\beta}\partial^{e_i+\alpha_1}\phi\partial_{\beta_1}v_i\partial^{\alpha-\alpha_1}_{\beta-\beta_1}f_\pm e^{\frac{\pm\phi}{2}},(\theta^ww^{l-|\alpha|-|\beta|}{(\partial^\alpha_\beta f_\pm e^{\frac{\pm\phi}{2}})^\wedge})^\vee\big)_{L^2_{v,x}} \\
 &\quad\mp \Re\big(\psi_{2|\alpha|+2|\beta|-4}w^{l-|\alpha|-|\beta|}\partial^{e_i+\alpha}\phi\partial_\beta(v_i\mu^{1/2})e^{\frac{\pm\phi}{2}},(\theta^ww^{l-|\alpha|-|\beta|}{(\partial^\alpha_\beta f_\pm e^{\frac{\pm\phi}{2}})^\wedge})^\vee\big)_{L^2_{v,x}}
 \\&\quad+\Re\big(\psi_{2|\alpha|+2|\beta|-4}w^{l-|\alpha|-|\beta|}\partial^\alpha_\beta L_\pm f e^{\frac{\pm\phi}{2}}, (\theta^ww^{l-|\alpha|-|\beta|}{(\partial^\alpha_\beta f_\pm e^{\frac{\pm\phi}{2}})^\wedge})^\vee\big)_{L^2_{v,x}}  \\ &\quad+\Re\big(\psi_{2|\alpha|+2|\beta|-4}w^{l-|\alpha|-|\beta|}\partial^\alpha_\beta\Gamma_\pm(f,f)e^{\frac{\pm\phi}{2}},(\theta^ww^{l-|\alpha|-|\beta|}{(\partial^\alpha_\beta f_\pm e^{\frac{\pm\phi}{2}})^\wedge})^\vee\big)_{L^2_{v,x}}  \\
 &\quad-\Re\big(\psi_{2|\alpha|+2|\beta|-4}w^{l-|\alpha|-|\beta|}\sum_{0\neq \beta_1\le \beta}\partial_{\beta_1}v_i\partial^{\alpha+e_i}_{\beta-\beta_1}f_\pm e^{\frac{\pm\phi}{2}},(\theta^ww^{l-|\alpha|-|\beta|}{(\partial^\alpha_\beta f_\pm e^{\frac{\pm\phi}{2}})^\wedge})^\vee\big)_{L^2_{v,x}}\\
 &\quad\pm \Re\big(\psi_{2|\alpha|+2|\beta|-4}w^{l-|\alpha|-|\beta|}\frac{1}{2}v_i\partial^{e_i}\phi e^{\frac{\pm\phi}{2}}\partial^\alpha_\beta f_\pm,(\theta^ww^{l-|\alpha|-|\beta|}{(\partial^\alpha_\beta f_\pm e^{\frac{\pm\phi}{2}})^\wedge})^\vee\big)_{L^2_{v,x}}.
\end{align*}
Denote these terms by $M_1$ to $M_8$. Notice that there's coefficient $\delta$ in \eqref{101a}, we only need to have a upper bound for these terms. 
For $M_1$, noticing that $\theta^w$ is self-adjoint, 
\begin{align*}
	&\quad\,\Re\big(-\psi_{2|\alpha|+2|\beta|-4}w^{l-|\alpha|-|\beta|}\partial_t\partial^\alpha_\beta f_\pm e^{\frac{\pm\phi}{2}},(\theta^ww^{l-|\alpha|-|\beta|}{(\partial^\alpha_\beta f_\pm e^{\frac{\pm\phi}{2}})^\wedge})^\vee\big)_{L^2_{v,x}}\\
	&\le \frac{1}{2}\partial_t\big(-\psi_{2|\alpha|+2|\beta|-4}w^{l-|\alpha|-|\beta|}\partial^\alpha_\beta f_\pm e^{\frac{\pm\phi}{2}},(\theta^ww^{l-|\alpha|-|\beta|}{(\partial^\alpha_\beta f_\pm e^{\frac{\pm\phi}{2}})^\wedge})^\vee\big)_{L^2_{v,x}}\\
	&\qquad+C\big|\big(-\psi_{2|\alpha|+2|\beta|-4-\frac{1}{N}}w^{l-|\alpha|-|\beta|}\partial^\alpha_\beta f_\pm e^{\frac{\pm\phi}{2}},(\theta^ww^{l-|\alpha|-|\beta|}{(\partial^\alpha_\beta f_\pm e^{\frac{\pm\phi}{2}})^\wedge})^\vee\big)_{L^2_{v,x}}\big|\\
	&\qquad+C\big|\big(\partial_t\phi\psi_{2|\alpha|+2|\beta|-4}w^{l-|\alpha|-|\beta|}\partial^\alpha_\beta f_\pm e^{\frac{\pm\phi}{2}},(\theta^ww^{l-|\alpha|-|\beta|}{(\partial^\alpha_\beta f_\pm e^{\frac{\pm\phi}{2}})^\wedge})^\vee\big)_{L^2_{v,x}}\big|.
\end{align*}
We denote the second and third term on the right hand side by $M_{1,1}$ and $M_{1,2}$. Then noticing $\theta^w$ is a bounded operator on $L^2_{v,y}$ and using the trick from \eqref{100a} to \eqref{101a}, we have 
\begin{align*}
	M_{1,1}\lesssim \delta\|\psi_{|\alpha|+|\beta|-2}(\tilde{b}^{1/2})^ww^{l-|\alpha|-|\beta|}\partial^\alpha_\beta f\|_{L^2_{v,x}}^2+\delta\D_{K,l}+C_\delta\E_{0,0}. 
\end{align*}
The boundedness of $\theta^w$ will be frequently used in the following without further mentioned. The term $M_{1,2}$ is similar to the case $I_1$, i.e.
\begin{align*}
	M_{1,2} \lesssim \|\partial_t\phi\|_{L^\infty_x}\|\psi_{|\alpha|+|\beta|-2}w^{l-|\alpha|-|\beta|}\partial^\alpha_\beta f\|_{L^2_{v,x}}^2\lesssim \|\partial_t\phi\|_{L^\infty_x}\E_{K,l}(t). 
\end{align*}
For $M_2$, when $\alpha_1=0$, noticing $\theta^w$ is self-adjoint, we use integration by parts over $v$ to obtain 
\begin{align*}
|M_2| &= \big|\big(\psi_{2|\alpha|+2|\beta|-4}w^{l-|\alpha|-|\beta|}\partial^{e_i}\phi\partial^{\alpha}_{\beta+e_i}f_\pm e^{\frac{\pm\phi}{2}},(\theta^ww^{l-|\alpha|-|\beta|}{(\partial^\alpha_\beta f_\pm e^{\frac{\pm\phi}{2}})^\wedge})^\vee\big)_{L^2_{v,x}}\big|\\
&\lesssim \big|\big(\psi_{2|\alpha|+2|\beta|-4}\partial_{e_i}(w^{l-|\alpha|-|\beta|})\partial^{e_i}\phi\partial^{\alpha}_{\beta}f_\pm e^{\frac{\pm\phi}{2}},(\theta^ww^{l-|\alpha|-|\beta|}{(\partial^\alpha_\beta f_\pm e^{\frac{\pm\phi}{2}})^\wedge})^\vee\big)_{L^2_{v,x}}\big|\\
&\quad+\big|\big(\psi_{2|\alpha|+2|\beta|-4}w^{l-|\alpha|-|\beta|}\partial^{e_i}\phi\partial^{\alpha}_{\beta}f_\pm e^{\frac{\pm\phi}{2}},(\underbrace{[\partial_{e_i},\theta^w]}_{\in S(1)}w^{l-|\alpha|-|\beta|}{(\partial^\alpha_\beta f_\pm e^{\frac{\pm\phi}{2}})^\wedge})^\vee\big)_{L^2_{v,x}}\big|\\
&\quad+\big|\big(\psi_{2|\alpha|+2|\beta|-4}w^{l-|\alpha|-|\beta|}\partial^{e_i}\phi\partial^{\alpha}_{\beta}f_\pm e^{\frac{\pm\phi}{2}},(\theta^w\partial_{e_i}(w^{l-|\alpha|-|\beta|}){(\partial^\alpha_\beta f_\pm e^{\frac{\pm\phi}{2}})^\wedge})^\vee\big)_{L^2_{v,x}}\big|\\
&\lesssim \|\partial^{e_i}\phi\|_{H^2_x}\|\psi_{|\alpha|+|\beta|-2}w^{l-|\alpha|-|\beta|}\partial^{\alpha}_{\beta}f_\pm\|_{L^2_{v,x}}^2\\
&\lesssim \delta_0\E_{K,l}(t),
\end{align*}by \eqref{13}. 
When $\alpha_1\neq 0$, then $\alpha\neq 0$, the total number of derivatives on the first $f_\pm$ is less or equal to $K$ and there's at least one derivative on the second $f_\pm$ with respect to $x$. Thus, 
\begin{align*}
	|M_2|\lesssim \E^{1/2}_{K,l}\D_{K,l}, 
\end{align*}where \eqref{13} is applied. 
For the term $M_3$ with $\alpha_1=\beta_1=0$, a nice observation is that it's the same as $M_8$ except the sign and hence, they are eliminated. When $\alpha_1+\beta_1\neq 0$, the derivative order for the first $f_\pm$ is less or equal to $K-1$ and hence, the term $\partial_{\beta_1}v_i$ can be controlled as $w^{l-|\alpha|-|\beta|}\partial_{\beta_1}v_i\lesssim w^{l-|\alpha-\alpha_1|-|\beta-\beta_1|}$. Then $M_3$ is bounded above by 
\begin{align*}
	\big|\big(\psi_{2|\alpha|+2|\beta|-4}&w^{l-|\alpha|-|\beta|}\frac{1}{2}\sum_{\substack{\alpha_1\le\alpha\\\beta_1\le\beta\\|\alpha_1|+|\beta_1|\neq 0}}\partial^{e_i+\alpha_1}\phi\partial_{\beta_1}v_i\partial^{\alpha-\alpha_1}_{\beta-\beta_1}f_\pm e^{\frac{\pm\phi}{2}},(\theta^ww^{l-|\alpha|-|\beta|}{(\partial^\alpha_\beta f_\pm e^{\frac{\pm\phi}{2}})^\wedge})^\vee\big)_{L^2_{v,x}}\big|\\
	&\lesssim \E^{1/2}_{K,l}\D_{K,l}.
\end{align*}
For $M_4$, there's exponential decay in $v$ and hence 
	$|M_4|\lesssim \E_{K,l}. $
For $M_5$, recalling that we only need upper bound, using Lemma \ref{lemmat}, we have $
	|M_5|\lesssim\E_{K,l}.$  For $M_6$, we use Lemma \ref{lemmag} to obtain $|M_6|\lesssim\E^{1/2}_{K,l}\D_{K,l}.$ For $M_7$, since $\beta_1\neq 0$, $|\partial_{\beta_1}v_i|\lesssim 1$ and the total number of derivatives on the first $f_\pm$ is less or equal to $K$. This yields that $|M_7|\lesssim \D_{K,l}$. Combining the above estimate with \eqref{104}, we have 
	\begin{align}\label{112}
		&\notag\quad\,\|(\tilde{b}^{1/2})^w\psi_{|\alpha|+|\beta|-2}w^{l-|\alpha|-|\beta|}\partial^\alpha_\beta f_\pm e^{\frac{\pm\phi}{2}}\|_{L^2_{v,x}}\\
		&\lesssim \delta^2\partial_t\big(-\psi_{2|\alpha|+2|\beta|-4}w^{l-|\alpha|-|\beta|}\partial^\alpha_\beta f_\pm e^{\frac{\pm\phi}{2}},(\theta^ww^{l-|\alpha|-|\beta|}{(\partial^\alpha_\beta f_\pm e^{\frac{\pm\phi}{2}})^\wedge})^\vee\big)_{L^2_{v,x}}\\
		&\notag\qquad+\E^{1/2}_{K,l}\D_{K,l}+\E_{0,0}+\|\partial_t\phi\|_{L^\infty_x}\E_{K,l}(t)+\D_{K,l}+\E_{K,l}.
	\end{align}
Substituting \eqref{112} into \eqref{101a}, then plugging into \eqref{111}, we have that for $0<\delta<1$, 
\begin{align*}
	\partial_t\E_{K,l}(t)+\lambda D_{K,l}(t)&\lesssim\delta^2\sum_{|\alpha|+|\beta|\le K}\partial_t\big(-\psi_{2|\alpha|+2|\beta|-4}w^{l-|\alpha|-|\beta|}\partial^\alpha_\beta f_\pm e^{\frac{\pm\phi}{2}},(\theta^ww^{l-|\alpha|-|\beta|}{(\partial^\alpha_\beta f_\pm e^{\frac{\pm\phi}{2}})^\wedge})^\vee\big)_{L^2_{v,x}}\\&\qquad+ \|\partial_t\phi\|_{L^\infty_x}\E_{K,l}(t) + \delta\big(\E^{1/2}_{K,l}(t)\D_{K,l}(t)+\D_{K,l}(t)+\E_{K,l}(t)\big)+C_\delta\E_{0,0}(t),
\end{align*}where we will use $\|\partial_t\phi\|_{L^\infty_x}\lesssim \E^{1/2}_{K,l}\lesssim \delta^{1/2}_0$ by \eqref{34} and \eqref{priori1}.
By choosing $\delta>0$ sufficiently small, using the $a$ $priori$ assumption \eqref{priori1} and
noticing that $\E_{K,l}\lesssim \D_{K,l}+\|\P f\|^2_{L^2_x}\lesssim\D_{K,l}+\E_{0,0}$, we have 
\begin{align*}
	&\quad\,\partial_t\E_{K,l}(t)+\lambda \E_{K,l}(t)\\&\lesssim\delta^2\sum_{|\alpha|+|\beta|\le K}\partial_t\big(-\psi_{2|\alpha|+2|\beta|-4}w^{l-|\alpha|-|\beta|}\partial^\alpha_\beta f_\pm e^{\frac{\pm\phi}{2}},(\theta^ww^{l-|\alpha|-|\beta|}{(\partial^\alpha_\beta f_\pm e^{\frac{\pm\phi}{2}})^\wedge})^\vee\big)_{L^2_{v,x}}+\E_{K,l}(t).
\end{align*}
By solving this ODE, choosing $\delta>0$ sufficiently small and noticing 
\begin{align*}
	\sum_{|\alpha|+|\beta|\le K}\big|\big(-\psi_{2|\alpha|+2|\beta|-4}w^{l-|\alpha|-|\beta|}\partial^\alpha_\beta f_\pm e^{\frac{\pm\phi}{2}},(\theta^ww^{l-|\alpha|-|\beta|}{(\partial^\alpha_\beta f_\pm e^{\frac{\pm\phi}{2}})^\wedge})^\vee\big)_{L^2_{v,x}}\big|
	\lesssim \E_{K,l}(t),
\end{align*} we have 
\begin{align*}
	\E_{K,l}(t) &\lesssim \E_{K,l}(0)+t (\E_{K,l}(t)+\epsilon_1^2),\\
	\E_{K,l}(t) &\lesssim \epsilon^2_1,
\end{align*}by choose $t<<1$, since $\E_{K,l}(0)\le \E_{2,l}(0)$.

\end{proof}

\begin{proof}
	[Proof of Theorem \ref{main2}]
		It follows immediate from the $a$ $priori$ estimate \eqref{priori1} that $\sup_{0\le t\le t_0}\E_{K,l}\lesssim \epsilon^2_1$ holds true for some $t_0>0$ and the details are the same as obtaining the global solution as in Theorem \ref{main1}; see also \cite{Guo2012,Gressman2011} and \cite{Deng2020b}. The Theorem \ref{main1} has already proved that $\E_{i+1,2}(t)\lesssim \epsilon_0^2C_\tau$ for any $t\ge \tau>0$. Notice that the constant in Lemma \ref{lem51} is independent of time $t$ and hence, we can apply Lemma \ref{lem51} to any time interval with length less than $t_0$ to obtain
		\begin{align}\label{106}
			\sup_{t>0}\E_{K,l}(t)\lesssim \epsilon_1^2. 
		\end{align}
	Recalling \eqref{Defe} and the choice \eqref{93} of $\psi$, we have that for any $\tau,T\in(0,\infty]$, any $l,K\ge 0$,
	\begin{align*}
		\sup_{\tau\le t\le T}\sum_{|\alpha|+|\beta|\le K}\|w^{l-|\alpha|-|\beta|}\partial^\alpha_\beta f\|^2_{L^2_{v,x}}+\sup_{\tau\le t\le T}\sum_{|\alpha|\le K}\|\partial^\alpha\nabla_x\phi\|_{L^2_x}^2\le  C_{\tau}<\infty.
	\end{align*}	
 This proves \eqref{19a}.

If additionally $
\sup_{l_0\ge 2}\E_{2,l_0}(0)<\infty$ is sufficiently small. Then for $l\ge K\ge 3$, by \eqref{19a}, we have 
\begin{align*}
	\sup_{\tau\le t\le T}\sum_{|\alpha|+|\beta|\le K}\|w^{l-|\alpha|-|\beta|}\partial^\alpha_\beta f\|^2_{L^2_{v,x}}\le C_{\tau}.
\end{align*}
For the regularity on $t$, the technique above is not applicable and we only make a rough estimate. For any $t>0$, applying $w^l\partial^k_t\partial^\alpha_\beta$ with $k\ge 0$, $|\alpha|+|\beta|\le K$ to equation \eqref{7} and taking $L^2_{v,x}$ norms, we have   
	\begin{align}\label{105}
		\|w^l\partial^{k+1}_t\partial^\alpha_\beta f_\pm\|^2_{L^2_{v,x}}
		&\notag\lesssim \|w^lv\cdot\nabla_x\partial^k_t\partial^\alpha_\beta f_\pm\|^2_{L^2_{v,x}}+\|w^l\sum_{k_1\le k}\partial^{\alpha}_\beta\big(\partial^{k_1}_t\nabla_x\phi\cdot v\partial^{k-k_1}_t f_\pm\big)\|^2_{L^2_{v,x}}
		\\
		&\qquad+\|w^l\sum_{k_1\le k}\partial^\alpha\big(\partial^{k_1}_t\nabla_x\phi\cdot\nabla_v\partial^{k-k_1}_t\partial_\beta f_\pm\big)\|_{L^2_{v,x}}^2+\|w^l\partial^k_t\partial^\alpha\nabla_x\phi\cdot \partial_\beta(v\mu^{1/2})\|^2_{L^2_{v,x}}\\
		&\qquad+\|w^l\partial^\alpha_\beta L_\pm \partial^k_tf_\pm\|^2_{L^2_{v,x}} + \|w^l\sum_{k_1\le k}\partial^\alpha_\beta\Gamma_\pm(\partial^{k_1}_tf,\partial^{k-k_1}_tf)\|_{L^2_{v,x}}^2. \notag
	\end{align}
Denoting $\E_{K,l,k}=\sum_{|\alpha|+|\beta|\le K,k_1\le k}\|w^l\partial^\alpha_\beta\partial^{k_1}_tf\|_{L^2_{v,x}}$, we estimate the right-hand terms one by one. The first term on the right hand is bounded above by $\E_{K+1,l+1,k}$.
For terms involving both $\phi$ and $f_\pm$, we use \eqref{13} to generate one more $x$ derivative on $\phi$. Applying the trick in \eqref{78}, the second term is bounded above by 
\begin{align*}
	\sum_{|\alpha|+|\beta|\le K+1,\,k_1\le k}\|\partial^{k_1}_t\partial^\alpha_\beta\nabla_x\phi\|^2_{L^2_x}\sum_{|\alpha|+|\beta|\le K+1,\,k_1\le k}\|w^{l+1}\partial^{k_1}_t\partial^\alpha_\beta f_\pm\|^2_{L^2_{v,x}}\lesssim \E_{K+1,l+1,k}^2.
\end{align*}
Similarly, the third term is bounded above by $\E_{K+1,l+1,k}^2.$ For the fourth term, when $k=0$, it's bounded above by $\E_{K,l,0}$. When $k\ge 1$, by using \eqref{34a}, it's bounded above by $\E_{K,l,k-1}$. For the fifth term, noticing $L_\pm\in S(\tilde{a})\subset S(\<v\>^{\gamma+2s}\<\eta\>^{2s})$ and $s\in(0,1)$, we have 
\begin{align*}
\|w^l\partial^\alpha_\beta L_\pm \partial^k_tf_\pm\|^2_{L^2_{v,x}}\lesssim \|w^{l+\gamma+2s}\<D_v\>^2\<(D_x,D_v)\>^K \partial^k_tf_\pm\|^2_{L^2_{v,x}}\lesssim \E_{K+2,l+\gamma+2s,k}.
\end{align*}
For the last term, using \eqref{12a}, it's bounded above by 
\begin{align*}
	\sum_{|\alpha|+|\beta|\le K,\,k_1\le k}\|w^{l+\frac{\gamma+2s}{2}}\partial^\alpha_\beta\partial^{k_1}_tf\|^2_{L^2_{v,x}}	\sum_{|\alpha|+|\beta|\le K+2,\,k_1\le k}\|w^{l+\gamma+2s}\partial^\alpha_\beta\partial^{k_1}_tf\|^2_{L^2_{v,x}}\lesssim \E_{K+2,l+\gamma+2s,k}^2. 
\end{align*} 
Combining the above estimate and taking summation $|\alpha|+|\beta|\le K$, $k\le k_0$ for any $k_0\ge 0$, we have 
\begin{align*}
	\E_{K,l,k_0+1}(t)\lesssim \E_{K,l,0} +\E_{K,l,k_0-1}(t)+ \E_{K+2,l+1+\gamma+2s,k_0}^2(t).
\end{align*} 
Hence, noticing \eqref{106}, for any $T>\tau>0$, we have
\begin{align*}
	\sup_{\tau\le t\le T}\E_{K,l,k_0}(t)\le C_{\tau,k_0}.
\end{align*}
The same standard argument for obtaining the local solution gives the result of Theorem \ref{main2} and the details are omitted for brevity; see also \cite{Guo2012,Gressman2011} and \cite{Deng2020b}. Consequently, by Sobolev embedding, $f\in C^\infty(\R^+_t;C^\infty(\R^3_x;\mathscr{S}(\R^3_v)))$.

\end{proof}

\section{Appendix}

\paragraph{Pseudo-differential calculus}

We recall some notation and theorem of pseudo differential calculus. For details, one may refer to Chapter 2 in the book \cite{Lerner2010} for details. Set $\Gamma=|dv|^2+|d\eta|^2$, but also note that the following are also valid for general admissible metric.
Let $M$ be an $\Gamma$-admissible weight function. That is, $M:\R^{2d}\to (0,+\infty)$ satisfies the following conditions:\\
(a). (slowly varying) there exists $\delta>0$ such that for any $X,Y\in\R^{2d}$, $|X-Y|\le \delta$ implies
\begin{align*}
	M(X)\approx M(Y);
\end{align*}
(b) (temperance) there exists $C>0$, $N\in\R$, such that for $X,Y\in \R^{2d}$,
\begin{align*}
	\frac{M(X)}{M(Y)}\le C\<X-Y\>^N.
\end{align*}
A direct result is that if $M_1,M_2$ are two $\Gamma$-admissible weight, then so is $M_1+M_2$ and $M_1M_2$. Consider symbols $a(v,\eta,\xi)$ as a function of $(v,\eta)$ with parameters $\xi$. We say that
$a\in S(M)=S(M,\Gamma)$ uniformly in $\xi$, if for $\alpha,\beta\in \N^d$, $v,\eta\in\Rd$,
\begin{align*}
	|\partial^\alpha_v\partial^\beta_\eta a(v,\eta,\xi)|\le C_{\alpha,\beta}M,
\end{align*}with $C_{\alpha,\beta}$ a constant depending only on $\alpha$ and $\beta$, but independent of $\xi$. The space $S(M,\Gamma)$ endowed with the seminorms
\begin{align*}
	\|a\|_{k;S(M,\Gamma)} = \max_{0\le|\alpha|+|\beta|\le k}\sup_{(v,\eta)\in\R^{2d}}
	|M(v,\eta)^{-1}\partial^\alpha_v\partial^\beta_\eta a(v,\eta,\xi)|,
\end{align*}becomes a Fr\'{e}chet space.
Sometimes we write $\partial_\eta a\in S(M,\Gamma)$ to mean that $\partial_{\eta_j} a\in S(M,\Gamma)$ $(1\le j\le d)$ equipped with the same seminorms.
We formally define the pseudo-differential operator by
\begin{align*}
	(op_ta)u(x)=\int_\Rd\int_\Rd e^{2\pi i (x-y)\cdot\xi}a((1-t)x+ty,\xi)u(y)\,dyd\xi,
\end{align*}for $t\in\R$, $f\in\S$.
In particular, denote $a(v,D_v)=op_0a$ to be the standard pseudo-differential operator and
$a^w(v,D_v)=op_{1/2}a$ to be the Weyl quantization of symbol $a$. We write $A\in Op(M,\Gamma)$ to represent that $A$ is a Weyl quantization with symbol belongs to class $S(M,\Gamma)$. One important property for Weyl quantization of a real-valued symbol is the self-adjoint on $L^2$ with domain $\S$. 

For composition of pseudodifferential operator we have $a^wb^w= (a\#b)^w$ with 
\begin{align}\label{compostion}
	a\#b = ab + \frac{1}{4\pi i}\{a,b\} + \sum_{2\le k\le \nu}2^{-k}\sum_{|\alpha|+|\beta|=k}\frac{(-1)^{|\beta|}}{\alpha!\beta!}D^\alpha_\eta\partial^\beta_xaD^{\beta}_\eta\partial^\alpha_xb+r_\nu(a,b),
\end{align}where $X=(v,\eta)$,
\begin{align*}
	r_\nu(a,b)(X) & = R_\nu(a(X)\otimes b(Y))|_{X=Y},\\
	R_\nu &= \int^1_0\frac{(1-\theta)^{\nu-1}}{(\nu-1)!}\exp\Big(\frac{\theta}{4\pi i}\<\sigma\partial_X,\partial_Y\>\Big)\,d\theta\Big(\frac{1}{4\pi i}\<\sigma\partial_X,\partial_Y\>\Big)^\nu.
\end{align*}

Let $a_1(v,\eta)\in S(M_1,\Gamma),a_2(v,\eta)\in S(M_2,\Gamma)$, then $a_1^wa_2^w=(a_1\#a_2)^w$, $a_1\#a_2\in S(M_1M_2,\Gamma)$ with
\begin{align*}
	a_1\#a_2(v,\eta)&=a_1(v,\eta)a_2(v,\eta)
	+\int^1_0(\partial_{\eta}a_1\#_\theta \partial_{v} a_2-\partial_{v} a_1\#_\theta \partial_{\eta} a_2)\,d\theta,\\
	g\#_\theta h(Y):&=\frac{2^{2d}}{\theta^{-2n}}\int_\Rd\int_\Rd e^{-\frac{4\pi i}{\theta}\sigma(X-Y_1)\cdot(X-Y_2)}(4\pi i)^{-1}\<\sigma\partial_{Y_1}, \partial_{Y_2}\>g(Y_1) h(Y_2)\,dY_1dY_2,
\end{align*}with $Y=(v,\eta)$, $\sigma=\begin{pmatrix}
	0&I\\-I&0
\end{pmatrix}$.
For any non-negative integer $k$, there exists $l,C$ independent of $\theta\in[0,1]$ such that
\begin{align}\label{sharp_theta}
	\|g\#_\theta h\|_{k;S(M_1M_2,\Gamma)}\le C\|g\|_{l,S(M_1,\Gamma)}\|h\|_{l,S(M_2,\Gamma)}.
\end{align}
Thus if $\partial_{\eta}a_1,\partial_{\eta}a_2\in S(M'_1,\Gamma)$ and $\partial_{v}a_1,\partial_{v}a_2\in S(M'_2,\Gamma)$, then $[a_1,a_2]\in S(M'_1M'_2,\Gamma)$, where $[\cdot,\cdot]$ is the commutator defined by $[A,B]:=AB-BA$.

We can define a Hilbert space $H(M,\Gamma):=\{u\in\S':\|u\|_{H(M,\Gamma)}<\infty\}$, where
\begin{align}\label{sobolev_space}
	\|u\|_{H(M,\Gamma)}:=\int M(Y)^2\|\varphi^w_Yu\|^2_{L^2}|g_Y|^{1/2}\,dY<\infty,
\end{align}and $(\varphi_Y)_{Y\in\R^{2d}}$ is any uniformly confined family of symbols which is a partition of unity. If $a\in S(M)$ is a isomorphism from $H(M')$ to $H(M'M^{-1})$, then $(a^wu,a^wv)$ is an equivalent Hilbertian structure on $H(M)$. Moreover, the space $\S(\Rd)$ is dense in $H(M)$ and $H(1)=L^2$.

Let $a\in S(M,\Gamma)$, then
$a^w:H(M_1,\Gamma)\to H(M_1/M,\Gamma)$ is linear continuous, in the sense of unique bounded extension from $\S$ to $H(M_1,\Gamma)$.
Also the existence of $b\in S(M^{-1},\Gamma)$ such that $b\#a = a\#b = 1$ is equivalent to the invertibility of $a^w$ as an operator from $H(MM_1,\Gamma)$
onto $H(M_1,\Gamma)$ for some $\Gamma$-admissible weight function $M_1$.

The following Lemmas come from \cite{Deng2020a}.

\begin{Lem}\label{inverse_bounded_lemma}Let $m,c$ be $\Gamma$-admissible weight and $a\in S(m)$.
	Assume $a^w:H(mc)\to H(c)$ is invertible.
	If $b\in S(m)$, then there exists $C>0$, depending only on the seminorms of symbols to $(a^w)^{-1}$ and $b^w$, such that for $f\in H(mc)$,
	\begin{align*}
		\|b(v,D_v)f\|_{H(c)}+\|b^w(v,D_v)f\|_{H(c)}\le C\|a^w(v,D_v)f\|_{H(c)}.
	\end{align*}
	Consequently, if $a^w:H(m_1)\to L^2\in Op(m_1)$, $b^w:H(m_2)\to L^2\in Op(m_2)$ are invertible, then for $f\in\S$, 
	\begin{align*}
		\|b^wa^wf\|_{L^2}\lesssim \|a^wb^wf\|_{L^2},
	\end{align*}where the constant depends only on seminorms of symbols to $a^w,b^w,(a^w)^{-1},(b^w)^{-1}$.
\end{Lem}

\begin{Lem}
	Let $m,c$ be $\Gamma$-admissible weight and $a^{1/2}\in S(m^{1/2})$.
	Assume $(a^{1/2})^w:H(mc)\to H(c)$ is invertible and $L\in S(m)$. Then 
	\begin{align*}
		(Lf,f)_{L^2} = (\underbrace{((a^{1/2})^w)^{-1}L}_{\in S(m^{1/2})}f,(a^{1/2})^wf)_{L^2}\lesssim \|(a^{1/2})^wf\|^2_{L^2}.
	\end{align*}
\end{Lem}

\paragraph{Carleman representation}

For measurable function $F(v,v_*,v',v'_*)$, if any sides of the following equation is well-defined, then
\begin{align}
	&\int_{\R^d}\int_{\mathbb{S}^{d-1}}b(\cos\theta)|v-v_*|^\gamma F(v,v_*,v',v'_*)\,d\sigma dv_*\notag\\
	&\quad=\int_{\R^d_h}\int_{E_{0,h}}\tilde{b}(\alpha,h)\1_{|\alpha|\ge|h|}\frac{|\alpha+h|^{\gamma+1+2s}}{|h|^{d+2s}}F(v,v+\alpha-h,v-h,v+\alpha)\,d\alpha dh,\label{Carleman}
\end{align}where $\tilde{b}(\alpha,h)$ is bounded from below and above by positive constants, and $\tilde{b}(\alpha,h)=\tilde{b}(|\alpha|,|h|)$, $E_{0,h}$ is the hyper-plane orthogonal to $h$ containing the origin.


\paragraph{Acknowledgments:} The author would thank Prof. Tong Yang for the valuable comments on the manuscript. 

\small
\bibliographystyle{plain}
\bibliography{1}

\end{document}